\documentclass[a4paper,11pt,reqno]{amsart}
\usepackage[a4paper,top=3cm,bottom=2cm,left=3cm,right=3cm,marginparwidth=1.75cm]{geometry}
\usepackage[english]{babel}
\usepackage{amsfonts, amsmath, amssymb, amsthm}
\usepackage{thmtools}
\usepackage{bm}
\usepackage{hyperref} 
\usepackage{cleveref}
\usepackage{tikz-cd}
\usetikzlibrary{arrows}

\usepackage[OT2,T1]{fontenc}
\DeclareSymbolFont{cyrletters}{OT2}{wncyr}{m}{n}
\DeclareMathSymbol{\Beh}{\mathalpha}{cyrletters}{"42}
\DeclareMathSymbol{\beh}{\mathalpha}{cyrletters}{"62}
\DeclareMathSymbol{\Sha}{\mathalpha}{cyrletters}{"58}

\newcommand{\N}{\mathbb{N}}
\newcommand{\Z}{\mathbb{Z}}
\newcommand{\Q}{\mathbb{Q}}

\newcommand{\C}{\mathbb{C}}
\newcommand{\F}{\mathbb{F}}
\newcommand{\G}{\mathbb{G}}
\newcommand{\A}{\mathbb{A}}
\newcommand{\Adele}{\mathbf{A}}
\newcommand{\Proj}{\mathbb{P}}

\newcommand{\Frob}{\mathrm{Frob}}
\newcommand{\Hur}{\mathrm{Hur}}
\newcommand{\HurStack}{\mathcal{H}\mathrm{ur}}
\newcommand{\CHur}{\mathrm{CHur}}
\newcommand{\Conf}{\mathrm{Conf}}

\DeclareMathOperator{\Hom}{Hom}
\DeclareMathOperator{\Aut}{Aut}

\DeclareMathOperator{\Gal}{Gal}
\DeclareMathOperator{\Spec}{Spec}
\DeclareMathOperator{\Br}{Br}

\DeclareMathOperator{\HH}{H}

\numberwithin{equation}{section}

\newtheorem{theorem}{Theorem}[section]
\newtheorem{proposition}[theorem]{Proposition}
\newtheorem{corollary}[theorem]{Corollary}
\newtheorem{lemma}[theorem]{Lemma}

\theoremstyle{definition}
\newtheorem{definition}[theorem]{Definition}

\newtheorem{remark}[theorem]{Remark}
\newtheorem{construction}[theorem]{Construction}
\title{The leading constant in Malle's conjecture over function fields}
\author{Tim Santens}
\address{
	University of Cambridge \\ 
	DPMMS \\
	Centre for Mathematical Sciences\\
	Wilberforce Road \\
	Cambridge \\
	CB3 0WB \\ UK}
\email{Tim.Santens@cam.ac.uk}

\begin{document}

\begin{abstract}
    We prove the analogue of Malle's conjecture for the global function field $\F_q(t)$ with $q$ sufficiently large, including a precise formula for the leading constant. The main ingredients are the recent breakthrough of Landesman--Levy on the stable homology of Hurwitz spaces, a novel interpretation of the Frobenius fixed components of Hurwitz spaces in terms of the Brauer group of the stack $BG$ and an interpretation of the number of $\F_q$ points of configuration spaces as a certain Tamagawa volume.
\end{abstract}
\subjclass[2020]
{14G05 (primary), 
11R58,
14D23, 
14F22 
(secondary).}
\maketitle
\tableofcontents
\section{Introduction}
Let $G$ be a finite group. Based on numerical data Malle \cite{Malle2002Galois, Malle2004Galois} made a prediction on the asymptotic distribution of $G$-extensions of
$\Q$ of bounded discriminant. These predictions have been generalized from $\Q$ to general global fields and to more general ways of enumerating $G$-extensions than the discriminant.

For $\Q$ and other number fields, Malle's conjecture is known only for a limited class of groups. However, over the global function field $\F_q(t)$, the seminal work of Ellenberg--Venkatesh--Westerland \cite{Ellenberg2005Counting, Ellenberg2016Homological, ellenberg2013homologicalstabilityhurwitzspaces} showed that Malle's conjecture (among other questions) can be studied using the topology of Hurwitz spaces. At least as long as $q$ is sufficiently large in terms of $G$.

This topological approach has been significantly refined since then. For example Ellenberg--Tran--Westerland \cite{ellenberg2023foxneuwirthfukscellsquantumshuffle} were able to show a weak form of Malle's conjecture for all groups $G$. Very recently, Landesman and Levy \cite{landesman2025cohenlenstramomentsfunctionfields,Landesman2025Homological,landesman2025stablehomologyhurwitzmodules} achieved a major breakthrough in the homological stability of Hurwitz spaces, which allowed them to prove the order of magnitude version of Malle's conjecture over $\F_q(t)$ for all groups $G$.

In another research direction, the author and Loughran \cite{Loughran2025mallesconjecturebrauergroups} recently proposed a prediction of the leading constant in Malle's conjecture for number fields. The version of Malle's conjecture that Landesman and Levy prove only provides the upper and lower bounds. The goal of this paper is to use their powerful topological results to prove Malle's conjecture with a leading constant and to compare the leading constant with the predicted constant in the number field case.

Let us first state a simpler form of the main theorem and introduce some notation. Let $\F_q$ be a finite field of order $q$ coprime to $|G|$. For $d \in \N$ we define $N(G, q^{d})$ as the number of isomorphism classes of geometrically connected $G$-covers $X \to \Proj^1_{\F_q}$ ramified at $d$ geometric points. Let $G(-1)$ denote the anticyclotomic twist of $G$. This is a finite \'etale scheme with a (noncanonical) identification $G(-1)(\bar{\F}_q) \cong G$ such that the (arithmetic) Frobenius of $\F_q$ acts on $g \in G$ by sending it to $g^{q^{-1}}$. Let $\mathcal{C}_G := G(-1)/G$ where $G$ acts on $G(-1)$ via conjugation. Let $b(G, \F_q)$ be the number of connected components of $\mathcal{C}_G^* := \mathcal{C}_G \setminus \{1\}$.

Define the normalised count as $n(G, q^d) := q^{-d} d^{-b(G, \F_q) + 1} N(G, q^d)$. We would like to consider the limit of $n(G, q^d)$ as $d \to \infty$ as the leading constant for $G$. Unfortunately, the function $n(G, q^d)$ does not converge. Instead, it has a certain periodicity in the limit\footnote{This periodicity is the main reason why the author and Loughran restricted ourselves in \cite{Loughran2025mallesconjecturebrauergroups} to the number field case.}. For now we deal with this by averaging over $d$.

\begin{theorem}\label{thm:main_theorem_radical_disc}
    Assume that $q$ is coprime to $|G|$ and sufficiently large in terms of $|G|$. We have
    \[
        \lim_{d \to \infty} |G|^{-2}\sum_{k = 1}^{|G|^2} n(G, q^{d + k}) = \frac{|Z(G)| \cdot |\Br_{\mathrm{un}} (BG)_{\F_q}|}{|\Hom(G, \F_q^{\times})| (b(G, \F_q) - 1)!} \tau(\F_q, G)
    \]
    Here $Z(G)$ is the center of $G$, $\Br_{\mathrm{un}} (BG)_{\F_q}$ is the unramified Brauer group of $BG$ and $\tau(\F_q, G)$ is the normalized Tamagawa measure, explicitly given by the following conditionally convergent Euler product over closed points $P \in \Proj^1_{\F_q}$
    \[
    (1 - q^{-1})^{-b(G, \F_q)}\prod_{P \in \Proj^1_{\F_q}}(1 - q^{-\deg P})^{b(G, \F_q)}(1 + |\mathcal{C}_G(\F_q(P))|q^{-\deg P}).
    \]
    
    Here $\deg P$ is the degree of $P$ and $\F_q(P)$ denotes the residue field of $P$.
\end{theorem}
Without the explicit formula for the leading constant, this is \cite[Thm~10.1.10]{Landesman2025Homological}. The formula for the leading constant answers \cite[Question 11.1.2]{Landesman2025Homological}.

The only case of the following theorem with the leading constant in the literature is the case $G = C_{\ell}^n$ for $\ell \neq p$ a prime by \cite{Wright1989Abelian}, as the discriminant of a $C_{\ell}^n$-extension is determined by the number of ramified points. The case where $G$ is an abelian group is known to experts as a straightforward extension of the number field case in \cite{Maki1985Density, Wood2010Abelian}. 

This averaged constant can be compared with the conjectural constant in the number field case \cite[Conj.~9.1]{Loughran2025mallesconjecturebrauergroups}. One sees that it is exactly analogous, where one interprets $(1 - q^{-1})$ as the value of $(1 - q^{1 - s})\zeta_{\Proj^1_{\F_q}}(s)$ at $s = 1$ and uses that $\Br_{\mathrm{un}} (BG)_{\F_q} = \Br_{\mathrm{un}} (BG)_{\F_q(t)}/\Br \F_q(t)$ and there is no Brauer-Manin obstruction by \cite[Thm.~7.11, Cor.~7.12]{Loughran2025mallesconjecturebrauergroups}.

\subsection{Finer results}
We prove a statement that is stronger than \Cref{thm:main_theorem_radical_disc} in multiple ways. 
\begin{enumerate}
    \item We will also allow $G$ to be a finite \'etale group scheme. The reason is that no new ideas are required, and such group schemes appear when attempting to count extensions that are not geometrically connected, as explained in~\Cref{rem:non_geometrically_connected_covers}.
    \item We want to understand the effect of fixing the local behavior at $\infty$. This will show that the effect of the Brauer-Manin obstruction on equidistribution is exactly as predicted in \cite[\S9.4]{Loughran2025mallesconjecturebrauergroups}.
    \item We want to allow more general heights than just the number of ramified geometric points. Indeed, we will determine the leading constant for all \emph{balanced} height functions \cite[Def.~3.30]{Loughran2025mallesconjecturebrauergroups}.
    \item We want to give a geometric interpretation of the periodicity in the leading constant. In particular, there are examples of groups $G$ where the periodicity cannot be directly explained via the geometry of Hurwitz spaces but instead comes from a reciprocity law. Ellenberg--Venkatesh--Westerland noticed the first example in $G = \mathrm{PSL}_2(\F_7)$ \cite[\S9.3.3]{ellenberg2013homologicalstabilityhurwitzspaces}. We explain this type of periodicity via transcendental Brauer elements whose residues are algebraic but non-trivial.
\end{enumerate}
We leave the statement of the most general version of our main theorem to \Cref{thm:main_theorem_multiheight} as it involves too much notation. Instead, we will only discuss points (3) and (4) now, as they are the most interesting. Let us introduce some additional notation.

Let $f \in \Hom(\mathcal{C}^*_G, \Z_{> 0})$ be a function. Define $a(f) := \min_{c \in \mathcal{C}_G^*} (f(c))^{-1}$ and let $\mathcal{C}_f \subset \mathcal{C}^*_G$ be the scheme of conjugacy classes on which $f$ is minimal. Let $b(\F_q, f)$ be the number of connected components of $\mathcal{C}_f$. 

For a $G$-cover $\varphi: C \to \Proj^{1}_{\F_q}$ and a non-trivial conjugacy class $c \in \mathcal{C}_G^*(\bar{\F}_q)$ let $n_c$ be the number of points $P \in \Proj^1_{\F_q}$, counted with multiplicity, such that the ramification type\footnote{See \Cref{con:inertia_type}, this is basically the monodromy of the cover at $P$, except that it keeps careful track of the Galois action.} at $P$ is $c$. The \emph{height} of $\varphi$ is defined as  $H_f(\varphi) = q^{\sum_{c \in \mathcal{C}_G^*(\bar{\F}_q)} n_c f(c)}$. For $d \in \N$ let $N_f(G, q^d)$ be the number of isomorphism classes of $G$-covers of $\Proj^1_{\F_q}$ of height $q^d$.

To state the theorem, we require the notion of \emph{residue} $\partial_c(\beta)$ of a Brauer element $\beta$ from \cite[Def.~5.25]{Loughran2025mallesconjecturebrauergroups}. For any element $\ell \in \HH^1(\F_q, \Q/\Z)$ we define a subset $\Br_{\mathcal{C}_f, \ell} BG_{\F_q} \subset \Br BG_{\F_q}$ in \Cref{def:Br_C_ell} in terms of this residue. 

Let $e: \HH^1(\F_q, \Q/\Z) \to \C^{\times}$ be the embedding defined by $\chi \to e^{2 i \pi \chi(\Frob_q)}$. For a finite extension $\F_{q^k}/\F_q$ we write $\mathrm{cor}_{\F_{q^k}/\F_q}: \HH^1(\F_{q^k}, \Q/\Z) \to \HH^1(\F_q, \Q/\Z)$ for the corestriction map.

\begin{theorem}\label{thm:main_theorem}
    Assume that $\mathcal{C}_f$ generates $G/Z(G)$ and that $q$ is sufficiently large in terms of $|G|$. We then have
    \[
    N_f(G, q^{d}) = c(G, f, d)d^{b(f, \F_q) - 1}q^{a(f) d} + O(d^{b(f, \F_q) - 2}q^{a(f) d})
    \]
    Here $c(G, f, d)$ is a constant which is $a(f)^{-1}|G|^2$-periodic in $d$. If $\mathcal{C}_f$ generates $G$ then it is equal to the finite sum of Euler products
    \[
    c(G, f, d) = \frac{a(f)^{b(f, \F_q)}|Z(G)|}{|\Hom(G, \F_q^{\times})| (b(f, \F_q) - 1)!}\sum_{\alpha \in \HH^1(\F_q, \Q/\Z)} e(d \cdot \alpha ) \sum_{\beta \in \Br_{\mathcal{C}_f, a(f)^{-1} \cdot \alpha} BG_{\F_q}} \hat{\tau}_{f, \alpha}(\beta).
    \]
    Here $\hat{\tau}_{f, \alpha}(\beta)$ is the conditionally convergent Euler product
    \[
    (1 - q^{-1})^{-b(f, \F_q)}\prod_{P \in \Proj^1_{\F_q}}(1 - q^{-\deg P})^{b(f, \F_q)}\hat{\tau}_{f, \alpha, P}(\beta).
    \]
    Where $\hat{\tau}_{f, \alpha, P}(\beta)$ is equal to
    \[
     1 + \sum_{\substack{c \in \mathcal{C}_G^*(\F_q(P)) \\ \partial_c(\beta) \in \HH^1(\F_q(P), \Q/\Z)}} e( \mathrm{cor}_{\F_q(P)/\F_q}(\partial_c(\beta)))e(-\deg P \cdot f(c) \cdot \alpha)q^{-(\deg P) a(f) f(c)}.
    \]

    Moreover, $\Br_{\mathcal{C}_{f}, a(f)^{-1} \cdot \alpha} BG_{\F_q} = 0 $ unless $\alpha \in \HH^1(\F_q, a(f)|G|^{-2}\Z/\Z)$ so $c(G, f, d)$ is $ a(f)^{-1}|G|^2$-periodic in $d$.
    \end{theorem}
We also prove an explicit constant in the general case that $\mathcal{C}_f$ generates $G/Z(G)$. We leave the precise formula for \Cref{thm:main_theorem_height_etale_unbalanced} as it is more complicated.

This theorem was proven for $G$ abelian and $f$ the index function by Wright \cite{Wright1989Abelian}. The only other known case is when $G = S_n$ and $f$ is the index. This case was proven for $G = S_3, S_4, S_5$ and all $q > 2$ in \cite{bhargava2015geometryofnumbers} and for all $n$ and $q$ sufficiently large in \cite[Thm~1.2.4]{landesman2025stablehomologyhurwitzmodules}\footnote{To get the exact same result one uses \Cref{prop:hurwitz_spaces_Tamagawa_measures} and a Tauberian theorem, such as \Cref{lem:Tauberian_theorem}.}.

To recover \Cref{thm:main_theorem_radical_disc} from \Cref{thm:main_theorem}, we can average over $d$ to eliminate all $\alpha \neq 0$. The correct analogue of the leading constant for the number field case of \cite{Loughran2025mallesconjecturebrauergroups} is thus only the contribution for $\alpha = 0$. A careful comparison shows that they are completely analogous, except for the fact that \Cref{thm:main_theorem} includes a factor $a(f)^{b(f, \F_q)}$ instead of the factor $a(f)^{b(f, \F_q(t)) -1}$ in \cite[Conj.~9.1]{Loughran2025mallesconjecturebrauergroups}. This discrepancy is explained in \Cref{rem:different_powers_a(f)}.
\begin{remark}
    One might wonder at first how the above formula explains the $G =\mathrm{PSL}_2(\F_7)$ example \cite[\S9.3.3]{ellenberg2013homologicalstabilityhurwitzspaces}. This is because $G$ has trivial abelianization, so \cite[Lem.~6.7]{Loughran2025mallesconjecturebrauergroups} shows that $\Br \mathrm{PSL}_2(\F_7)_{\F_q} \cong \Z/2 \Z$ and the non-trivial class $\beta$ is represented by the central extension $\mathrm{SL}_2(\F_7) \to \mathrm{PSL}_2(\F_7)$. Let $f$ be such that $\mathcal{C}_f$ consists of the single conjugacy class $c$ of order $4$ elements. Then \cite[Lem.~6.9]{Loughran2025mallesconjecturebrauergroups} and similar reasoning as in \cite[\S9.3.3]{ellenberg2013homologicalstabilityhurwitzspaces} shows that $\partial_{c}(\beta) \in \HH^1(\F_q, \Z/2 \Z)$ corresponds to the extension $\F_q(\sqrt{2})/\F_q$.

    If $\sqrt{2} \not \in \F_q$ and $a(f) = 1$ then this leads to $c(\mathrm{PSL}(\F_7), f, d)$ being the sum of two Euler products which causes significantly fewer extensions of height $q^d$ for $d$ odd than for $d$ even.
\end{remark}
\begin{remark}
    One interpretation of the subset $\Br_{\mathcal{C}_f, a(f)^{-1} \cdot \alpha} BG_{\F_q} \subset \Br (BG)_{\F_q}$ is that it is the subset of $\beta \in \Br BG_{\F_q}$ such that the Euler product converges. It diverges to $0$ for all other $\beta \in \Br BG_{\F_q}$.
\end{remark}
\begin{remark}\label{rem:non_geometrically_connected_covers}
    We count only geometrically connected $G$-covers in Theorems \ref{thm:main_theorem_radical_disc} and \ref{thm:main_theorem}. This is compatible with the framework of \cite{Loughran2025mallesconjecturebrauergroups} as the non-geometrically connected covers form a thin set.

    T\"urkelli \cite{Turkelli2015Connected} has observed that the non-geometrically connected covers can be decomposed into natural subfamilies such that the components of the corresponding Hurwitz spaces can be understood separately. In the language of \cite{Loughran2025mallesconjecturebrauergroups} these families are images of maps of the form $BH \to BG$ where $H$ is a finite \'etale group scheme over $\F_q$ with $H(\bar{\F}_q) \subsetneq G$ normal. We prove a form of \Cref{thm:main_theorem} which is valid for \'etale group schemes. One could thus try to apply this more general theorem to count points on $BH$.
    
    The issue is that $\mathcal{C}_f \cap H(\bar{\F}_q)$ generally does not generate $H(\bar{\F}_q)$, that is, the restriction of $f$ to $BH$ is not balanced in the sense of \cite[Def.~3.3]{Loughran2025mallesconjecturebrauergroups}. The leading constant in the unbalanced case will be significantly more complicated than in the balanced case, as explained in \cite[\S9.2]{Loughran2025mallesconjecturebrauergroups}. In particular, we expect that to handle the unbalanced case in full generality we need to be able to count $G$-torsors over $\F_q(t)$ for $G$ a finite \'etale group scheme over $\F_q(t)$ which does not come from $\F_q$.
\end{remark}

\subsection{Structure of the argument}
The work of Landesman--Levy \cite{landesman2025cohenlenstramomentsfunctionfields,Landesman2025Homological, landesman2025stablehomologyhurwitzmodules} shows that the number of points on a component of a Hurwitz space is well-approximated by the number of $\F_q$-points on a configuration space. To determine the asymptotics for the number of $G$-extensions we need to carry out the following three steps.
\begin{enumerate}
    \item Determine the stable components of Hurwitz spaces over $\bar{\F}_q$.
    \item Compute how many of these components are fixed by Frobenius.
    \item Determine the asymptotic for the number of points on configuration spaces.
\end{enumerate}

The first problem has been mostly solved by Ellenberg--Venkatesh--Westerland and Wood \cite{ellenberg2013homologicalstabilityhurwitzspaces, Wood2021Lifting}. The main tool is the \emph{lifting invariant} which depends on a conjugacy invariant subset $C \subset G$. For our application we still need two ingredients: a definition of the lifting invariant over general fields when $C$ is not closed under invertible powers and a proof that the lifting invariant is a bijection when one only stabilizes in some of the directions. This is the content of \Cref{thm:lifting_invariant} and \Cref{thm:components_partial_stability}.

How we deal with the second part is one of the main novelties of the paper. We show that the number of components which are fixed by Frobenius is equal to a sum over Brauer elements of $\Br BG_{\F_q}$ whose residues are algebraic. In some sense, the Brauer group is dual to the components. This is done in \Cref{sec:Brauer_group} via a cocycle-theoretic computation. This part of the argument relies crucially  on the theory developed in \cite{Loughran2025mallesconjecturebrauergroups}.

The last step is dealt with in two parts. First we show in \Cref{prop:hurwitz_spaces_Tamagawa_measures} that the number of points on configuration spaces is exactly equal to the Tamagawa volume of certain natural subsets of $BG[\Adele_{\F_q(t)}]$. In the setting of varieties the asymptotic volumes of similar subsets were studied in \cite{Chambert-Loir2010Igusa} and we can use exactly the same tools in this case, namely a Tauberian theorem and the computation of the local Igusa zeta function.

\subsection{Possible extensions}
The homological stability results of \cite{landesman2025stablehomologyhurwitzmodules} hold for Hurwitz modules, which are significantly more general than the Hurwitz spaces used in this paper. For example, they allow one to consider more general global function fields than $\F_q(t)$ and they also allow one to impose arbitrary local conditions at finitely many points of $\Proj^{1}_{\F_q}$. Unfortunately, the strategy used in this paper to turn such homological stability results into quantitative arithmetic results fails already at the first step: there is currently no good description of the stable components of more general Hurwitz modules.

The stable components of the Hurwitz modules for more general global functions fields were determined over $\C$ in \cite{Samperton2020Schur}. However, to carry out the second step it is not sufficient to have a description of the geometric stable components, one needs control on the action of Frobenius.

If a good description of stable components of Hurwitz modules is found then we expect that the second and third step can be generalized to deal with such more general situations.
\subsection{Notations and conventions}
We write $\N$ for the non-negative integers. 

A finite \'etale group scheme $G$ over a base $B$ is \emph{tame} if its order is invertible in $B$. For such a tame group scheme we let $G(-1) = \varinjlim_n \Hom(\mu_n, G)$ be the anti-cyclotomic twist. We also consider $\mathcal{C}_G := G(-1)/G$ and $\mathcal{C}_G^* := \mathcal{C}_G  \setminus \{1\}$.

For schemes $X,Y$ over a base scheme $B$ we will write $\Hom(X, Y)$ for the \'etale sheaf of maps $X \to Y$. We will write $\Hom_B(X, Y) := \Hom(X, Y)(B)$ for the set of maps over $B$.

We will fix an algebraic closure $\bar{\F}_q$ of $\F_q$ for all prime powers $q$. We let $\Frob_q \in \Gal(\bar{\F}_q/\F_q)$ be the (arithmetic) Frobenius. We will freely use Galois theory to identify finite \'etale schemes $X/\F_q$ with their Galois set of points $X(\bar{\F}_q)$ and similarly for profinite/ind-finite \'etale schemes.

Given a scheme $X$ we let $\pi_0(X)$ denote the set of connected components, for a finite \'etale scheme $X/\F_q$ this is the same as the Frobenius orbits of $X(\bar{\F}_q)$.

We will freely use the language of stacks. All stacks in this paper will be assumed to be separated and Deligne-Mumford. If $G$ is a group scheme acting on the left on a scheme $X$ then we write $[X/G]$ for the corresponding quotient stack. For a stack $X$, a scheme $S$, resp.~a ring $R$ we let $X(S)$, resp.~$X(R)$ be the groupoid of morphisms $S \to X$, resp.~$\Spec R \to X$ and $X[S]$, resp.~ $X[R]$ the set of components of this groupoid. If $x \in X(S)$ then we write $\Aut_X(x)$ for the sheaf of automorphism groups of $x$.

The main stack which appears is $BG := [S/G]$ for a finite \'etale group scheme $G$ over a base scheme $S$, for the basic theory of the arithmetic of this stack we refer to \cite[\S2]{Loughran2025mallesconjecturebrauergroups}. 

All cohomology groups are \'etale/Galois cohomology groups, unless otherwise noted.

We will use the embedding $e: \HH^1(\F_q, \Q/\Z) \to \C^{\times}$ from the introduction, defined by the formula $\chi \to e^{2 i \pi \chi(\Frob_q)}$.

We use big $O$ and Vinogradov's $\ll$ notation.

In the following table we give symbols which will be frequently used and their meaning.
\begin{table}[h]
    \centering
\begin{tabular}{|c|c|}
    \hline
     \textbf{Symbol} & \textbf{Meaning}  \\ \hline
      $B$ & A base scheme. \\ \hline
      $G$ & A tame finite \'etale group scheme over $B$. \\ \hline
      $g, h$ & Elements of $G$. \\ \hline
      $\gamma, \xi$ & Elements of $G(-1)$. \\ \hline
      $\zeta$ & A root of unity $\zeta \in \mu_n$ or a compatible system of roots of unity. \\ \hline
      $\varphi$ & An element of $BG(B)$, i.e.~a $G$-torsor over $B$. \\ \hline
      $C$ & A conjugacy invariant subset $C \subset G(-1) \setminus \{1\}$. \\ \hline
      $\mathcal{C}$ & A subscheme of $\mathcal{C}_G^*$, often equal to $C/G$. \\ \hline
      $P$ & A closed point $P \in \Proj^1$. \\ \hline
      $Z$ & A connected component of a Hurwitz scheme. \\ \hline
      $\mathbf{P}$ & A non-empty collection of prime numbers. \\ \hline
      $\underline{n}$ & An element of $\Hom(\mathcal{C}, \N)$. \\ \hline
      $\beta$ & An element of the Brauer group $\Br BG$. \\ \hline
\end{tabular}
\end{table}

In Table 1 we provide a summary and a reference to the definition for the important notation introduced later in the paper. 
\begin{table}[h]\label{tab:notations}
    \centering
    \caption{Summary of notations introduced in the paper.}
    \renewcommand{\arraystretch}{1.3}
    \begin{tabular}{|c|p{10cm}|c|}
        \hline
        \textbf{Notation} & \textbf{Description} & \textbf{Ref.} \\
        \hline
        $\sqrt[w]{D/X}$ & Root stack of $X$ of degree $w$ along D & \S\ref{sec:root_stacks} \\
        \hline
        $\sqrt[w]{D}$ & Divisor of $\sqrt[w]{D/X}$ which is a gerbe over $D$ & \S\ref{sec:root_stacks} \\
        \hline
        $\rho_{G, D}$ & Local ramification type of a $G$-cover at a divisor $D$. & \ref{con:inertia_type} \\
        \hline
        $\Conf_{C/B, \underline{n}}$ & \text{Multicolored configuration space for the colors $C$}. & \ref{def:configuration_space} \\
        \hline
        $\Hur^{G, \varphi}_{B, n}$ & Hurwitz space of $G$-covers ramified at $n$ points and  trivialized with respect to $\varphi$ at $\infty$. & \ref{def:Hurwitz_space} \\
        \hline
        $\CHur^{G, \varphi}_{B, n}$  & Components of $\Hur^{G, \varphi}_{B, n}$ corresponding to geometrically connected covers. & \ref{def:Hurwitz_space} \\
        \hline
        $\pi_0(\Hur^{G, \varphi}_{B, n}/B)$ & Finite \'etale scheme of connected components of $\Hur^{G, \varphi}_{B, n}$. & \ref{def:scheme_of_connected_components} \\
        \hline
        $\mathfrak{R}$ & Map $\pi_0(\Hur^{G, \varphi}_{B}/B) \to \Hom(\mathcal{C}_G^*, \N)$ counting points with given ramification type. & \ref{def:Multiplicity_Hurwitz}
        \\
        \hline
        $(\mathrm{C})\Hur^{G, C, \varphi}_{B}$ & Components of $(\mathrm{C})\Hur^{G, \varphi}_{B}$ corresponding to curves with $\rho_{G,P} \in \mathcal{C}$ & \ref{def:Multiplicity_Hurwitz} \\
        \hline
        $(\mathrm{C})\Hur^{G, C, \varphi}_{B, \underline{n}}$ & Components $Z \subset (\mathrm{C})\Hur^{G, C, \varphi}_{B}$ with $\mathfrak{R}(Z) = \underline{n}$ & \ref{def:Multiplicity_Hurwitz} \\
        \hline
        $(\mathrm{C})\Hur^{G, C, (\varphi, \gamma)}_{B, \underline{n}}$ & Components $Z \subset (\mathrm{C})\Hur^{G, C, \varphi}_{B, \underline{n}}$ corresponding to $G$-covers with monodromy $\gamma$ at $\infty$. & \ref{def:Huritz_space_gamma} \\
        \hline
        $\HurStack^{G}_{B, n}$ & Hurwitz stack of $G$-covers ramified at $n$-points. & \ref{def:Hurwitz_stack} \\
        \hline
        $\hat{U}_{B, \mathbf{P}}(G,  \mathcal{C})(1)$ & Universal $\mathbf{P}$-primary $\mathcal{C}$-marked central extension & \ref{def:universal_marked_central_extension} \\ 
        \hline
        $\HH_{2,\mathrm{orb}}^{\mathcal{C}}(G, \hat{\Z}_{\mathbf{P}}(1))$ & Kernel of $\hat{U}_{B, \mathbf{P}}(G,  \mathcal{C})(1) \to G$ & \ref{def:universal_marked_central_extension} \\
        \hline
        $U_{B}(G,  \mathcal{C})$ & Ind-finite \'etale scheme where the lifting invariant is valued. & \ref{def:lifting_invariant_multiplicity} \\
        \hline 
        $\HH_{2,\mathrm{orb}}^{\mathcal{C}}(G, \Z)$ & Fiber at $1$ of the map $U_{B}(G,  \mathcal{C}) \to G(-1)$ &  \ref{def:lifting_invariant_multiplicity} \\
        \hline
        $\HH_2(G, \mathcal{C})$ & Kernel of $\HH_{2,\mathrm{orb}}^{\mathcal{C}}(G, \Z) \to \Hom(\mathcal{C}, \N)$. & \ref{def:lifting_invariant_multiplicity} \\
        \hline
        $U(G_{\varphi}, \mathcal{C})_{\underline{n}, \gamma}$ & Elements of $U_{B}(G_{\varphi},  \mathcal{C})$ which map to $\underline{n} \in \Hom_B(\mathcal{C}, \N)$ and $\gamma^{-1} \in G(-1)$. & \ref{def:U(G,C) n gamma} \\
        \hline
        $\mathfrak{z}$ & The lifting invariant $\pi_0(\Hur^{G, C, \varphi}_{B}/B) \to U(G_{\varphi},  \mathcal{C}).$ & \ref{def:lifting_invariant} \\ \hline
        $\Br_{\bar{\mathcal{C}}} BG_{\F_q}$ & Group of elements which have algebraic residues & \ref{def:Brauer_group_C_bar}. \\ \hline
        $\Br_{\mathcal{C}, \ell} BG_{\F_q}$ & Brauer elements whose residues are $\ell$. & \ref{def:Br_C_ell}\\ \hline
        $\tau_P$ & Local Tamagawa measure at $P$. & \S\ref{sec:Tamagawa_measures} \\ \hline
        $\tau_{BG, \F_q}$ & Global Tamagawe measure $\prod_{P \in \Proj^1_{\F_q}} \tau_P$ on $BG[\Adele_{\F_q(t)}]$. & \S \ref{sec:Tamagawa_measures} \\ \hline
        $\tau_{BG, \Adele_{K, \infty}}$ & Global Tamagawe measure $\prod_{P \in \A^1_{\F_q}} \tau_P$ on $BG[\Adele_{\F_q(t), \infty}]$.  & \S \ref{sec:Tamagawa_measures} \\ \hline
        $\hat{\tau}_P(\beta ; \Omega_P)$ & Local Brauer transform. & \S \ref{sec:Brauer_Manin_pairing}\\ \hline
        $\hat{\tau}(\beta ; \Omega)$ & Global Brauer transform. & \S \ref{sec:Brauer_Manin_pairing} \\\hline
    \end{tabular}
\end{table}
\subsection{Acknowledgments}
I am grateful to  Melanie Wood, Aaron Landesman, Ishan Levy and Daniel Loughran for useful discussions and encouragement. In particular I want to thank Aaron Landesman for carefully reading a previous version and the suggestion to generalise the paper to the case where $\mathcal{C}_f$ only generates $G/Z(G)$ instead of $G$. The author was supported by the Herschel-Smith Fund.
\section{Geometry}
\subsection{Root stacks}\label{sec:root_stacks}
We will recall some facts about root stacks, see \cite{Cadman2007stacks} for a detailed account. These stacks have proven play in important role when studying the arithmetic of stacks. For example, they appear in the arithmetic valuative criterion for properness \cite{Bresciani_2024Valuative} which plays a crucial role in \cite[\S4,5]{Loughran2025mallesconjecturebrauergroups}. They also appear in the definition of the Hurwitz schemes \cite[Def.~2.4.5]{ellenberg2025homologicalstabilitygeneralizedhurwitz} used in the work of Landesman--Levy.

If $X$ is a stack, $D \subset X$ is a divisor and $w \in \N$ then Cadman \cite[Def.~2.2.1]{Cadman2007stacks} defines the \emph{root stack} $\sqrt[w]{D/X}$ of order $w$ along $D$, which is functorial in $X$ and $D$. If $X = \Spec A$ and $D$ is the vanishing locus of $a \in A$ then by \cite[Ex.~2.4.1]{Cadman2007stacks} we have
\begin{equation}\label{eq:charts_root_stack}
    \sqrt[w]{D/X} \cong [\Spec A[x]/(x^w = a)/ \mu_w].
\end{equation}
Where $\mu_w$ acts by multiplication on $x$. As this is true locally on $X$ it describes the root stack on charts.

There is a divisor $\sqrt[w]{D} \subset \sqrt[w]{D/X}$ which is a $\mu_r$-gerbe over $D$, see \cite[Def.~2.4.4]{Cadman2007stacks}. In the chart \eqref{eq:charts_root_stack} it is given by $x = 0$.

\subsection{$G$-covers}\label{sec:G_cover}
Let $G$ be a finite \'etale group scheme over a base scheme $B$. The goal of this section will be to recall some of the theory relevant to the study of $G$-covers. If $G$ is constant group scheme then the notion of the cover of a smooth scheme is well-known. But in this paper $G$ will in general be non-constant, which is slightly subtler.
\begin{definition}\label{def:G_cover}
Let $X$ be a smooth and finitely presented scheme over $B$. A \emph{left $G$-cover} of $X$ is a finite map of $B$-schemes $f:Y \to X$, where $Y$ is equipped with a left $G$-action such that $f$ is a $G$-torsor when restricted to a $B$-dense open substack $U \subset X$. We call the complement of the maximal such $U$ the \emph{ramification divisor of $f$}.

Note that $f$ is completely determined by its restriction to any dense $U \subset X$ as $X$ is normal.

A morphism of $G$-covers from $f: Y \to X$ to $f': Y' \to X$ is a $G$-equivariant map $Y \to Y'$ making the obvious triangle commute\footnote{If $X$ is a stack then the diagram commuting is not just a property but a piece of data. But $X$ is assumed to have a dense open subscheme so this technicality can be ignored.}. Note that a morphism of $G$-covers is automatically an isomorphism as it is finite and birational.
\end{definition}

We will study $G$-covers of curves. It will turn out to be important to work with stacky curves, whose definition we recall.
\begin{definition}\label{def:stacky_curve}
A \emph{stacky curve}\footnote{This is also known as a twisted curve.} over $B$ is a smooth finitely presented stack $X$ over $B$ of relative dimension $1$ and with geometrically connected fibers and a $B$-dense open subscheme. If $X$ is a scheme then we call it a curve.

We call a stacky curve \emph{tame} if it is tame as a stack, i.e. the orders of automorphism group schemes are invertible.
\end{definition}
We note that the stackiness of a stacky curve is very mild, it only arrives via root stack constructions.
\begin{lemma}\label{lem:stacky_curves_are_root_stacks}
    A tame stacky curve $C$ over $B$ is canonically obtained by starting with a curve $C_{\mathrm{coarse}}$ over $B$ (its \emph{coarse moduli space}) and taking iterated root stacks of $C_{\mathrm{coarse}}$ along disjoint divisors $D_i \subset C_{\mathrm{coarse}}$, where each $D_i$ is finite \'etale over $B$.
\end{lemma}
\begin{proof}
    If $C$ is proper and $B$ Noetherian then this is \cite[Thm.~4.1.]{Cadman2007stacks}. The proof is local and never uses properness, the general case is reduced to the Noetherian case in the standard way using that $f$ is finitely presented.
\end{proof}

The following proposition explains why stacky curves play a role when studying $G$-covers of curves.
\begin{proposition}\label{prop:G_covers_curves}
    Let $C$ be a curve over $B$ and $f:X \to C$ a $G$-cover. The ramification divisor $D$ of $f$ is a divisor and finite \'etale over $B$. Moreover, for any finite \'etale divisor $D \subset C$ the following three categories are equivalent, moreover the equivalences are functorial in $B$ and $G$.
    \begin{enumerate}
        \item The category of $G$-covers over $C$ whose ramification divisor is contained in $D$.
        \item The category of pairs $(\mathcal{C}, \pi: X \to \mathcal{C})$ where $\mathcal{C}$ is a stack obtained by rooting along divisors contained in $D$, $\pi$ is a $G$-torsor and $X$ is a scheme.
         \item The category of $G$-torsors over $C \setminus D$.
    \end{enumerate}

    The equivalence is given by the following maps
    \begin{enumerate}
        \item  $\to \mathrm{(2)}$: Sends a $G$-cover $X \to C$ to the pair $([X/G], X \to [X/G])$.
        \item  $\to \mathrm{(3)}$: Restrict $\pi$ to $C \setminus D$
        \item  $\to \mathrm{(1)}$: It sends $X \to C \setminus D$ to the finite map $Y \to C$ extending $X \to C$.
    \end{enumerate}
\end{proposition}
\begin{proof}
   Apply \Cref{lem:stacky_curves_are_root_stacks} to the stacky curve $[X/G]$. The ramification divisor of $f$ is exactly the union of divisors on which one has to root $C$ to get $[X/G]$ and is thus finite \'etale. This reasoning also implies that the first functor is well-defined.

    The second functor is clearly well-defined. The third functor is well-defined as the $G$-action as $Y \to C$ is unique up to unique isomorphism so the $G$-action extends.

    Any composition of these three functors is equivalent to the identity. Indeed all stacks are separated so their morphisms are determined by the open dense subscheme $C \setminus D$ where the functors are the identity.
\end{proof}
\subsection{Ramification type}
To study Hurwitz spaces for general group schemes we will need a notion of the monodromy of a $G$-extension at a point which has good functorial properties. This will be provided by the \emph{ramification type}. Over fields this was defined in \cite[\S7.1]{Loughran2025mallesconjecturebrauergroups}. We extend the definition to general base schemes.

\begin{construction}\label{con:inertia_type}
    Let $C \to B$ be a curve and $\varphi: X \to C$ a $G$-cover whose ramification divisor is $D$. This divisor is finite \'etale over $B$ by \Cref{prop:G_covers_curves}. Consider the stacky curve $[X/G]$. It follows from \Cref{lem:stacky_curves_are_root_stacks} that $[X/G]$ is an iterated root stack over $C$ along a decomposition of the divisors $D = \sqcup_{i \in I} D_i$. Let $w_i$ be the degree of the rooting along $D_i$. 

    Consider the $\mu_{w_i}$-gerbe $\sqrt[w_i]{D_i} \subset [X/G]$ over $D_i$. Choose an \'etale cover $S \to B$ such that the gerbe $\sqrt[w_i]{D_i}$ is neutral and choose an isomorphism $(\sqrt[w_i]{D_i})_{S} \cong (B\mu_{w_i})_{(D_i)_S}$. Note that this isomorphism is unique (up to a non-unique $2$-isomorphism). This defines a map $(D_i)_S \to I_{\mu} BG$ to the cyclotomic inertia stack of $BG$, see \cite[Def.~4.5]{Loughran2025mallesconjecturebrauergroups}. We can compose this with the map $I_{\mu} BG \to G(-1)/G =  \mathcal{C}_G$ from \cite[Prop.~4.6]{Loughran2025mallesconjecturebrauergroups}. The condition that $f$ is ramified at $D$ ensures that it factors through $\mathcal{C}_G^*$:

    Moreover, the composition $(D_i)_S \to \mathcal{C}_G^*$ descends to a map $D_i \to \mathcal{C}_G^*$ as it is a map of schemes and the chosen isomorphism $(\sqrt[w_i]{D_i})_S \cong (B\mu_{\underline{w}})_S$ is unique. We call the union of maps $D = \bigsqcup_i D_i \to \mathcal{C}_G^*$ map the \emph{ramification type} and denote it by $\rho_{G, D}(\varphi)$. More generally, for any subdivisor $D' \subset D$ we write $\rho_{G, D'}(\varphi) = \rho_{G, D}|_{D'}: D' \to \mathcal{C}_G^*$. If $B$ is the spectrum of a field $k$ and $P \in D(k)$ then $\rho_{G, P}(\psi) \in \mathcal{C}_G^*(R)$ is exactly the ramification type defined in \cite[\S7.1]{Loughran2025mallesconjecturebrauergroups}.
\end{construction}
\subsection{Twisted configuration spaces}
We will need certain twisted versions of the multicolored configuration spaces in \cite[Def.~2.2.1]{Landesman2025Homological}, which we now define.

Let $B$ be a base scheme, $C \to B$ a finite \'etale map, which we consider as a set of colors, and $\underline{n}: C \to \N$ a locally constant map. 
\begin{definition}\label{def:configuration_space}
    We define $\Conf_{C/B, \underline{n}}$ as the configuration space parametrizing $n_c$ unordered points in $\A_B^1$ for each color $c \in C$. More precisely, it represents the functor which sends $T \to B$ to the set of isomorphism classes of pairs $(D \to C_T, i: D \to \A^{1}_T)$ such that $D \to C_T$ is a finite \'etale cover of degree $n_c$ at $c \in C$ and $i$ is a closed immersion. An isomorphism $(D, i) \cong (D', i')$ consists of an isomorphism $D \cong D'$ making the obvious diagrams commute.

    We drop $B$ and $C$ from the notation if they are clear from the context.

    An important special case is when $C = B$. In this case we write $\Conf_{B, n} := \Conf_{B/B, n}$. This is the usual configuration space.
\end{definition}
\begin{remark}

    Note that if $C = \{1, \cdots, k\} \times B$ and $\underline{n}$ takes the value $n_i$ on $\{i\} \times B$ then this agrees with the usual multicolored configuration $\Conf_{B, n_1, \dots, n_k}$.

    The given functor is clearly an \'etale sheaf. Moreover, by the above it is \'etale locally on $B$ represented by a scheme affine over $B$, namely the usual multicolored configuration space. Descent implies that the functor is represented by a scheme affine over $B$.
\end{remark}

Now, consider the case in which $B = \Spec \F_q$. In this case, $\underline{n}$ corresponds to a Galois invariant function $C(\bar{\F}_q) \to \N$. We can determine the size of $\Conf_{C, \underline{n}}(\F_q)$ directly from the definition. For a closed point $P \in \A^{1}_{\F_q}$ we let $\F_q(P)/\F_q$ be its field of definition and $\deg P$ its degree.
\begin{lemma}\label{lemma:number_points_Configuration_space}
    For every connected component $c \in \pi_0(C)$ consider a variable $T_c$. 
    
    The size $\# \Conf_{C, \underline{n}}(\F_q)$ is equal to the coefficient of $\underline{T}^{\underline{n}} = \prod_{c \in \pi_0(C)} T_c^{[\F_q(c): \F_q]n_c}$ in the power series
    \[
    \prod_{P \in \A_{\F_q}^{1}}(1 + \sum_{c \in \pi_0(C)} |c(\F_q(P))|T_c^{\deg P})
    \]
\end{lemma}
\begin{proof}
    A point of $\Conf_{C, \underline{n}}(\F_q)$ consists by definition for each component $c \subset C$ of a closed subscheme $D_c \subset \A^1_{\F_q}$ and a map $D_c \to c$ of degree $n_c$ such that the $D_c$ are pairwise disjoint. The closed subscheme $D_c$ is a union of the closed points $\sqcup_{i \in I_c} P_i$. It remains to determine the number of maps $D_c \to c$.

    As $P_i = \Spec \F_q(P_i)$ the number of maps $P_i \to c$ is exactly $|c(\F_q(P_i))|$. Thus, the number of maps $D_c \to c$ is $\prod_{i \in I}|c(\F_q(P_i))|$. 

    We have thus shown that $\#\Conf_{C, \underline{n}}(\F_q)$ is equal to the sum over pairwise disjoint collections of points $\{P_i: i \in I_c\}$ for each component $c \subset C$ such that $\sum_{i \in I_c} [\F_{P_i}: \F_q(c)] = n_c$ and the sum is weighted by $\prod_{c \in \pi_0(C)}\prod_{i \in I_c}|c(\F_{P_i})|$. This is exactly the coefficient of $\underline{T}^{\underline{n}}$ in the given Euler product.
\end{proof}
Define $|\underline{n}| := \sum_{c \in C(\bar{\F}_q)} n_c$.
\begin{lemma}\label{cor:points_configuration_space}
    We have $\#\Conf_{C, \underline{n}}(\F_q) \leq q^{|\underline{n}|}$.
\end{lemma}
\begin{proof}
    We have the coefficient wise inequality
    \[
    \begin{split}
    \prod_{P \in \A_{\F_q}^{1}}(1 + \sum_{c \subset C} |c(\F_q(P))|T_c^{\deg P}) &\leq \prod_{P \in \A_{\F_q}^{1}} \prod_{c \in \pi_0(C)}(1 +  T_c^{\deg P})^{|c(\F_q(P))|} \\ &\leq \prod_{c \in \pi_0(C)} \prod_{P \in \A^1_{\F_q(c)}} (1 - T_c^{-[\F_q(c): \F_q]\deg P})^{-1}.
    \end{split}
    \]

    The Euler product $\prod_{P \in \A^1_{\F_q(c)}} (1 - T_c^{-[\F_q(c): \F_q]\deg P})^{-1}$ is the zeta function of $\A^1_{\F_q(c)}$ which is equal to the geometric series $\sum_{i = 0}^\infty q^{[\F_q(c): \F_q] i } T_c^{[\F_q(c): \F_q] i }$. By \Cref{lemma:number_points_Configuration_space} we thus have 
    $\#\Conf_{C, \underline{n}}(\F_q) \leq \prod_{c \in \pi_0(C)} q^{[\F_q(c): \F_q] n_c} = q^{|n|}$.
\end{proof}
    
\subsection{Twisted Hurwitz schemes}
To count $G$-extensions of $\A^1_{\F_q}$ for $G$ a finite tame \'etale group scheme with specified local conditions at $\infty$ we will need a twisted version of Hurwitz schemes. We study these schemes in this section.

Fix a base scheme $B$ and a finite \'etale tame group scheme $G$ over $B$. 

As in \cite[Def.~2.3.1]{Landesman2025Homological}, we let $\mathcal{P}^w_T := \sqrt[w]{\infty/ \Proj^1_T}$ be the root stack \cite[Def.~2.2.1]{Cadman2007stacks} of order $w \in \N$ of $\Proj^1_T$ along $\infty$ for any scheme $T$. There is a canonical section $\tilde{\infty}_T: T \to \mathcal{P}^w_T$ lifting the section at $\infty$ given on the chart $t \neq 0$ as 
\begin{equation}\label{eq:infinity_tilde_charts}
    T \xrightarrow{z = 0} \Spec \mathcal{O}_T[t^{-1}]/(z^w = t^{-1}) \to [\Spec \mathcal{O}_T[t^{-1}]/(z^w = t^{-1})/\mu_w] \subset \mathcal{P}^w_T.
\end{equation}
See \eqref{eq:charts_root_stack} for this chart.
\begin{definition}\label{def:Hurwitz_space}
    For $n \in \N$ we define $\Hur^{G, \varphi}_{B, n}$ as the scheme over $B$ represented by the functor which sends $T \to B$ to the collection of isomorphism classes of tuples
    \[
       (D \to T, i: D \to \A^1_T, w, f: X \to \mathcal{P}^w_T , \eta:  f|_{\tilde{\infty}_T} \cong \varphi ).
    \]
    such that $D \to T$ is finite \'etale of degree $n$, $i$ is a closed immersion, $w \in \N$, $X$ is a scheme equipped with a $G$-action, $f$ is a map which is a $G$-torsor over $\mathcal{P}^w_T \setminus D$, the composition with $\mathcal{P}^w_T \to \Proj^1_T$ is a $G$-cover whose ramification divisor contains the image of $i$ and $\eta$ is an isomorphism of $G$-torsors over $T$.

    An isomorphism of tuples $(D, i, w, f, \eta) \to (D', i', w, f', \eta')$ consists of a $T$-isomorphism $D \cong D'$, an equality $w = w'$ and a $G$-equivariant morphism $X \to X'$ making all the obvious diagrams commute.

    We define $\CHur^{G, \varphi}_{B, n} \subset \Hur^{G, \varphi}_{B, n}$ as the open and closed locus where $X \to T$ has geometrically connected fibers.

    We put $\Hur^{G, \varphi}_B := \coprod_n \Hur^{G, \varphi}_{B, n}$ and $\CHur^{G, \varphi}_B := \coprod_n \CHur^{G, \varphi}_{B,n}$.

    If $\varphi$ is the trivial torsor $e_{BG}$, then it may be dropped from the notation.
\end{definition}
\begin{remark} \hfill
    \begin{enumerate}
        \item If $G$ is a constant group scheme and $\varphi$ is the trivial torsor then this definition agrees with \cite[Def.~2.3.1]{Landesman2025Homological}. This is because the section $t$ of loc.~cit.~is equivalent data to an isomorphism with the trivial torsor.

        \item We note that the pair $(D, i)$ is determined up to unique isomorphism by $f$. 
    \end{enumerate}
\end{remark}
A consequence of the following lemma is that $\Hur^{G, \varphi}_{B, n}$ is reprresented by an affine scheme.
\begin{lemma}\label{lem:Hurwitz_scheme_finite_etale}
    The map $\Hur^{G, \varphi}_{B, n} \to \Conf_{B/B, n}$ which sends $(D, i, w, f, \eta)$ to $D \to B$ is finite \'etale.
\end{lemma}
\begin{proof}
    This can be checked \'etale locally so we may assume that $G$ is constant and $\varphi = e_{BG}$, which is \cite[Rem.~2.1.4]{landesman2025cohenlenstramomentsfunctionfields}.
\end{proof}
We need to recall some facts on twisting by a torsor. What follows and more can be found in \cite[pp.~20-22]{Skorobogatov2001Torsors}. Given a scheme $X$ over $B$ equipped with a left $G$-action and a right $G$-torsor $\psi: B' \to B$ we can get the twisted algebraic space $X_{\psi} = [B' \times X/G]$ which is equipped with a left $G_{\psi}$-action, where $G_{\psi}$ is the inner twist of $G$. Moreover, twisting is functorial and if $X$ was a left $G$-torsor then $X_{\psi}$ is a left $G_{\psi}$-torsor.

Recall that if $\varphi: B' \to B$ is a left $G$-torsor then the \emph{inverse torsor} $\varphi^{-1}$ is the right $G$-torsor defined by $b \cdot_{\varphi^{-1}} g := g^{-1} \cdot_{\varphi} b$. To ease notation write $G_{\varphi} := G_{\varphi^{-1}}$, then $\varphi_{\varphi^{-1}} \cong e_{BG_{\varphi}}$.
\begin{construction}\label{cons:assume_varphi_trivial}
    We construct an isomorphism $\Hur^{G, \varphi}_{B, n} \to \Hur^{G_{\varphi}}_{B, n}$. It is the identity on $D, i$. It sends $f$ to the twisted cover $f_{\varphi^{-1}}: X_{\varphi^{-1}} \to \mathcal{P}_w$ and sends $\eta$ to the composition  $f_{\varphi^{-1}}|_{\tilde{\infty}_T} \cong \varphi_{\varphi^{-1}} \cong e_{BG_{\varphi}}.$ This map is the identity if $\varphi = e_{BG}$ which is true \'etale locally on $B$, so it is an isomorphism.
\end{construction}
\begin{definition}\label{def:scheme_of_connected_components}
    We recall from \cite[Def.~2.1.1]{Romagny2011Composantes} that an \emph{open connected $T$-component} of a finitely presented $T$-scheme $X \to T$ is a flat finitely presented open $T$-subscheme $U \subset X$ such that for all geometric points $t \in T$ the open subscheme $U_t \subset X_t$ is a connected component.
    
    We define $\pi_0(\Hur^{G, \varphi}_{B, n}/B)$ as in \cite[Def.~2.1.1]{Romagny2011Composantes}. It is the functor which sends $T \to B$ to the set of open connected $T$-components of $\Hur^{G, \varphi}_{B, n}$.
    
    We claim that this functor is represented by a finite \'etale scheme. By descent we may assume that $G$ is a constant group scheme and $\varphi = e_{BG}$. In this case $\Hur^{G, \varphi}_{B, n}$ has a smooth $T$-compactification by \cite[Cor.~B.1.4]{ellenberg2025homologicalstabilitygeneralizedhurwitz}. The claim then follows by applying \cite[Prop.~3.2.5]{Romagny2011Composantes} to the compactification, where we use \cite[Cor.~2.6.2.]{Romagny2011Composantes} to deduce that $\Hur^{G, \varphi}_{B, n}$ and its compactification have the same open connected components.
    
    There is a map $\Hur^{G, \varphi}_n \to \pi_0(\Hur^{G, \varphi}_n/B)$ which sends a point to the unique open connected component containing it, see \cite[Prop.~2.2.1]{Romagny2011Composantes} for the existence of this open connected component. This map is initial among maps from $\Hur^{G, \varphi}_n$ to finite \'etale schemes since a map to a finite \'etale scheme is constant on connected components. 

    We have $\pi_0(\Hur^{G, \varphi}_{B}/B) = \bigsqcup_n \pi_0(\Hur^{G, \varphi}_{B, n}/B)$ by definition.
\end{definition}
\subsubsection{Ramification type}
The components of Hurwitz spaces are often distinguished by the number of points whose inertia is contained within a certain conjugacy class. To keep track of Galois actions and for compatibility with \cite{Loughran2025mallesconjecturebrauergroups} we use the essentially equivalent notion of ramification type which we introduced in \Cref{con:inertia_type}. For an element $f := (D, i, w, f, \eta) \in \Hur^{G, \varphi}(T)$ we will write $\rho_{G, f} := \rho_{G, D}(f): D \to \mathcal{C}_G^*$ for simplicity.

\begin{definition}\label{def:Multiplicity_Hurwitz}.
    Let $\mathfrak{R}: \Hur^{G, \varphi} \to \Hom(\mathcal{C}_G^*, \N)$ be the map which sends the tuple $(D, i, w, f, \eta)$ to the function $\mathfrak{R}(X)$ such that $\mathfrak{R}(X)(c)$ is the degree of the ramification type $\rho_{G, f}: D \to \mathcal{C}_G^*$ at $c \in \mathcal{C}_G^*$. This map factors through $\pi_0(\Hur^{G, \varphi}/B)$ as $\Hom(\mathcal{C}_G^*, \N)$ is ind-finite \'etale.

    For a conjugacy-invariant subset $C \subset G(-1) \setminus \{1\}$ with corresponding scheme of conjugacy classes $\mathcal{C} \subset \mathcal{C}_G^*$. We define $\Hur^{G, C, \varphi}, \CHur^{G, C,\varphi}$ as the union of components $Z$ of $\Hur^{G, \varphi}, \CHur^{G, \varphi}$ such that $\mathfrak{R}(Z) \in \Hom(\mathcal{C}, \N)$. In other words, these are the components corresponding to $G$-covers whose ramification type is supported on $\mathcal{C}$.

    For $\underline{n} \in \Hom_B(\mathcal{C}, \N)$ we denote by $\Hur^{G, C, \varphi}_{\underline{n}}$ the union of components $Z  \subset \Hur^{G, C, \varphi}$ such that $\mathfrak{R}(Z) = \underline{n}$.
\end{definition}
\begin{remark}
    If $B = \Spec \C$ and we choose a compatible system of roots of unity to make $G(-1) \cong G$ then the ramification type at a point is just the conjugacy class which contains the generator of the inertia subgroup, see \cite[Lem.~4.12]{Loughran2025mallesconjecturebrauergroups}. This shows that the notation is compatible with \cite[Def.~2.3.1]{Landesman2025Homological} up to twists.
\end{remark}
\begin{definition}\label{def:map_to_configuration_space}
    There is a canonical map $\Hur^{G, C, \varphi}_{\underline{n}} \to \Conf_{\mathcal{C}, \underline{n}}$ which sends the tuple $(D, i, w, f, \eta)$ to its ramification type $\rho_{G, f}: D \to \mathcal{C}$.
\end{definition}
    
\subsection{Hurwitz stacks}
The actual objects we want to count are $G$-covers up to isomorphism, which are parameterised by Hurwitz stacks.
\begin{definition}\label{def:Hurwitz_stack}
    For $n \in \N$ define $\HurStack^{G}_{B, n}$ as the stack which sends $T \to B$ the groupoid of triples
    \[
       (D \to T, i: D \to \A^1_T, f: X \to \Proj_T^1).
    \]
    such that $D \to T$ is finite \'etale of degree $n$, $i$ is a closed immersion and $f$ is a $G$-cover whose ramification divisor is the image of $i$ and potentially $\infty$. A morphism $(D, i, w, f) \to (D', i', w', f')$ in this groupoid consists of an isomorphism $D \cong D'$, an equality $w = w'$ and an isomorphism of $G$-covers $f \to f'$ making all the obvious diagrams commute.
\end{definition}
It follows from \Cref{prop:G_covers_curves} that there exists a unique $w$ such that $f$ extends to a map $\tilde{f}: X \to \mathcal{P}_T^w$ which is a torsor over $\infty$. Pulling back $f$ along $\tilde{\infty}_T$ then defines a $G$-torsor. This construction is functorial in $T$ and thus defines a map $\HurStack^{G,}_{B, n} \to BG$. The following lemma is then immediate from the definitions.
\begin{lemma}\label{lem:Hurwitz_stack_versus_spaces}
    The scheme $\Hur^{G, \varphi}_{B, n} \to B$ is the fiber of the map $\HurStack^{G}_{B, n} \to BG$ at $\varphi \in BG(B)$.
\end{lemma}
\begin{remark}
    If $\varphi = e_{BG}$, this corresponds to the fact that the Hurwitz stack is $[\Hur^{G}_{B, n}/G]$, c.f.~\cite[Rem.~2.4.6]{ellenberg2025homologicalstabilitygeneralizedhurwitz}. This is also the way the Hurwitz stack is defined in  \cite{ellenberg2025homologicalstabilitygeneralizedhurwitz, Landesman2025Homological}.
\end{remark}
There is an even finer map $\HurStack^{G}_{B, n} \to [G(-1)/G]$.
\begin{construction}
    The image of the map $\tilde{\infty}_T$ is a $\mu_w$-gerbe for some $w \in \Z_{> 0}$. It is equipped with a point and thus canonically isomorphic to $(B \mu_w)_T$. Pulling back along $f$ defines a representable map $(B \mu_w)_T \to (BG)_T$. We have thus constructed a map $\HurStack^{G}_{B, n} \to I_{\mu} BG$ where $I_{\mu} BG$ is the cyclotomic inertia stack of $BG$, see \cite[Def.~4.5]{Loughran2025mallesconjecturebrauergroups}. We further compose this with the isomorphism $I_{\mu} BG \cong [G(-1)/G]$ from \cite[Prop.~4.6]{Loughran2025mallesconjecturebrauergroups}.
\end{construction}
We recall that points on the stack $[G(-1)/G]$ have the following concrete description \cite[Lemma~4.4]{Loughran2025mallesconjecturebrauergroups}.
\begin{lemma}\label{lem:description_G(-1)/G}
    For any scheme $T$ we have that $[G(-1)/G](T)$ is the groupoid of pairs $(\varphi, \gamma)$. Here, $\varphi \in BG(T)$ is the image of the projection $[G(-1)/G] \to BG$ and $\gamma \in G_{\varphi}(-1)(T)$. A morphism $(\varphi, \gamma) \to (\varphi', \gamma')$ consists of an isomorphism $\varphi \cong \varphi'$ in $BG(T)$ such that $\gamma$ maps to $\gamma'$ under the induced isomorphism $G_{\varphi}(-1) \cong G_{\varphi'}(-1)$.
\end{lemma}
\begin{definition}\label{def:Huritz_space_gamma}
    For $(\varphi, \gamma) \in [G(-1)/G](B)$ we let $\Hur^{G, (\varphi, \gamma)}_{B, n} \subset \Hur^{G, \varphi}_{B, n}$ be the open and closed substack which is the fiber of the map $\HurStack^{G}_{B, n} \to [G(-1)/G]$ at $(\varphi, \gamma) \in [G(-1)/G](B)$. 

    For any conjugacy-invariant subscheme $C \subset G(-1) \setminus \{1\}$ we define $\Hur^{G, C, (\varphi, \gamma)}_{B, \underline{n}} := \Hur^{G, (\varphi, \gamma)}_{B} \cap \Hur^{G, C, \varphi}_{B, \underline{n}} \subset \Hur^{G, \varphi}_{B}$ and $\CHur^{G, C, (\varphi, \gamma)}_{B, \underline{n}} := \Hur^{G, (\varphi, \gamma)}_{B} \cap \CHur^{G, C, \varphi}_{B, \underline{n}} \subset \Hur^{G, \varphi}_{B}$. The definition implies that all squares in the following diagram are cartesian
    \[\begin{tikzcd}
	{\Hur^{G, C, (\varphi, \gamma)}_{B, \underline{n}}} & {\Hur^{G, C, \varphi}_{B, \underline{n}}} & { \Hur^{G, \varphi}_{B, n}} \\
	B & {G_{\varphi}(-1)} & {[G(-1)/G]}
	\arrow[from=1-1, to=1-2]
	\arrow[from=1-1, to=2-1]
	\arrow[from=1-2, to=1-3]
	\arrow[from=1-2, to=2-2]
	\arrow[from=1-3, to=2-3]
	\arrow["\gamma", from=2-1, to=2-2]
	\arrow["{\gamma \to (\varphi, \gamma)}", from=2-2, to=2-3]
\end{tikzcd}\]
\end{definition}
Note that $\Hur^{G, C, (\varphi, \gamma)}_{B, n} \subset \Hur^{G, C, \varphi}_{B, \underline{n}}$ is an open and closed substack by \Cref{lem:Hurwitz_stack_versus_spaces} and because the map $B \xrightarrow{\gamma} G_{\varphi}(-1)$ is finite \'etale.
\subsection{Functoriality}
We will need some functoriality of Hurwitz schemes in $G$. Functoriality for isomorphisms is clear, but for general maps it is slightly subtler. Let us start with the following observation.
\begin{lemma}
    For any $T \to B$ the groupoid $\HurStack^{G}_B(T) := \bigsqcup_{n} \HurStack^{G}_{B,n}(T)$ is equivalent to the groupoid of $G$-covers of $\Proj^{1}_T$.
\end{lemma}
\begin{proof}
    This will follow from \Cref{prop:G_covers_curves}. Given that proposition it remains to show that for every $G$-cover $X \to \Proj^{1}_T$ that we can find a unique up to unique isomorphism pair $(D,i)$ as in the definition. This is true because such a pair is determined uniquely up to unique isomorphism by its image which we get from \Cref{prop:G_covers_curves}.
\end{proof}
The groupoid of $G$-covers is clearly functorial in $G$. Moreover, the natural map $\HurStack^{G}_B \to BG$ of \Cref{lem:Hurwitz_stack_versus_spaces} is also functorial in $G$ so this implies that $\Hur^{G, \varphi}$ is functorial in $G$ by that lemma.
\begin{lemma}\label{lem:functoriality}
    Let $f:G \to H$ be a map of finite \'etale group schemes. Let $\varphi \in BG(B)$ and $f_*: BG \to BH$ the map induced by $B$. Let $C_G \subset G(-1) \setminus \{1\}, C_H \subset H(-1) \setminus \{1\}$ be conjugacy invariant subsets and assume that $f(C_G) \subset C_H \cup \{1\}$
    
    The natural map $\Hur^{G, \varphi} \to \Hur^{H, f_*(\varphi)}$ maps $\Hur^{G, C_G, \varphi}$ into $\Hur^{H, C_G, f_*(\varphi)}$.
\end{lemma}
\begin{proof}
    After unfolding definitions and using the equivalences in \Cref{prop:G_covers_curves} this is reduced to the following claim. 

    Let $T \to B$ be a map and $f: X \to \mathbb{P}^1_T$ a $G$-cover whose ramification divisor is contained in $D$. Let $f_H: X_H \to \mathbb{P}^1_T$ be the induced $H$-cover, which is uniquely determined by the existence of a $G$-equivariant morphism $X \to X_H$. The claim is that the ramification type $\rho_H(f_H): D \to \mathcal{C}_H$ is equal to the composition of $\rho_G(f): D \to \mathcal{C}_G$ and the map $\mathcal{C}_G \to \mathcal{C}_H$.

    This claim follows immediately from the construction of the ramification type in \Cref{con:inertia_type}.
\end{proof}
\section{Components of Hurwitz schemes}
A crucial tool for understanding the components of Hurwitz spaces is the lifting invariant of Ellenberg--Venkatesh--Westerland \cite[\S8.5]{ellenberg2013homologicalstabilityhurwitzspaces}, the properties of which were studied in detail by Wood \cite{Wood2021Lifting}. This lifting invariant is only defined in \cite{ellenberg2013homologicalstabilityhurwitzspaces, Wood2021Lifting} for constant group schemes, we will show that it extends to general tame finite \'etale group schemes $G$.

For this section we will fix a base scheme $B$, a tame finite \'etale scheme $G$ over $B$ and a conjugacy-invariant subscheme $C \subset G(-1) \setminus \{1\}$ which generates $G$. Let $\mathcal{C} := C/G$ be the scheme of conjugacy classes in $C$.

For technical reasons it will be convenient to fix a set of primes $\mathbf{P}$ such that if $p \in \mathbf{P}$ then $p$ is invertible on $B$ and for all geometric points $b \in B$ all primes dividing the order of $G_b$ are contained in $\mathbf{P}$. If $B$ is connected then we may take $\mathbf{P}$ the set of primes dividing the order of $G_b$, which is independent of $b$.

We will need the following profinite group \'etale schemes 
$\hat{\Z}_{\mathbf{P}}:= \varprojlim_{n} (\Z/ n \Z), \hat{\Z}_{\mathbf{P}}(1) := \varprojlim_{n} \mu_n$ where the limits are over those $n$ such that if $p \mid n$ is prime then $p \in \mathbf{P}$.

We will also consider the profinite \'etale subschemes $\hat{\Z}_{\mathbf{P}}^{\times} \subset \hat{\Z}_{\mathbf{P}}, \hat{\Z}_{\mathbf{P}}^{\times}(1) \subset \hat{\Z}_{\mathbf{P}}(1)$ consisting of topological generators. An element of $\hat{\Z}_{\mathbf{P}}^{\times}(1)$ is thus a compatible system of primitive roots of unity. Note that $\hat{\Z}_{\mathbf{P}}^{\times}(1)$ is a $\hat{\Z}_{\mathbf{P}}^{\times}$-torsor.
\begin{definition}
    Let $E = \varprojlim_i E_i$ be a profinite \'etale group scheme over $B$ with $E_i$ finite \'etale. Then $E$ is \emph{$\mathbf{P}$-primary} if for all geometric points $b \in B$ and $i$ the finite group $(E_i)_b$ is $\mathbf{P}$-primary in the sense that if a prime $p$ divides the order of $(E_i)_b$ then $p \in \mathbf{P}$.

    If $E$ is $\mathbf{P}$-primary then we put $E(-1) := \Hom_{\hat{\Z}_{\mathbf{P}}^{\times}}(\hat{\Z}_{\mathbf{P}}^{\times}(1), E)$ where $\lambda \in \hat{\Z}_{\mathbf{P}}^{\times}$ acts on $E$ by $e \to e^{\lambda}$, it is isomorphic to $\varprojlim_i E_i(-1).$
\end{definition}
\subsection{Marked central extensions}
The first goal will be to construct the ind-finite \'etale scheme where the lifting invariant is valued. In \cite{ellenberg2013homologicalstabilityhurwitzspaces, Wood2021Lifting} this is done over fields by defining a group in terms of generators and relations and by writing down the Galois action in terms of these generators (but note that the Galois action does not preserve the group structure). We will take a different approach by defining it via a universal property. We will then also define the lifting invariant itself via this universal property. Our reasons for this approach are as follows: it makes functoriality clearer, it allows for a definition of the lifting invariant without any choices and it clarifies the relation between the lifting invariant and the unramified Brauer group of $BG$.

We will need a minor modification of \cite[Def.~6.11]{Loughran2025mallesconjecturebrauergroups}. 
\begin{definition}\label{def:marked_central_extension}
    A $\mathbf{P}$-primary $\mathcal{C}$-\emph{marked central extension} is $1 \to A \to E \to G \to 1$ is a central extension of $\mathbf{P}$-primary profinite \'etale group schemes over $B$ equipped with a finite subscheme $M \subset E(-1)$. The subscheme $M$ is a \emph{marking}, which means that it is invariant under conjugation by elements of $E$ and that the map $E \to G$ induces an isomorphism $M\cong C$.
    
    To ease notation we will drop $\mathbf{P}$ and $\mathcal{C}$ if they are clear from the context.

    A \emph{morphism} of $\mathcal{C}$-marked central extensions is a map of central extensions which preserves the marking $M$. 
\end{definition}

\begin{definition}\label{def:universal_marked_central_extension}
    We define $\hat{U}_{B, \mathbf{P}}(G,  \mathcal{C})(1) \to G$ as the universal $\mathbf{P}$-primary $\mathcal{C}$-marked central extension, i.e. the initial object in the category of marked central extensions.

    The kernel of $\hat{U}_{\mathbf{P}}(G,  \mathcal{C})(1) \to G$ will be denoted $\HH_{2,\mathrm{orb}}^{\mathcal{C}}(G, \hat{\Z}_{\mathbf{P}}(1))$.
\end{definition}
We have to show that this central extension exists. One approach is to define it as the limit of all finite marked central extensions. We will take a different approach by describing it \'etale locally via generators and relations as this description will be useful in what follows.
\begin{proof}[Proof that $\hat{U}_{\mathbf{P}}(G,  \mathcal{C})(1)$ exists]
    All data in the definition of a marked central extension satisfies \'etale descent and all properties can be checked \'etale locally. As initial objects are unique up to unique isomorphism it suffices by descent to construct it for $B$ the spectrum of a strictly henselian ring. This is done in the following construction.
\end{proof}

\begin{construction}\label{con:universal_marked_central_extension}
    Assume that $B = \Spec R$ for $R$ a strictly henselian ring. In this case an \'etale group scheme over $R$ is constant so may be identified with its group of $R$-elements, which we will do. 
    
    Define $\hat{U}_{R, \mathbf{P}}(G,C)(1)$ as the $\mathbf{P}$-primary profinite group generated by symbols $[\gamma]^{\zeta}$ for $\gamma \in C(R)$ and $\zeta \in \hat{\Z}_{\mathbf{P}}^{\times}(1)(R)$ subject to the following relations: 
    \begin{align}\label{eq:relations_universal_central_extension}
            &[\gamma]^{\lambda \zeta} = ([\gamma]^\zeta)^{\lambda}
         &[\xi]^{\eta} [\gamma]^{\zeta} [\xi]^{-\eta} = [\xi(\eta) \gamma \xi(\eta)^{-1}]^{\zeta}
    \end{align}
    for all $\xi, \gamma \in C(R)$, $\zeta, \eta \in \hat{\Z}(1)(R)$ and $\lambda \in \hat{\Z}^{\times}(R)$.

    There is a natural map of groups $\hat{U}_{R, \mathbf{P}}(G,C)(1) \to G(R): [\gamma]^{\zeta} \to \gamma(\zeta).$ This is a central extension by \cite[Lem.~2.1]{Wood2021Lifting}.

    We equip $\hat{U}_{R, \mathbf{P}}(G,C)(1)$ with the marking $M(G, C) \subset \hat{U}_{R, \mathbf{P}}(G,C)$ consisting of the functions $\zeta \to [\gamma]^{\zeta}$. This is conjugation invariant by the second relation of \eqref{eq:relations_universal_central_extension}. The map $M(G, C) \to C(R)$ induced by $\hat{U}_{R, \mathbf{P}}(G,C)(1) \to G(R)$ is a bijection by the first relation of \eqref{eq:relations_universal_central_extension}. This shows that $M$ is a marking.

    If $1 \to A \to E \to G$ is a $\mathbf{P}$-primary $\mathcal{C}$-marked central extension with marking $M$ and induced isomorphism $\lambda: C \cong M$ then any map $\hat{U}_{R, \mathbf{P}}(G,C)(1) \to E(R)$ which preserves the marking has to send $[\gamma]^{\zeta}$ to $\lambda(\gamma)(\zeta)$ and is thus unique. This formula always defines a map of groups as the relations \eqref{eq:relations_universal_central_extension} map to the identity by the definition of a marking. We conclude that the constant group scheme $\hat{U}_{R, \mathbf{P}}(G,C)(1) \to G$ is an initial object in the category of marked central extensions.
\end{construction}
\begin{remark} \hfill
    \begin{enumerate}
        \item It follows from the above that the functor of points of $\hat{U}_{B, \mathbf{P}}(G,C)$ is given by sheafifying \Cref{con:universal_marked_central_extension} for general rings $R$.
        \item In \cite[\S7.4]{ellenberg2013homologicalstabilityhurwitzspaces} a definition of marked central extension is given for a finite group $G$ and a conjugacy-invariant subset $C \subset G \setminus \{1\}$. If we choose a compatible family of primite roots of unity to identify $\hat{\Z}_{\mathbf{P}} \cong \hat{\Z}_{\mathbf{P}}(1)$ then it agrees with our definition except that the kernels of their central extensions are discrete groups. A discrete version of \Cref{con:universal_marked_central_extension} is then given in \cite[\S7.5]{ellenberg2013homologicalstabilityhurwitzspaces} as a construction for an initial object in their category of marked central extensions.
        \item In \cite[\S6.4]{Loughran2025mallesconjecturebrauergroups} marked central extensions are introduced as they correspond to elements of the unramified Brauer group of $BG$. The universal marked central extension over a field $k$ is then introduced as a tool to compute the unramified Brauer group, although note that the notation $\hat{U}(G, \mathcal{C})$ is used for what we denote $\hat{U}_{k, \mathbf{P}}(G, C)(1)$ with $\mathbf{P}$ all primes coprime to the characteristic of $k$.
    \end{enumerate}
\end{remark}
We will also consider the following related schemes.
\begin{definition}
    We consider the profinite scheme $\hat{U}_{B, \mathbf{P}}(G,C) := \hat{U}_{B, \mathbf{P}}(G,C)(1)(-1)$. We will also consider $\HH_{2,\mathrm{orb}}^{\mathcal{C}}(G, \hat{\Z}_{\mathbf{P}}) := \HH_{2,\mathrm{orb}}^{\mathcal{C}}(G, \hat{\Z}_{\mathbf{P}}(1))(-1)$ for the fiber of the map $\hat{U}_{B, \mathbf{P}}(G,C) \to G(-1)$ at $1$. Note that $\HH_{2,\mathrm{orb}}^{\mathcal{C}}(G, \hat{\Z}_{\mathbf{P}})$ is a commutative profinite \'etale group scheme as $\HH_{2,\mathrm{orb}}^{\mathcal{C}}(G, \hat{\Z}_{\mathbf{P}}(1))$ is commutative. But $\hat{U}_{B, \mathbf{P}}(G,C) $ will in general not be a group scheme if $G$ is non-commutative.
\end{definition}
\begin{definition}\label{def:lifting_invariant_multiplicity}
    Consider the split central extension $1 \to \Hom(\mathcal{C}, \hat{\Z}_{\mathbf{P}}(1)) \to \Hom(\mathcal{C}, \hat{\Z}_{\mathbf{P}}(1)) \times G \to G \to 1$ equipped with the marking $M \subset \Hom(\mathcal{C}, \hat{\Z}_{\mathbf{P}}) \times G(-1)$ consisting of all elements of the form $(e_{\gamma}, \gamma)$, where $\gamma \in C$ and $e_{\gamma}$ is the indicator function of the conjugacy class containing $\gamma$. This induces a map $\mathfrak{R}: \hat{U}_{B, \mathbf{P}}(G,C)(1) \to \Hom(\mathcal{C}, \hat{\Z}_{\mathbf{P}}(1))$. 
    
    We define $U_{B, \mathbf{P}}(G,C), \HH_{2,\mathrm{orb}}^{\mathcal{C}}(G, \Z)$ as the fiber products making the following diagram cartesian
    \[\begin{tikzcd}
	{\HH_{2,\mathrm{orb}}^{\mathcal{C}}(G, \Z)} & {U_{B}(G,C)} & {\Hom(\mathcal{C}, \Z)} \\
	{\HH_{2,\mathrm{orb}}^{\mathcal{C}}(G, \hat{\Z}_{\mathbf{P}})} & {\hat{U}_{B, \mathbf{P}}(G,C)} & {\Hom(\mathcal{C}, \hat{\Z}_{\mathbf{P}})}
	\arrow[from=1-1, to=1-2]
	\arrow[from=1-1, to=2-1]
	\arrow[from=1-2, to=1-3]
	\arrow[from=1-2, to=2-2]
	\arrow[from=1-3, to=2-3]
	\arrow[from=2-1, to=2-2]
	\arrow[from=2-2, to=2-3]
    \end{tikzcd}\]

    We define $\HH^2(G, \mathcal{C}) := \ker\left(\mathfrak{R}: \HH_{2,\mathrm{orb}}^{\mathcal{C}}(G, \Z) \to \Hom(\mathcal{C}, \Z)\right)$.
\end{definition}
    
The following lemma shows that this construction is well-behaved.
    \begin{lemma}\label{lem: U(G, C) ind-finite \'etale}
        The maps $\hat{U}_{B, \mathbf{P}}(G,C)(1) \to \Hom(\mathcal{C}, \hat{\Z}_{\mathbf{P}}(1))$ and $U_{B}(G, C) \to \Hom(\mathcal{C}, \Z)$ are finite. In particular $U_{B}(G, C)$ and $\HH_{2,\mathrm{orb}}^{\mathcal{C}}(G, \Z)$ are ind-finite \'etale schemes. Moreover, they are independent of the choice of $\mathbf{P}$.
    \end{lemma}
    \begin{proof}
        Finiteness is compatible with base change so it suffices to prove finiteness of the first map. By descent we may reduce to the case that $B$ is the spectrum of a strictly henselian ring $R$.
        
        We may then work with the underlying abstract group of $\hat{U}_{R, \mathbf{P}}(G, C)$ in \Cref{con:universal_marked_central_extension}. This is the same as the $\mathbf{P}$-primary profinite completion of the group denoted $U(G(R), C(R))$ in \cite[\S2]{Wood2021Lifting} as they have the same generators and relations. The finiteness statement then follows from \cite[Thm.~2.5]{Wood2021Lifting}. The above reasoning also shows that $U_{R}(G, C)(R)$ is isomorphic to $U(G(R), C(R))$ and is thus independent of $\mathbf{P}$.
    \end{proof}
    \begin{remark}\label{rem:universal_marked_extension} \hfill
        \begin{enumerate}
            \item If $\Omega$ is an algebraically closed field and $b \in B(\Omega)$ is a geometric point then the above lemma shows that $U_{B}(G, C)_b$ is the group $U(G_b, C_b)$ of \cite{Wood2021Lifting}, or equivalently the group $\tilde{G_b}$ of \cite{ellenberg2013homologicalstabilityhurwitzspaces}. If $G$ is constant then the action of the \'etale fundamental group $\pi_1^{\text{\'et}}(B, b)$ on $U_{B}(G, C)_b$ is identified with the action given by composing the cyclotomic character $\pi_1^{\text{\'et}}(B) \to \hat{\Z}^{\times}$ and the inverse of the discrete action of \cite[p.~7]{Wood2021Lifting} \cite[\S8.1.7]{ellenberg2013homologicalstabilityhurwitzspaces}.
            
            Indeed, the definitions show that $\hat{U}_{B, \mathbf{P}}(G, C)(1)_b$ is isomorphic to the maximal $\mathbf{P}$-primary quotient of $\hat{U}(G_b, C_b), \tilde{G_b}^{\wedge}$. The induced action of $\pi_1^{\text{\'et}}(B, b)$ is by \Cref{con:universal_marked_central_extension} equal to the composition of the cyclotomic character and the action denoted $\cdot$ in \cite[p.~7]{Wood2021Lifting}, \cite[\S8.1.6]{ellenberg2013homologicalstabilityhurwitzspaces}.
            \item In particular, for $B = \Spec \F_q$ we have $\HH_2(G, \mathcal{C})(\bar{\F}_q) = \HH_2(G, C(\bar{\F}_q))(-1)$ where the later group is defined in \cite[Definition]{Wood2021Lifting}, \cite[Def.~7.3]{ellenberg2013homologicalstabilityhurwitzspaces}.
        \end{enumerate}
    \end{remark}
    \begin{lemma}\label{lem:M_is_abelianization}
        The map $\mathfrak{R}: \hat{U}_{B, \mathbf{P}}(G, C)(1) \to \Hom(\mathcal{C}, \hat{\Z}_{\mathbf{P}}(1))$ is the abelianization of $\hat{U}_{B, \mathbf{P}}(G, C)(1)$.
    \end{lemma}
    \begin{proof}
        As in the above we reduce to the group of \Cref{con:universal_marked_central_extension}, where this is \cite[Lemma~2.4]{Wood2021Lifting}.
    \end{proof}
The main importance of the scheme $U_B(G, C)$ is because of the following theorem.
\begin{theorem}\label{thm:lifting_invariant}
    There exists a map $\mathfrak{z}: \pi_0(\Hur^{G, C, \varphi}/B) \to U_B(G_{\varphi}, C)$ called the \emph{lifting invariant} such that the composition with $\mathfrak{R}: U(G_{\varphi}, C) \to \Hom(\mathcal{C}, \Z)$ is the map $\mathfrak{R}$ of \Cref{def:Multiplicity_Hurwitz}. Moreover, there exists an $N > 0$ such that the lifting invariant is an isomorphism when restricted to $\pi_0(\CHur^{G, C, \varphi}/B)$ over $\Hom(\mathcal{C}, \Z_{> N})$.
\end{theorem}
\begin{remark}
    Certain versions of this theorem are known. When $B = \Spec \C$ this is \cite[Thm.~3.1]{Wood2021Lifting}. Under the assumptions that $C$ is closed under taking invertible powers, $G$ is constant and $\varphi = e_{BG}$ is is known over general base schemes. For the components of $\Hur^{G, C}$ corresponding to $G$-covers unramified at $\infty$ the map is constructed over fields in \cite[\S5]{Wood2021Lifting} and it is shown in \cite[Thm~6.1]{Wood2021Lifting} that the definition can be extended to schemes. The construction is generalized to the other components of $\Hur^{G, C}$ in \cite[\S2]{liu2025imaginarycasenonabeliancohenlenstra}.

    The assumptions that $G$ is constant and $\varphi = e_{BG}$ are not crucial and can be removed with descent. The assumption that $C$ is closed under invertible powers cannot be removed with formal nonsense.

    We need a version which allows arbitrary subschemes $C \subset G(-1) \setminus \{1\}$. If $G$ is constant and $B = \Spec \F_q$ then this means that $C(\bar{\F}_q)$ is closed under powers of $q$.
    
    To remove this assumption we need to extend the construction of the lifting invariant. we will take a different approach than usual to this construction. We will directly define the lifting invariants when $B$ is a scheme instead of first defining it for the spectrum of a field. We will also use the universal property of $\hat{U}(G, C)(1)$ to define the lifting invariant instead of via well-chosen families of generators of inertia. This allows us to give a definition without any choices which makes functoriality, and thus compatibility with Galois actions, obvious.
\end{remark}

The existence of $\mathfrak{z}$ will be deduced from the following proposition.
\begin{proposition}\label{prop:marked_central_extension_same_Hurwitz spaces}
    Let $E \to G$ be a finite $\mathbf{P}$-primary $\mathcal{C}$-marked extension with marking $M \subset E(-1)$. The natural map $\Hur^{E, M}_B \to \Hur^{G, C}_B$ from \Cref{lem:functoriality} is an isomorphism.
\end{proposition}
\begin{proof}
       This is a morphism of finite \'etale schemes over $\bigsqcup_n\Conf_{B/B, n}$ by \Cref{lem:Hurwitz_scheme_finite_etale}. By working \'etale locally we may assume that $G$ is constant and $\varphi = e_{BG}$. We may then reduce by functoriality to a base $B$ such that each connected component has a $\C$-point. It then suffices to show that the map is an isomorphism when restricted to a $\C$-point of each connected component of $\bigsqcup_n\Conf_{B/B, n}$. In particular, we may assume that $B = \Spec \C$. 
       
       The fact that $M$ is a marking means that $M(\C) \to C(\C)$ is a bijection. It then follows immediately from the description of Hurwitz spaces over $\C$ in terms of braid group actions \cite[Def.~2.2.2, Ex.~2.2.3]{Landesman2025Homological} that $\Hur^{E, M}_B \to \Hur^{G, C}_B$ is an isomorphism.
\end{proof}

\begin{construction}\label{con:lifting_invariant}
    Let $E \to G$ be a finite $\mathbf{P}$-primary $\mathcal{C}$-marked extension with marking $M \subset E(-1)$. Let $T \to B$ be a map. For a component $Z \in \pi_0(\Hur^{G, C, e_{BG}}_B/B)(T)$ define $\mathfrak{z}_E(Z)  = \gamma_E^{-1} \in E(-1)(T)$ where $\gamma_E$ is the image of $Z$ under the composition
    \[
    \Hur^{G, C}_T \cong \Hur^{E, M}_T \to E(-1).
    \]
    The first isomorphism is \Cref{prop:marked_central_extension_same_Hurwitz spaces} and the second map is from \Cref{def:Huritz_space_gamma}.

    This map is functorial in $T$ and thus defines a morphism $\mathfrak{z}_E: \pi_0(\Hur^{G, C, e_{BG}}_B/B) \to E(-1)$.

    Note that $\mathfrak{z}_E(Z)$ is functorial in $E$ and can thus be extended formally to profinite $\mathbf{P}$-primary $\mathcal{C}$-marked extensions. In particular, if $E =  \hat{U}_{\mathbf{P}}(G, C)(1)$ then we will write $\mathfrak{z}: \pi_0(\Hur^{G, C}_B/B) \to \hat{U}_{\mathbf{P}}(G, C)$ for this extension.
\end{construction}
\begin{lemma}\label{lem:lifting_invariant_factors}
    For all $T \to B$, $\underline{n} \in \Hom_T(\mathcal{C}, \N)$ and components $Z \in \pi_0(\Hur^{G, C}_{B,\underline{n}}/B)(T)$ we have $\mathfrak{R}(\mathfrak{z}(Z)) = \underline{n}$.
    
    In particular, the map $\mathfrak{z}: \pi_0(\Hur^{G, C}) \to \hat{U}_{\mathbf{P}}(G, C)$ factors through $U_B(G, C)$.
\end{lemma}
\begin{proof}
    The map $\mathfrak{R}: \hat{U}_{\mathbf{P}}(G, C)(1) \to \Hom(\mathcal{C}, \hat{\Z}_{\mathbf{P}}(1))$ is by definition induced by equipping the trivial central extension $\Hom(\mathcal{C}, \hat{\Z}_{\mathbf{P}}(1)) \times G \to G$ with the marking $M$ consisting of pairs $(e_{\gamma}, \gamma)$ with $\gamma \in G(-1)$ and $e_{\gamma}$ the indicator function of the conjugacy class containing $\gamma$, see \Cref{def:lifting_invariant_multiplicity}.

    Unfolding \Cref{con:lifting_invariant} shows that the lemma is equivalent to the following statement. Let $k \in \Z_{> 1 }$ be an integer which is a product of primes in $\mathbf{P}$. For all components $Z \in \pi_0(\Hur^{\Hom(\mathcal{C}, \mu_k) \times G, M}_{T, \underline{n}})(T)$ the composition 
    \[
        \Hur^{\Hom(\mathcal{C}, \mu_k) \times G, M}_{T, \underline{n}} \to \Hom(\mathcal{C}, \Z/k \Z) \times G(-1) \to \Hom(\mathcal{C}, \Z/k \Z).
    \]
    is $-\underline{n} \pmod k$. Here we used the implicit identification $\underline{n} \in \Hom(\mathcal{C}, \N) \cong \Hom(M/G, \N)$.

    This equality can be checked \'etale locally so we may assyme that $T =  \Spec R$ where $R$ is the spectrum of a strictly henselian ring. In that case the image of this composition is the image of the monodromy at $\infty$ in $\Hom(\mathcal{C}, \Z/k \Z)$. This is equal to the sum of the inverses of the image of the monodromy at all points in $\A^1_T$. This is $-\sum_{c \in \mathcal{C}} n_c e_c = -\underline{n}$.
\end{proof}
\begin{definition}\label{def:lifting_invariant}
    We call the map $\mathfrak{z}:\pi_0(\Hur^{G, C}_B/B) \to U_B(G, C)$ induced by \Cref{lem:lifting_invariant_factors} the \emph{lifting invariant}.
\end{definition}

To prove \Cref{thm:lifting_invariant} we need to compare the map $\mathfrak{z}$ with the lifting invariant over the complex numbers of \cite{Wood2021Lifting}. We will now recall its definition from \cite[\S 5]{Wood2021Lifting}. There is a small difference in that we use the fundamental group of the root stack $\mathcal{P}^w$ as it appears in our definition of Hurwitz schemes.
\begin{construction}\label{con:lifting_invariant_complex}
    Let $G$ be a finite group and $C \subset G \setminus \{1\}$ a conjugacy invariant subset which generates $G$. Let $U(G, C)$ be the group defined in \cite[\S2]{Wood2021Lifting}. To be explicit, it is the group generated by symbols $[g]$ for $g \in C$ and the following relations for all $g,h \in C$
    \begin{equation}\label{eq:relations_U(G, C)_over_C}
        [h g h^{-1}] = [h] \cdot [g] \cdot [h]^{-1}.
    \end{equation}

    Let $D = \{P_1, \dots, P_n\} \in \Conf_{n}(\C)$. Fix $w \in \N$ and write $U := \mathcal{P}_{\C}^w\setminus D$.

    Consider the fundamental group $\pi_1^{\text{\'et}}(U, \tilde{\infty})$, where $\tilde{\infty}$ is defined in \eqref{eq:infinity_tilde_charts}. The inclusion of $B \mu_w$ at infinity induces an element $\gamma_{\infty}: \mu_w = \pi_1^{\text{\'et}}(B \mu_w, \tilde{\infty}) \to \pi_1^{\text{\'et}}(U, \tilde{\infty})$. For each point $P_i$ there exists a conjugacy class of morphisms $\gamma_i: \hat{\Z}(1) \to \pi_1^{\text{\'et}}(U, \tilde{\infty})$ which generates the inertia subgroup of $\pi_1^{\text{\'et}}(U, \tilde{\infty})$ at $P_i$, see \cite[\S5.1]{Wood2021Lifting}.

    The fundamental group of $U$ was computed in \cite[Prop.~5.6]{Behrend2006Uniformization}. In particular, for each point $P_i$ there exists a $\zeta_i \in \hat{\Z}^{\times}(1)$ and a choice of $\gamma_i$ such that $\pi_1^{\text{\'et}}(B \mu_w, \tilde{\infty})$ is the group generated by the $\gamma_i(\zeta_i)$ modulo the relations $\gamma_1(\zeta_1) \cdots \gamma_n(\zeta_n) \gamma_{\infty}(\zeta_{\infty}) = 1$ for some $\zeta_{\infty} \in \mu_w$. If $w = 1$ then \cite[\S5.2]{Wood2021Lifting} shows that there exists $\zeta \in \hat{\Z}^{\times}(1)$ such that $\zeta = \zeta_1 = \cdots = \zeta_n = \zeta_{\infty}$ and the argument directly generalizes to general $w$.

    Let $C_{\zeta} \subset G(-1)$ be the image of $C$ under the isomorphism $G \cong G(-1)$ induced by $\zeta$. The lifting invariant in \cite[Thm.~5.2]{Wood2021Lifting} or \cite[Thm.~5.2]{liu2025imaginarycasenonabeliancohenlenstra} send an element of $\Hur^{G, C_{\zeta}}$ corresponding to a map $\psi: \pi_1^{\text{\'et}}(U, \tilde{\infty}) \to G$ to the product
    \begin{equation}\label{eq:lifting_invariant_complex}
         [\psi(\gamma_1(\zeta))] \cdots [\psi(\gamma_n(\zeta))] \in U(G, C). 
    \end{equation}

    Note that the references work over general fields and assume that $C$ is closed under invertible powers, but this assumption is not used when the base field is $\C$.
\end{construction}
\begin{lemma}\label{lem:lifting_compatible}
    If $B = \Spec \C$ then \Cref{def:lifting_invariant}, after a choice of isomorphism $\hat{\Z} \cong \hat{\Z}(1)$, agrees with the map to $U(G(\C), C(\C))$ from \Cref{con:lifting_invariant_complex} and thus the lifting invariant of \cite{Wood2021Lifting}.
\end{lemma}
\begin{proof}
    Let $w \in \N$, $D \subset \A^1_{\C}$ be as in \Cref{con:lifting_invariant_complex} and consider a point $\Hur^{G, C}_{\C}$ corresponding to a $G$-torsor $f: X  \to \mathcal{P}^w_{\C} \setminus D$ equipped with a lift of $\tilde{\infty}$ to $X$.
    Let $\psi: \pi_1^{\text{\'et}}(\mathcal{P}^w_{\C} \setminus D, \tilde{\infty}) \to G$ be the morphism corresponding to $f$. 

    Let $E \to G$ be a finite $\mathcal{C}$-marked extension with marking $M \subset E(-1)$. Denote the induced isomorphism $\lambda: C \cong M$. Consider the image of $f$ under the inverse of the isomorphism of \Cref{prop:marked_central_extension_same_Hurwitz spaces}. Using the same reasoning as above this corresponds to a $w  \mid w'$ and a morphism $\psi_E: \pi_1^{\text{\'et}}(\mathcal{P}^{w'}_{\C} \setminus D, \tilde{\infty}) \to E$ which lifts $\psi$. The element $\gamma_{\infty} \in  \pi_1^{\text{\'et}}(\mathcal{P}^{w'}_{\C} \setminus D, \tilde{\infty})(-1)$ maps to $\gamma_{\infty} \in \pi_1^{\text{\'et}}(\mathcal{P}^{w'}_{\C} \setminus D, \tilde{\infty})(-1)$ by construction so we may identify them.

    The assumption on the inertia type of the covers ensures that $\psi_E \circ \gamma_i \in C \subset G(-1)$ for all $i = 1, \dots n$.

    Let $C(\C)^{\zeta} \subset G$ be the image of $C(\C)$ under the isomorphism $G(-1) \cong G: \gamma \to \gamma(\zeta)$. There is a natural map $U(G(\C), C(\C)^{\zeta}) \to \hat{U}(G, C)(1)$ which sends $[\gamma(\zeta)]$ to $[\gamma]^{\zeta}$ for all $\gamma \in G(-1)$. 

    We thus have to show that $(\gamma_{\infty} \circ \psi_E)(\zeta)^{-1}$ agrees with the image of \eqref{eq:lifting_invariant_complex} under the composition $U(G(\C), C(\C)^{\zeta}) \to \hat{U}(G, C)(1) \to E$. The last map is induced by the marking so this composition is
    \[
        [\psi(\gamma_1(\zeta))] \cdots [\psi(\gamma_n(\zeta))] \to [\psi \circ \gamma_1]^{\zeta} \cdots [\psi \circ \gamma_n]^{\zeta} \to \lambda(\psi \circ \gamma_1)(\zeta) \cdots\lambda(\psi \circ \gamma_n)(\zeta).
    \]

    The assumption that $\psi_E$ corresponds to a point of $\Hur^{E, M}$ implies that $\psi_E \circ \gamma_i \in M \subset E(-1)$. Moreover, the image of $\psi_E \circ \gamma_i$ in $G(-1)$ is $\psi \circ \gamma_i$ so the assumption that the induced map $\lambda^{-1}: M \to C$ is an isomorphism implies that $\lambda(\psi \circ \gamma_i) = \psi_E \circ \gamma_i$. We are then done by the computation
    \[
    \lambda(\psi \circ \gamma_1)(\zeta) \cdots\lambda(\psi \circ \gamma_n)(\zeta) = \psi_E(\gamma_1(\zeta)) \cdots  \psi_E(\gamma_n(\zeta)) = \psi_E(\gamma_{\infty}(\zeta)^{-1}). 
    \]
\end{proof}
\begin{proof}[Proof of \Cref{thm:lifting_invariant}]
        We may assume that $\varphi$ is trivial by \Cref{cons:assume_varphi_trivial}. The desired map is then given by \Cref{def:lifting_invariant}. We may work \'etale locally to assume that $G$ is a constant group and $\mu_{|G|}$ is a constant finite \'etale scheme. This then reduces to the universal base $B = \Spec \Z[\mu_{|G|}, \frac{1}{|G|}]$. The lifting invariant is then a map of ind-finite \'etale schemes. Whether it is an isomorphism can thus be checked over a single geometric point $\Spec \C$ in each component. This case follows from \Cref{lem:lifting_compatible} and \cite[Thm.~3.1]{Wood2021Lifting}.
    \end{proof}
Let us also note the following.

\begin{definition}\label{def:U(G,C) n gamma}
    For $(\varphi, \gamma) \in BG(B)$ and $\underline{n} \in \Hom_B(\mathcal{C}, \N)$ we let $U(G_{\varphi}, \mathcal{C})_{\underline{n}, \gamma} \subset U(G_{\varphi}, C)$ be the subscheme of elements which map to $\underline{n}$ under $\mathfrak{R}$ and to $\gamma^{-1}$ under $G(-1)$.
\end{definition}
\begin{lemma}\label{lem:components_varphi_gamma}
    If $\underline{n} \in \Hom_B(\mathcal{C}, \N)$ then the components of $\Hur^{G, C, \varphi}_{B}$ included in $\Hur^{G, C, (\varphi, \gamma)}_{B, \underline{n}}$ are exactly those that map to $U(G_{\varphi}, \mathcal{C})_{\underline{n}, \gamma}$ under the lifting invariant $\mathfrak{z}$.
\end{lemma}
\begin{proof}
Let $Z \in \pi_0(\Hur^{G, C, (\varphi, \gamma)}_{B, \underline{n}})(B)$ be a components and $\mathfrak{z}(Z)$ its lifting invariant. \Cref{lem:lifting_invariant_factors} implies that $\mathfrak{R}(\mathfrak{z}(Z)) = \underline{n}$. It is immediate from the definition that the image of $\mathfrak{z}(Z)$ in $G(-1)$ is $\gamma^{-1}$.
\end{proof}
\subsection{Stable components in all directions}
We also need a description of $\pi_0(\Hur^{G, C, \varphi}/B)$ when only some of the $n_c$ go to infinity.

\begin{theorem}\label{thm:components_partial_stability}
    Let $C_{\mathrm{st}} \subseteq C \subset G(-1) \setminus \{1\}$ be conjugacy-invariant subsets. Assume that $C$ generates $G$ and that $C_{\mathrm{st}}$ generates a subgroup scheme $H \subset G$ which surjects onto $G/Z(G)$. Let $\mathcal{C}_{\mathrm{st}} \subset \mathcal{C} \subset \mathcal{C}_G^*$ be the corresponding conjugacy classes. 
    
    There exists $N > 0$ such that the lifting invariant 
    $\pi_0(\Hur^{G, C, \varphi}/B) \to U_B(G_{\varphi}, C)$ is an isomorphism over all $\underline{n} \in \Hom(\mathcal{C}, \Z_{> 0})$ with $n_c \geq N$ for all $c \in \mathcal{C}_{\mathrm{st}}$.
\end{theorem}
    This section will be devoted to the proof of this theorem.

    By working \'etale locally on $B$ we may assume that $G$ is constant, $\varphi$ is trivial and $\mu_{|G|}$ is a constant finite \'etale scheme. This case reduces by functoriality to the universal case $B = \Spec \Z[\mu_{|G|}, \frac{1}{|G|}]$. In this case it is a map of ind-finite \'etale schemes so it suffices to prove it over one geometric point in each connected component, so we may assume that $B = \Spec \C$.

    We may then choose an isomorphism $\hat{\Z} \cong \hat{\Z}(1)$ to make the identification $G \cong G(-1)$. The image of the lifting invariant may then be identified with $U(G(\C), C(\C))$ by \Cref{lem:lifting_compatible}, recall that this group was defined in \Cref{con:lifting_invariant_complex}. We will abuse notation and write $G = G(\C), H = H(\C)$ and $C = C(\C), C_{\mathrm{st}} = C_{\mathrm{st}}(\C)$ to keep the notation light. We need the following lemma to compare the lifting invariants.
\begin{lemma}\label{lem:relation_universal_marked_central_extensions}
    Let $G$ be a finite group and $C_{\mathrm{st}} \subset C \subset G \setminus \{1\}$ conjugacy invariant subsets. Assume that $C$ generates $G$ and that $C_{\mathrm{st}}$ generates a subgroup $H \subset G$ which surjects onto $G/Z(G)$.
    
    Ihe natural map $U(H, C_{\mathrm{st}}) \to U(G, C)$ surjects onto the elements $u \in U(G, C)$ such that $\mathfrak{R}(u)(c) = 0$ for all $c \in \mathcal{C} \setminus \mathcal{C}_{\mathrm{st}}$.

    Moreover, the kernel is generated by the following relations. For any $g \in C \setminus C_{\mathrm{st}}$  and sequences $h_1, \dots, h_n \in C_{\mathrm{st}}, \epsilon_1, \dots, \epsilon_n \in \{1 , -1\}$ such that $h_1 \cdots h_n \in H$ commutes with $g$ we get a relation
    \begin{equation}\label{eq:elements_quotient_different_universal_central_extension}
        [g h_1 g^{-1}]^{\varepsilon_1} \cdots [g h_n g^{-1}]^{\varepsilon_n} \sim [h_1]^{ \epsilon_1} \cdots [h_n]^{\epsilon_n}.
    \end{equation}
\end{lemma}
If $H = G$ then this can be deduced from \cite[Thm.~2.5]{Wood2021Lifting}, but we were unable to utilize that theorem for the general case and instead work with generators and relations.
\begin{proof}
    The elements \eqref{eq:elements_quotient_different_universal_central_extension} are contained in the kernel by a direct computation using the defining relations \eqref{eq:relations_U(G, C)_over_C}.

    Let $V \subset U(G, C)$ consist of all elements $u$ such that $\mathfrak{R}(u)(c) = 0$ for all $c \in C_{\mathrm{st}}$. The map $U(H, C_{\mathrm{st}}) \to U(G, C)$ clearly factors through $V$.
    
    We will prove the lemma by constructing a section from $V$ to $U(H, C_{\mathrm{st}})$ modulo the given relations. Let $v \in V$, by using the relations \eqref{eq:relations_U(G, C)_over_C} and using that $\mathfrak{R}(v)(c) = 0$ for all $c \in \mathcal{C} \setminus \mathcal{C}_{\mathrm{st}}$ we may write $v$ as a product of generators of $U(H, C_{\mathrm{st}})$ and elements of the form $[g'] [g]^{-1}$, where $g \in C \setminus C_{\mathrm{st}}$ and $g'$ is conjugate to $g$. We will construct the section by describing where each of these elements goes.

    A generator of $U(H, C_{\mathrm{st}})$ gets sent to itself, this ensures that the map will be a section. For an element of the form $[g'] \cdot [g]^{-1}$ we proceed as follows. The set $C_{\mathrm{st}} \subset G$ generates the inner automorphism group $G/Z(G)$ of $G$ by assumption. There thus exist $h_1, \dots, h_k \in C_{\mathrm{st}}$ such that $g' = h_1 \cdots h_k \cdot g \cdot h_k^{-1} \cdots h_1^{-1}$. We then send $[g'] \cdot [g]^{-1}$ to $[h_1] \cdots [h_k] \cdot [g h_k g^{-1}]^{-1} \cdots [g h_1 g^{-1}]^{-1}$.
    
    If $f_1, \dots, f_{\ell} \in C_{\mathrm{st}}$ are also such that $g' = f_{1} \cdots f_{\ell} \cdot g \cdot f_{\ell}^{-1} \cdots f_{1}^{-1}$ then \eqref{eq:relations_U(G, C)_over_C} implies that 
    \[
    [g f_{\ell} g^{-1}]^{-1} \cdots [g f_1 g^{-1}]^{-1} [g h_1 g^{-1}]\cdots [g h_k g^{-1}] \sim [f_{\ell}]^{-1} \cdots [f_{1}]^{-1} [h_1] \cdots [h_k] 
    \]
    This implies that 
    \[
    [h_1] \cdots [h_k] \cdot [g h_k g^{-1}]^{-1} \cdots [g h_1 g^{-1}]^{-1} = [f_1] \cdots [f_{\ell}] \cdot [g f_{\ell} g^{-1}]^{-1} \cdots [g f_1 g^{-1}]^{-1}
    \]
    so the image of $[g'] \cdot [g]^{-1}$ is well-defined.

    To show that the section is well-defined we have to show that it is compatible when changing the order of the generators of $V$ using the relations \eqref{eq:relations_U(G, C)_over_C}. The case when both $g,h  \in C_{\mathrm{st}}$ is trivial. There remain three cases.
    \begin{enumerate}
        \item If $h \in C_{\mathrm{st}}$ then we need that the image of $[h] [g] [g']^{-1}$ and $[h g h^{-1}] [h g' h^{-1}]^{-1} [h]$ are equal, this is immediate from the construction.
        \item If $h \in C \setminus C_{\mathrm{st}}$ and $h'$ is conjugate to $h$ then we need that the image of $[h] [h']^{-1} [g] [g']^{-1}$ is equal to the image of $[h h'^{-1} g h  h^{-1}] [h h'^{-1} g' h' h^{-1}]^{-1} [h] [h^{'}]^{-1}$. Here we use that $C_{\mathrm{st}}$ generates the inner automorphism group $G/Z(G)$ of $G$ so there exists $f_1, \dots, f_k \in C_{\mathrm{st}}$ such that conjugating by $h h^{'-1}$ is equal to conjugating by $f_1 \cdots f_k$. The desired equality is then immediate from the construction and the relations \eqref{eq:relations_U(G, C)_over_C} for $f_1, \cdots, f_k$ in $U(H, C_{\mathrm{st}})$.
        \item The last case is when $g,g',h,h' \in C \setminus C_{\mathrm{st}}$ are all conjugate. In this case we need that the image of $[g] [g']^{-1} [h] [h']^{-1}$ is equal to the image of $[g] [g^{'-1} h' h^{-1} g']^{-1} [g' hg^{'-1}] [g']^{-1}$. In this case we do the same as in the previous case, we write the conjugation by $g'$ and $h$ as conjugation by a product of elements in $C_{\mathrm{st}}$, use the construction and then use the relations \eqref{eq:relations_universal_central_extension} in $U(H, C_{\mathrm{st}})$.
    \end{enumerate}
\end{proof}

\begin{proof}[Proof of \Cref{thm:components_partial_stability}]
    We will closely follow the proof of \cite[Thm.~3.1]{Wood2021Lifting} until the end. As in loc.~cit.~ we will use following description of the components of Hurwitz spaces.

    For all $n \in \N$ let $V_n = C^n$. We will describe an element of $V_n$ as a formal product of $n$ terms of the form $(g)$ for $g \in C$. Let $V_n^G \subset V_n$ be the subset of formal products whose terms generate $G$. We let $V_n^G/B_n$ be the quotient of $V_n^G$ modulo the \emph{braiding relations} $(g) (h) \sim (g h g^{-1})(g)$ and $(g) (h) \sim (h) (h^{-1} g h)$. The notion of braiding relation comes from an action of the \emph{Braid group} $B_n$, which explains the notation.
    
    As explained in \cite[\S3]{Wood2021Lifting} there is a natural bijection between the components of $\CHur^{G, C}_n$ and $V_n^G/B_n$ such that the lifting invariant has the simple description of turning $(g)$ into $[g] \in U(G, C)$. Moreover, a formal product $v \in \sqcup_n V_n^G$ corresponds to a connected component of $\CHur^{G, C}_{\underline{n}}$ if and only if for all conjugacy classes $c \subset C$ the number of $(g)$ with $g \in c$ appearing in the formal product is $n_c$. We will denote this subset of formal products by $V_{\underline{n}}^G$ and its quotient by the braiding relations by $V_{\underline{n}}^G/B$.

    It follows from the first part of the proof of \cite[Thm.~3.1]{Wood2021Lifting} that if $g \in c \in \mathcal{C}$ then the natural map $V_{\underline{n}}^G/B \to V_{\underline{n} + \mathbf{1}_c}^G/B$ given by concatenating with $(g)$ is surjective as long as $n_c$ is sufficiently large. And similar to loc.~cit.~this implies that there exists an $N > 0$ such that the map $V_{\underline{n}}^G/B \to V_{\underline{n} + \mathbf{1}_c}^G/B$ is a bijection for all $g \in c \in \mathcal{C}_{\mathrm{st}}$ as long $n_c \geq N$ for all $c \in C_{\mathrm{st}}$.
    
    Let $\mathcal{C}_{\text{unst}} = \mathcal{C} \setminus \mathcal{C}_{\mathrm{st}}$. Consider now the stable value $M_{\underline{n}} = \varinjlim_{\underline{m} \in \Hom(\mathcal{C}_{\mathrm{st}}, \N)} V_{\underline{n}}^G/B_n \to V_{\underline{n} + \underline{m}}^G$, where the transition maps are concatenation with $(g)$ for all $g \in C_{\mathrm{st}}$. The elements of $M_{\underline{n}}$ can be thought of as formal products of symbols $(g)$ for $g \in C$ and $(h)^{-1}$ for $h \in C_{\mathrm{st}}$ modulo the relation $(h) (h)^{-1} = 1$ and the braiding relation, such that the number of times that $(g)$ appears minus the times that $(h)^{-1}$ appears for $g, h \in c \in \mathcal{C}$ is $n_c$.

    Note that left concatenation with elements of $C_{\mathrm{st}}$ induces a left action of $K_{\mathrm{st}} := \ker(U(g,C_{\mathrm{st}}) \to \Hom(\mathcal{C}_{\mathrm{st}}, \Z))$ on $M_{\underline{n}}$ which is compatible with the lifting invariant via the natural map $U(G,C_{\mathrm{st}}) \to U(G, C)$. 

    Recall that \Cref{lem:relation_universal_marked_central_extensions} says that the map $H_{\mathrm{st}} \to \ker(U(G, C) \to \Hom(\mathcal{C}, \Z))$ is surjective and that the kernel is generated by the following relations 
    For all $g \in C \setminus C_{\mathrm{st}}$ and all sequences $h_1, \dots, h_k \in C_{\mathrm{st}}$, $\varepsilon_1, \dots, \varepsilon_k \in \{1, -1\}$ such that $h := h_1^{\varepsilon_1} \cdots h_k^{\varepsilon_k}$ commutes with $g$ we have a relation
    \begin{equation}\label{eq:commutators}
     [g^{-1} h_1 g]^{\epsilon_1} \cdots [g^{-1} h_{k} g]^{\varepsilon_k} \sim [h_1]^{\varepsilon_1} \cdots [h_k]^{\varepsilon_k} 
    \end{equation}

   By the compatibility of the $K_{\mathrm{st}}$-action on $M_{\underline{n}}$ with the lifting invariant we are thus reduced to the following claim: the action of $K_{\mathrm{st}}$ on $M_{\underline{n}}$ is transitive and the stabilizers contain all elements of the form \eqref{eq:commutators}. Note that $K_{\mathrm{st}}$ is finite so that all of these facts actually already hold at some finite level of the colimit and not just in $M_{\underline{n}}$ itself. The definition of $M$ then ensures that the lifting invariant is a bijection already at $V_{\underline{n}}^G/B \cong \pi_0(\CHur^{G, C}_{\underline{n}})$.

   The transitivity follows from a standard argument. By multiplying by $(h) \cdot (h)^{-1}$ for $h \in C_{\mathrm{st}}$ and braiding you can conjugate the $(g)$ for $g \not \in C_{\mathrm{st}}$ so that the only $(g)$ remaining are such that $g$ is one of a choice of fixed representatives of the conjugacy classes in $C_{\mathrm{st}}$. Combining this with braiding all those $(g)$ to the right you can thus write every element of $M_{\underline{n}}$ as a product of symbols $(h)$ for $h \in C_{\mathrm{st}}$ times a product of symbols $(g)$ for $g \not \in C_{\mathrm{st}}$ which was fixed in advance.

   That elements of the form \eqref{eq:commutators} are contained in the stabilizers is also straightforward. It follows from $\underline{n} \in \Hom(\mathcal{C}, \N_{> 0})$ that every element of $M_{\underline{n}}$ contains a symbol $(g')$ with $g'$ conjugate to $g$. By a similar argument as before you can make every element in $M_{\underline{n}}$ of the form $(g) \cdot x$ for some $x$. If you multiply this with the $ [h_1] \cdots [h_{\ell}] $ then by braiding $(g)$ to the right and then to left you get
   \begin{align*}
       [h_1]^{\epsilon_1} \cdots [h_{k}]^{\varepsilon_k} \cdot (g) \cdot x = (g) \cdot  [g^{-1} h_1 g]^{\varepsilon_1} \cdots [g^{-1} h_{k} g]^{\varepsilon_k} \cdot x = [g^{-1} h_1 g]^{\varepsilon_1} \cdots [g^{-1} h_{k} g]^{\varepsilon_k} \cdot  (g) \cdot x.  
   \end{align*}
   Where the second equality is because $g^{-1} h_1^{\varepsilon_1} g \cdots g^{-1} h_{k}^{\varepsilon_k} g = g^{-1} h g$ commutes with $g$.
\end{proof}

We summarize the main result we will need from this section.
\begin{lemma}\label{lem:summary_hurwitz_spaces}
    Let $C_{\mathrm{st}} \subset C$ be a conjugacy-invariant subscheme which generates $G/Z(G)$ and let $\mathcal{C}_{\mathrm{st}} \subset \mathcal{C}$ be the corresponding conjugacy classes. 
    
    There exists $N > 0$ such that if $n_c \geq N$ for all $c \in \mathcal{C}_{\mathrm{st}}$ and $n_c > 0$ for $c \not \in \mathcal{C}_{\mathrm{st}}$ then the map $\pi_0(\CHur^{G, C, (\varphi, \gamma)}_{B, \underline{n}}/B) \to U(G_{\varphi}, \mathcal{C})_{\underline{n}, \gamma}$ is an isomorphism.

    Moreover, $U(G_{\varphi})_{\underline{n}, \gamma}$ is either empty or a $\HH_2(G, \mathcal{C})$-torsor.
\end{lemma}
\begin{proof}
    The statement about $N$ is a corollary of \Cref{thm:components_partial_stability}. The fact that $U(G_{\varphi})_{\underline{n}, \gamma}$ is empty or a $\HH_2(G, \mathcal{C})$-torsor is because $\HH_2(G, \mathcal{C})$ is by definition the scheme of elements in $U(G, C)$ mapping to $(0,1)$ in $\Hom(\mathcal{C}, \Z) \times G(-1)$.
\end{proof}
\subsection{Brauer group computations}\label{sec:Brauer_group}
We now specialize the objects of the previous section to the base scheme $B = \Spec \F_q$, we let $p$ be the characteristic of $\F_q$. As in the previous section, we fix a tame finite \'etale group scheme $G$ over $\F_q$ a conjugacy invariant subset $C \subset G(-1) \setminus \{1\}$ which generates $G$ and we let $\mathcal{C} \subset \mathcal{C}_G^*$ be the corresponding conjugacy classes. We may take $\mathbf{P}$ the set of all primes not equal to $p$. To ease notation we will leave this choice of $\mathbf{P}$ implicit everywhere and identify the group schemes of the previous section with their Galois set of $\bar{\F}_q$-points.

Recall that if $k$ is a field and $G$ a finite \'etale group scheme over $k$ then the Brauer group is $\Br BG_k := \HH^2(BG_k, \mathbb{G}_m)$. To understand which components of Hurwitz spaces are fixed by Frobenius, we need a certain subgroup of the Brauer group.

We need to recall some notation from \cite{Loughran2025mallesconjecturebrauergroups} to define this subgroup of the Brauer group. We have provided an alternative description of this subgroup in terms of certain central extensions in \S\ref{sec:description_Brauer_group} for the reader unfamiliar with \cite{Loughran2025mallesconjecturebrauergroups}. Such a reader may take this alternative description as a definition.

Given $c \in \pi_0(\mathcal{C}_{G}^*)$ we let $\mathcal{S}_c \subset [G(-1)/G]$ be the corresponding sector, whose points are given by the groupoid of pairs $(\varphi, \gamma)$ with $\gamma \in c$. We let $\partial_c: \Br BG_{\F_q} \to \HH^1(\mathcal{S}_c, \Q/\Z)$ be the residue defined in \cite[Def.~5.25]{Loughran2025mallesconjecturebrauergroups}.

Recall \cite[Cor.~4.7]{Loughran2025mallesconjecturebrauergroups} that the geometric components of $\mathcal{S}_c$ are in bijection with $\Spec \F_q(c)$ so there is an inclusion $\HH^1(\F_q(c), \Q/\Z) \subset \HH^1(\mathcal{S}_c, \Q/\Z)$. The residue of $\beta \in \Br BG_{\F_q}$ is \emph{algebraic} if $\partial_c(\beta) \in \HH^1(\F_q(c), \Q/\Z)$. 

Recall \cite[Def.~5.21]{Loughran2025mallesconjecturebrauergroups} that $\Br_{\mathcal{\mathcal{C}}} BG_{\F_q}$ consists of those $\beta \in \Br BG_{\F_q}$ such that $\partial_c(\beta) = 0$ for all $c \in \mathcal{C}$.
\begin{definition}\label{def:Brauer_group_C_bar}
    Given a subset $\mathcal{C} \subset \mathcal{C}_G^*$ we let $\Br_{\bar{\mathcal{C}}} BG_{\F_q} \subset \Br BG_{\F_q}$ be the subgroup of Brauer elements $\beta \in \Br BG_{\F_q}$ such that $\beta_{\bar{\F}_q} \in \Br_{\mathcal{C}} BG_{\bar{\F}_q}$. In other words, it consists of those Brauer elements such that for all $c \in \mathcal{C}$ the residue is algebraic.
\end{definition}
Recall that the \emph{algebraic Brauer group} is $\Br_1 BG_{\F_q} := \ker(\Br BG_{\F_q} \to \Br BG_{\bar{\F}_q})$.
\begin{lemma}\label{lem:exact_sequence_brauer_group}
    We have an exact sequence
    \[
    0 \to \Br_1 BG_{\F_q} \to \Br_{\bar{\mathcal{C}}} BG_{\F_q} \to (\Br_{\mathcal{C}} BG_{\bar{\F}_q})^{\Frob_q} \to 0
    \]

    In particular, $\Br_{\bar{\mathcal{C}}} BG_{\F_q} $ is $|G|^2$-torsion.
\end{lemma}
\begin{proof}
    The only non-obvious part is the surjectivity. By the definition of $\Br_{\bar{\mathcal{C}}} BG_{\F_q}$ it suffices to show that $\Br BG_{\F_q} \to (\Br  BG_{\bar{\F}_q})^{\Frob_q}$ is surjective.

    This follows from the Hochschild-Serre spectral sequence $\HH^p(\F_q, \HH^q(BG_{\bar{\F}_q}, \G_m)) \implies \HH^{p + q}(BG_{\F_q}, \G_m)$ and the fact that $\F_q$ has cohomological dimension $1$. 
    
    The statement about torsion follows as $\Br_1 BG_{\F_q}$ and $(\Br_{\mathcal{C}} BG_{\bar{\F}_q})^{\Frob_q}$ are $|G|$-torsion.
\end{proof}

Another subset of the Brauer group which will play a role is the following.
\begin{definition}\label{def:Br_C_ell}
    For $\ell \in \HH^1(\F_q, \Q/\Z)$ define $\Br_{\mathcal{C}, \ell} BG_{\F_q} \subset \Br BG_{\F_q}$ as consisting of those $\beta$ such that the residue $\partial_c(\beta)$ is algebraic and equal to $\ell|_{\F_q(c)} \in \HH^1(\F_q(c), \Q/\Z)$ for all $c \in \pi_0(\mathcal{C})$.
\end{definition}
\begin{remark}
    Note that $\Br_{\mathcal{C}, \ell} BG_{\F_q} \subset \Br_{\bar{\mathcal{C}}} BG_{\F_q}$ by definition.
\end{remark}

The following theorem provides a Brauer-theoretic description of the Frobenius fixed components of Hurwitz spaces.

If $(\varphi, \gamma) \in BG[\F_q((t^{-1}))]$ and $c_{\gamma} \in \mathcal{C}_G^*$ is the conjugacy class containing $\gamma$ then we write $\partial_\gamma(\beta)(\varphi) := \partial_{c_{\gamma}}(\beta)((\varphi, \gamma))$.
\begin{theorem}\label{prop:Brauer_group_controls_components}
    The scheme $U(G_{\varphi}, \mathcal{C})_{\underline{n}, \gamma}$ has a $\F_q$-point if and only if for all $\beta \in \Br_{\bar{\mathcal{C}}} BG_{\F_q}$ we have
    \[
    \partial_\gamma(\beta)(\varphi) + \sum_{c \in \pi_0(\mathcal{C}_G^*)} n_c \mathrm{cor}_{\F_q(c)/\F_q}(\partial_c(\beta)) = 0 \in \HH(\F_q, \Q/\Z).
    \]
\end{theorem}
A first reduction is that we may assume that $\varphi = e_{BG}$ and $G_{\varphi} = G$ as $BG_{\varphi} \cong BG$ by \cite[Lemma~2.10]{Loughran2025mallesconjecturebrauergroups}.
\subsubsection{Description of the Brauer group} \label{sec:description_Brauer_group}
We will use the non-standard notation $\Q/\Z(1) := \lim_{p \nmid n} \mu_n$ to keep the notation light.
\begin{construction}\label{con:central_extensions_Brauer_group}
    Consider a central extension of group schemes $1 \to \Q/\Z(1) \to E \to G \to 1$ of group schemes over $\F_q$. This defines an element $[E] \in \Br BG_{\F_q}$ by \cite[Def.~6.5]{Loughran2025mallesconjecturebrauergroups}. This induces an isomorphism between $BG_{\bar{\F}_q}$ and the group of $\Q/\Z(1)$ central extensions up to isomorphism under the Baer sum. This follows by taking the colimit over all $p \nmid n$ of \cite[Lem.~6.6,6.7]{Loughran2025mallesconjecturebrauergroups} and the fact that $\Br \F_q = 0$.
\end{construction}
\begin{lemma}\label{lem:central_extension_C_bar}
    A central extension $1 \to \Q/\Z(1) \to E \to G \to 1$ is such that $[E] \in \Br_{\bar{\mathcal{C}}} BG_{\F_q}$ if and only if there exists a conjugacy invariant subset $M \subset E(-1)(\bar{\F}_q)$ such that the map $E \to G$ induces a bijection $M \to C(\bar{\F}_q)$.
\end{lemma}
\begin{proof}
    Immediate from \cite[Lem.~6.10]{Loughran2025mallesconjecturebrauergroups} applied over $\bar{\F}_q$.
\end{proof}
\begin{remark}
    If we also assume that $M$ is Galois invariant then it defines a marking in the sense of \Cref{def:marked_central_extension}. In this case we have $[E] \in \Br_{\mathcal{C}} BG_{\F_q}$ by \cite[Lem.~6.10]{Loughran2025mallesconjecturebrauergroups} applied over $\F_q$.
\end{remark}
We can also give a description of the residues in terms of central extensions.
\begin{lemma}\label{lem:computation_central_extension_residue}
    Let $1 \to \Q/\Z(1) \to E \to G \to 1$ be a central extension of group schemes over $\F_q$.
    \begin{enumerate}
        \item Let $\gamma \in G(-1)(\F_q)$ and choose a lift $\tilde{\gamma} \in E(-1)(\bar{\F}_q)$ of $\gamma$. For all $\sigma \in \Gal(\bar{\F}_q/\F_q)$ we have $\sigma(\tilde{\gamma}) - \tilde{\gamma} = \partial_{\gamma}([E])(\sigma) \in \Q/\Z$.
        \item Assume that $[E] \in \Br_{\bar{\mathcal{C}}} BG_{\F_q}$ and choose a subset $M \subset E(-1)(\bar{\F}_q)$ as in \Cref{lem:central_extension_C_bar}. Let $c \in \mathcal{C}(\bar{\F}_q)$ and let $M_c \subset M$ be the subset of elements lying above $c$. For all $\sigma \in \Gal(\bar{\F}_q/\F_q(c))$ we then have $\sigma(M) = M + \partial_{c}([E])(\sigma)$
    \end{enumerate}
\end{lemma}
\begin{proof}
    Both statements follows immediately from \cite[Lem.~6.9]{Loughran2025mallesconjecturebrauergroups} by taking the colimit over all $p \nmid n$ and writing down the torsors in loc.~cit.~in terms of cocycles.
\end{proof}
\subsubsection{Computations}
To prove this theorem we will essentially need to do some computations in the derived category $\mathcal{D}(\Spec \F_q)$ of $\Gal(\bar{\F}_q/\F_q)$-modules. But everything is simple enough that we can give cocycle-theoretic descriptions of all involved objects and will not actually need any facts about derived categories. We leave the argument that the cocycle-theoretic definitions agree with the analogous definitions in the derived category to the interested reader, it follows from a standard bar complex argument.

We will use cohomological duality over finite fields, the statement of which we now recall.
\begin{lemma}\label{lem:cohomological_duality_finite_fields}
    For any finite $\Gal(\bar{\F}_q/\F_q)$-module $A$ the cup product pairing
    \[
    \HH^1(\F_q, A) \times \HH^0(\F_q, \Hom(A, \Q/\Z)) \xrightarrow{ \cdot \cup \cdot } \HH^1(\F_q, \Q/\Z) \cong \Q/\Z: (\varphi, f) \to f_*(\varphi)(\Frob_q).
    \]
    is a perfect pairing.
\end{lemma}
\begin{proof}
    The group $\Gal(\bar{\F}_q/\F_q)$ is profinite and freely generated by $\Frob_q$. This implies that the map $\HH^1(\F_q, A) \to A_{\Frob_q}: \varphi \to \varphi(\Frob_q)$ is an isomorphism, where $A_{\Frob_q}$ denotes the coinvariants of $\Frob_q$. The pairing is then the standard perfect pairing between the coinvariants and the invariants of the dual.
\end{proof}

Let $K_{G} \in \mathcal{D}(\F_q)$ be the complex $\HH_{2, \mathrm{orb}}^{\mathcal{C}}(G, \Z) \to \Hom(\mathcal{C}, \Z)$, where $\HH_{2, \mathrm{orb}}^{\mathcal{C}}(G, \Z)$ has cohomological degree $0$. 

\begin{definition}
    We define $\Hom_{\mathcal{D}(\F_q)}(K_{G}, \Q/\Z)$ as follows.
    
    We will consider pairs $(\alpha, \psi)$ where $\alpha: \HH_{2, \mathrm{orb}}^{\mathcal{C}}(G, \Z) \to \Q/\Z$ and $\psi:\Gal(\bar{\F}_q/ \F_q) \to \Hom(\Hom(\mathcal{C}, \Z), \Q/\Z)$ is a cocycle such that for all $\sigma \in \Gal(\bar{\F}_q/ \F_q)$ we have
    \begin{equation}\label{eq:cocycle_condition_Hom_derived}
       \alpha - \sigma(\alpha) = \psi(\sigma)|_{\HH_{2, \mathrm{orb}}^{\mathcal{C}}(G, \Z)}. 
    \end{equation}

    For all $f \in \Hom(\Hom(\mathcal{C}, \Z),  \Q/\Z)$ we have a pair of the form $(f|_{\HH_{2, \mathrm{orb}}^{\mathcal{C}}(G, \Z)}, d f)$, where $df$ is the coboundary defined by $f$. Then $\Hom_{\mathcal{D}(\F_q)}(K_{G}, \Q/\Z)$ is the group of pairs $(\alpha, \psi)$ as above modulo pairs of the form $(f|_{\HH_{2, \mathrm{orb}}^{\mathcal{C}}(G, \Z)}, d f)$.
\end{definition}
\begin{remark}
    It will often be convenient to think of $\alpha$ as a map $\HH_{2, \mathrm{orb}}^{\mathcal{C}}(G, \hat{\Z}(1)) \to \Q/\Z(1)$ and of $\psi$ as a cocycle $\Gal(\bar{\F}_q/ \F_q) \to \Hom(\Hom(\mathcal{C}, \hat{\Z}(1)), \Q/\Z(1)).$ This is possible because the isomorphism $\Hom(\Hom(\mathcal{C}, \hat{\Z}(1)), \Q/\Z(1)) \cong \Hom(\Hom(\mathcal{C}, \Z), \Q/\Z)$ is natural.
\end{remark}
This group provides obstructions to $U(G, C)_{\underline{n}, \gamma}$ having a $\F_q$-point as follows.

Let $\tilde{\gamma} \in U(G, C)$ be any lift of $\gamma \in G(-1)$. This defines a cocycle $\eta: \Gal(\bar{\F}_q/\F_q) \to \HH_{2, \mathrm{orb}}^{\mathcal{C}}(G, \Z)$ via $\eta(\sigma) := \tilde{\gamma} - \sigma(\tilde{\gamma})$. Note that there exists $u \in U(G, C)_{\underline{n}, \gamma}(\F_q)$ if and only if $v :=  u + \tilde{\gamma} \in \HH_{2, \mathrm{orb}}^{\mathcal{C}}(G, \Z)$ is such that $dv = \eta$ and $\mathfrak{R}(v) = \underline{n} + \mathfrak{R}(\tilde{\gamma})$. We call the pair $(\eta, \underline{n} + \mathfrak{R}(\tilde{\gamma}))$ the \emph{obstruction class}. 

\begin{remark}
    In terms of derived categories the obstruction class lives in the hypercohomology group $\HH^1(\F_q, K_G)$ and the below lemma is implied by the fact that the Yoneda pairing $\HH^1(\F_q, K_G) \times \Hom_{\mathcal{D}(\F_q)}(K_G, \Q/\Z) \to \HH^1(\F_q, \Q/\Z) \cong \Q/\Z$ is perfect.
\end{remark}
\begin{lemma}\label{lem:cohomological duality}
    There exists a $v$ as above if and only if $\psi(\Frob_q)(\underline{n} + \mathfrak{R}(\tilde{\gamma})) - \alpha_*(\eta)(\Frob_q^{-1}) = 0$ for all $(\alpha, \psi) \in \Hom_{\mathcal{D}(\F_q)}(K_{G}, \Q/\Z)$.
\end{lemma}
The expression is well defined as for $(\alpha, \psi) = (f|_{\HH_{2, \mathrm{orb}}^{\mathcal{C}}(G, \Z)}, d f)$ it is 
\[
f(\underline{n} + \mathfrak{R}(\tilde{\gamma})) - f(\Frob_q^{-1}(\underline{n}) + \Frob_q^{-1}(\mathfrak{R}(\tilde{\gamma}))) - f(\mathfrak{R}(\tilde{\gamma} -\Frob_q^{-1}(\tilde{\gamma}))).
\]
This expression is $0$ as $f$ is a homomorphism and $\Frob_q(\underline{n}) = \underline{n}$.
\begin{proof}
    If $v$ exists then the expression is by \eqref{eq:cocycle_condition_Hom_derived} equal to
    \[
    \psi(\Frob_q)(\mathfrak{R} v) - \alpha(dv)(\Frob_q^{-1}) =   (\alpha  - \Frob_q^{-1}(\alpha))(v) - \alpha(v) + \alpha(\Frob_q^{-1}(v)) = 0.
    \]
    First consider pairs of the form $(\alpha , 0)$ with $\alpha \in \HH_{2, \mathrm{orb}}^{\mathcal{C}}(G, \Z)^{\Frob_q}$. The expression is then $\alpha_*(\eta)(\Frob_q^{-1})  = 0$. \Cref{lem:cohomological_duality_finite_fields} implies that this equality for all $\alpha$ implies that $\eta = 0 \in \HH^1(\F_q, \HH_{2, \mathrm{orb}}^{\mathcal{C}}(G, \Z))$. By adding a coboundary we may thus assume that $\eta = 0$ as a cocycle.

    Now consider general pairs $(\alpha, \psi)$. Note that since $\psi$ is a cocycle  and since $\Gal(\overline{\F}_q/\F_q)$ is topologically generated by $\Frob_q$ it suffices for condition \eqref{eq:cocycle_condition_Hom_derived} to hold for $\psi(\Frob_q)$. The possible $\psi(\Frob_q)$ are by \eqref{eq:cocycle_condition_Hom_derived} exactly those $\psi(\Frob_q)$ which map to $0$ in the Galois coinvariants 
    \[
    \Hom(\HH_{2, \mathrm{orb}}^{\mathcal{C}}(G, \Z), \Q/\Z)_{\Frob_q} \cong \Hom(\HH_{2, \mathrm{orb}}^{\mathcal{C}}(G, \Z)^{\Frob_q}, \Q/\Z).
    \]
    Or in other words, the functions which are trivial on the image of $\HH_{2,\mathrm{orb}}^{\mathcal{C}}(G, \Z)^{\Frob_q}$. 
    It follows that $\psi(\Frob_q)(\underline{n} + \mathfrak{R}(\tilde{\gamma})) = 0$ for all such $\psi$ if and only if there exists $v \in \HH_{2,\mathrm{orb}}^{\mathcal{C}}(G, \Z)^{\Frob_q}$ with $\mathfrak{R}(v) = \underline{n}+\mathfrak{R}(\tilde{\gamma})$, as desired.
\end{proof}

The next goal is to relate this condition to the Brauer group.
\begin{construction}
    Given $(\alpha, \psi) \in \Hom_{\mathcal{D}(\Z)}(K_{G_{\varphi}}, \Q/\Z)$ we will construct a central extension of $G$ by $\Q/\Z(1)$. 

    Consider the $\Gal(\bar{\F}_{q}/\F_q)$-group $E_{\psi}$ whose underlying group is $\hat{U}(G, C)(1) \times \Q/\Z(1)$ and such that $\sigma \in \Gal(\bar{\F}_q/\F_q)$ acts as 
    \[
        (u, \zeta) \to (\sigma(u), \sigma(\zeta) \cdot \varphi(\sigma)(\mathfrak{R}(\sigma(u))).
    \]
    
    Now define $\tilde{G}_{(\alpha, \psi)}$ as the quotient of $E_{\psi}$ modulo the subgroup of elements $(v, \alpha(v))$ for $v \in \HH_{2, \mathrm{orb}}^{\mathcal{C}}(G, \hat{\Z}(1))$. This subgroup is Galois invariant as by \eqref{eq:cocycle_condition_Hom_derived}
    \[
    \sigma (v, \alpha(v)) = (\sigma(v), \sigma(\alpha)(\sigma(v)) \cdot \varphi(\sigma)(\mathfrak{R}(\sigma(v))) = (\sigma(v), \alpha(\sigma(v)).
    \]
    
    The sequence $\HH_{2, \mathrm{orb}}^{\mathcal{C}}(G, \hat{\Z}(1)) \times \Q/\Z(1) \to E_{\psi} \to G$ is a central extension because $\HH_{2, \mathrm{orb}}^{\mathcal{C}}(G, \hat{\Z}(1)) \to \hat{U}(G, C)(1) \to G$ is a central extension. So $\Q/\Z(1) \to \tilde{G}_{(\alpha, \psi)} \to G$ is a central extension. Let $[\tilde{G}_{(\alpha, \psi)}] \in \Br BG_{\F_q}$ be its class from \Cref{con:central_extensions_Brauer_group}.

    It is straightforward but tedious to check that the map $(\alpha, \psi) \to [\tilde{G}_{(\alpha, \psi)}]$ is additive using that the addition of extensions is given by the Baer sum. Moreover, if $(\alpha, \psi) = (f|_{\HH_{2, \mathrm{orb}}^{\mathcal{C}}(G, \Z)}, d f)$ then the map $(u, \zeta) \to (u, \zeta \cdot f(\mathfrak{R}(u))$
    clearly defines a Galois equivariant map $\hat{U}(G, C)(1) \times \Q/\Z(1) \to E_{df}$. We have thus constructed a well-defined map $\Hom_{\mathcal{D}(\Z)}(K_{G_{\varphi}}, \Q/\Z) \to \Br BG_{\F_q}$.
\end{construction}
\begin{lemma}\label{lem:cocycles_geometric_unramified}
    The map $\Hom_{\mathcal{D}(\Z)}(K_{G_{\varphi}}, \Q/\Z) \to \Br BG_{\F_q}$ is injective and its image is given by $ \Br_{\bar{\mathcal{C}}} BG_{\F_q}$.
\end{lemma}
\begin{proof}
    We start with injectivity. If $[\tilde{G}_{(\alpha, \psi)}] = 0$ then the central extension splits as \Cref{con:central_extensions_Brauer_group} is an isomorphism. Such a splitting induces a retract $E_{\psi} \to \Q/\Z(1)$. This retract factors as $E_{\psi} \xrightarrow{\mathfrak{R}} \Hom(\mathcal{C}, \hat{\Z}(1)) \xrightarrow{f} \Q/\Z(1)$ by \Cref{lem:M_is_abelianization}.

    By the definition of a retract it induces an Galois equivariant isomorphism
    \[
    E_{\psi} \cong \hat{U}(G_{\varphi}, C)(1) \times \Q/\Z(1): (u, \zeta) \to (u, \zeta \cdot f(\mathfrak{R}(u))
    \]
    which preserves the kernels of the maps to  $\tilde{G}_{(\alpha, \psi)}$ and $G \times \Q/\Z(1)$, i.e.~ it sends $\{(v, \alpha(v) : v \in \HH_{2, \mathrm{orb}}^{\mathcal{C}}(G, \hat{\Z}(1))\}$ to $\HH_{2, \mathrm{orb}}^{\mathcal{C}}(G, \hat{\Z}(1)) \times 0$.
    
    That $\psi$ preserves the Galois action means that for all $\sigma \in \Gal(\bar{\F}_q/\F_q)$ we have
    \[
    (\sigma(u), \sigma(\zeta) \cdot \psi(\sigma)(\mathfrak{R}(\sigma(u))) \cdot f(\sigma(u))) = (\sigma(u), \sigma(\zeta) \cdot \sigma(f)(\mathfrak{R}(\sigma(u)))).
    \]
    Which means that $\psi = - df$. The condition that the map sends the kernels to each other is that exactly that $f|_{\HH_{2, \mathrm{orb}}^{\mathcal{C}}(G, \hat{\Z}(1))} = \alpha$.  We have thus shown that $(\alpha, \psi) = - (f|_{\HH_{2, \mathrm{orb}}^{\mathcal{C}}(G, \Z)}, df)$ is trivial in $\Hom_{\mathcal{D}(\Z)}(K_{G_{\varphi}}, \Q/\Z)$.
    
    We now show $[\tilde{G}_{(\alpha, \psi)}] \in \Br_{\bar{\mathcal{C}}} BG_{\F_q}$. We will apply \Cref{lem:central_extension_C_bar}. The image of $\bigcup_{\gamma \in C}([\gamma], 0) \subset \hat{U}(G, C)(\bar{\F}_q) \times \Q/\Z$ in $\tilde{G}_{(\alpha, \psi)}(-1)(\bar{\F}_q)$ is a set satisfying the conditions of $M$ in \Cref{lem:central_extension_C_bar} by construction.

    It remains to show surjectivity. As \Cref{con:central_extensions_Brauer_group} induces an isomorphism we have to show that every $\Q/\Z(1)$-central extension $E \to G$ with a subset $M \subset E(-1)$ as in \Cref{lem:central_extension_C_bar} is isomorphic to some $\tilde{G}_{(\alpha, \psi)}$. For each $\xi \in C(\bar{\F}_q)$ let $\widehat{\gamma} \in M$ be the corresponding element.

    For all $\sigma \in \Gal(\bar{\F}_q/\F_q)$ the image of $\sigma(\hat{\xi})$ in $G(-1)(\bar{\F}_q)$ is $\sigma(\xi)$ so there exists $h_{\sigma, \xi} \in \Q/\Z$ such that $\sigma(\widehat{\xi}) = h_{\sigma, \xi} \widehat{\sigma(\xi)}$. The assumption that $M$ is invariant under conjugation by elements of $G$ implies that the element $h_{\sigma, \xi}$ only depends on the conjugacy class $c_{\xi} \in \mathcal{C}(\bar{\F}_q)$ containing $\xi$ and we write $h_{\sigma, c_{\xi}} := h_{\sigma, \xi}$.

    Note that for $\sigma, \tau \in \Gal(\bar{\F}_q/\F_q)$ we have
    \[
    \widehat{\sigma(\tau(\xi))}h_{\sigma \circ \tau, c_{\xi}} =\sigma (\tau(\widehat{\xi})) = \sigma(\widehat{\tau(\xi)} h_{\tau, c_{\xi}} ) =  \widehat{\sigma(\tau(\xi))}h_{\sigma, \tau(c_{\xi})} \sigma(h_{\tau, \tau(c_{\xi})} ).
    \]
    It follows that the formula $\sigma \to (c \to h_{\sigma, \sigma^{-1}(c)})$ defines a cocycle $\psi: \Gal(\bar{\F}_q/\F_q) \to \Hom(\mathcal{C}, \Q/\Z)$.

    Consider the map $E_{\psi} \to \tilde{G}$ which sends $\Q/\Z(1) \to \Q/\Z(1)$ and $[\xi]^{\zeta} \in \hat{U}(G, C)(1)$ to $\hat{\xi}(\zeta)$. This map sends the relations defining $\hat{U}(G, C)(1)$ to $0$ by the definition of a marking, and is Galois invariant by construction. 
    
    The map is compatible with the maps to $G$ by construction. The kernel is thus of the form $\{(v, \alpha(v)): v \in \HH_{2, \mathrm{orb}}^{\mathcal{C}}(G, \hat{\Z}(1)) \}$ for some function $\alpha: \HH_{2, \mathrm{orb}}^{\mathcal{C}}(G, \hat{\Z}(1)) \to \Q/\Z(1)$. For this kernel to be a group $\alpha$ needs to be a morphism. The condition that the kernel is Galois invariant is then exactly \eqref{eq:cocycle_condition_Hom_derived}. We then have $E \cong \tilde{G}_{(\alpha, \psi)}$ by construction.
\end{proof} 
To prove \Cref{prop:Brauer_group_controls_components} it suffices by \Cref{lem:cohomological duality} to show the following.
\begin{lemma}\label{lem:cocycle_residue_formula}
    The expression $\psi(\Frob_q)(\underline{n} + \mathfrak{R}(\tilde{\gamma})) - \alpha_*(\eta)(\Frob_q^{-1})) \in \Q/\Z$ is equal to
    \[
    \partial_\gamma(\beta)(e_{BG})(\Frob_q) + \sum_{c \in \pi_0(\mathcal{C}_G^*)} n_c \mathrm{cor}_{\F_q(c)/\F_q}(\partial_c(\beta))(\Frob_q)
    \]
    for $\beta = [\tilde{G}_{(\alpha, \psi)}]$.
\end{lemma}
\begin{proof}
    We first do the case $\underline{n} = 0$. Let us first describe $\partial_{\gamma} := \partial_\gamma([\tilde{G}_{\alpha, \psi}])(\Frob_q)$. Note that $\gamma \in G(-1)(\F_q)$.

    Note that because $\gamma \in G(-1)(\F_q)$ the element $\Frob_q^{-1} (\tilde{\gamma}) \in U(G, C)(\bar{\F}_q)$ also maps to $\gamma$. By construction we have the following equality in $E_{\psi}(-1)(\bar{\F}_q)$.
    \begin{equation*}
        \begin{split}
         &\Frob_q(\Frob_q^{-1}(\tilde{\gamma}), 0) = (\tilde{\gamma}, \psi(\Frob_q)(\mathfrak{R}(\tilde{\gamma})))) =\\ &(\Frob_q^{-1}(\tilde{\gamma}), \psi(\Frob_q(\mathfrak{R}(\tilde{\gamma}))  - \alpha(\eta(\Frob_q^{-1})) + (\eta(\Frob_q^{-1}), \alpha(\eta(\Frob_q^{-1}))).
         \end{split}
    \end{equation*}
    The last term maps to $0$ in $\tilde{G}_{\alpha, \psi}$ by construction. So $\partial_{\gamma} = \psi(\Frob_q)(\mathfrak{R}(\tilde{\gamma})) - \alpha_*(\eta)(\Frob_q^{-1})$ by \Cref{lem:computation_central_extension_residue}

    Fix $c \in \mathcal{C}(\bar{\F}_q)$ and $\xi \in c(\bar{\F}_q) \subset G(-1)(\bar{\F}_q)$. Let $\mathbf{1}_c \in \Hom(\mathcal{C}, \N)$ be the function which takes the value $1$ on $c$ and the value $0$ on all other conjugacy classes. Let us abbreviate $\partial_c := \mathrm{cor}_{\F_q(c)/\F_q}(\partial_c(\beta))(\Frob_q) = \partial_c(\beta)(\Frob_q^n)$. Let $m := [\F_q(c)/\F_q]$ By linearity and since $\underline{n}$ is assumed to be Galois-fixed it suffices to show that $\psi(\Frob_q)(\sum_{i = 1}^m\mathbf{1}_{\Frob_q^i c}) =  \partial_c$.

    Let us first describe the residue $\partial_c$. The choice of $m$ means that $\Frob_q^m(\xi)$ is conjugate to $\xi$. Let $h \in G$ be such that $\Frob_q^m(\xi) = h \xi h^{-1}$. Let $\hat{\xi} \in \tilde{G}_{(\alpha, \psi)}(-1)$ be a lift of $\xi$. The fact that $[\tilde{G}_{(\alpha, \psi)}] \in \Br_{\bar{\mathcal{C}}} BG_{\F_q}$ implies by \Cref{lem:computation_central_extension_residue} that $\Frob_q^m(\hat{\xi}) = \partial_c \hat{\xi} h^{-1}$.

    Now consider the symbol $[\xi] \in U(G, C)(\bar{\F}_q)$. By construction the element $\Frob_q^n([\xi], 0)$ in $E_{\psi}(-1)$ is equal to
    \[
    \left([\Frob_q^m\xi], \sum_{i = 1}^m \varphi(\Frob_q)( \mathbf{1}_{\Frob_q^i(c)})\right) = \left(h [\xi] h^{-1},  \varphi(\Frob_q)\left(\sum_{i = 1}^m \mathbf{1}_{\Frob_q^i(c)}\right)\right).
    \]
    We conclude that $\partial_c = \varphi(\Frob_q)(\sum_{i = 1}^m \mathbf{1}_{\Frob_q^i(c)})$.
\end{proof}
We also need the following lemma.
\begin{lemma}\label{lem:galois_module_Br_C}
    There is an isomorphism of Galois modules $\Br_{\mathcal{C}} BG_{\bar{\F}_q} \cong \Hom(\HH_2(G, \mathcal{C}), \Q/\Z))$.
\end{lemma}
\begin{proof}
    By definition, we have $\Br_{\mathcal{C}} BG_{\bar{\F}_q} = \varinjlim_k \Br_{\bar{\mathcal{C}}} BG_{\F_q^k}$. Applying \Cref{lem:cocycles_geometric_unramified} for all $k$ shows that this is the Galois module $\Hom(\HH_{2, \mathrm{orb}}^{\mathcal{C}}(G, \Z), \Q/\Z)$ modulo the image of $\Hom(\Hom(\mathcal{C}, \Z), \Q/\Z)$. This is $\Hom(\HH_2(G, \mathcal{C}), \Q/\Z)$ by the definition of $\HH_2(G, \mathcal{C})$.
\end{proof} 
\begin{corollary}\label{cor:size_Brauer_group}
    We have an equality $|\Br_{\bar{\mathcal{C}}} BG_{\F_q}| = |G^{\mathrm{ab}}(-1)(\F_q)| \cdot |\HH_2(G, \mathcal{C})(\F_q)|$.
\end{corollary}
\begin{proof}
    By \Cref{lem:exact_sequence_brauer_group} this follows from the following two facts.

    We have $\Br_1 BG_{\F_q} \cong \HH^1(\F_q, \Hom(G^{\mathrm{ab}}, \Q/\Z(1)))$ by \cite[Lem.~6.2]{Loughran2025mallesconjecturebrauergroups}. This later group is dual to $\HH^0(\F_q, G^{\mathrm{ab}}(-1))$ by \Cref{lem:cohomological_duality_finite_fields}. So $|\Br_1 BG_{\F_q}| = |G^{\mathrm{ab}}(-1)(\F_q)|$.

    Moreover by \Cref{lem:galois_module_Br_C} we have $|(\Br_{\mathcal{C}} BG_{\bar{\F}_q})^{\Frob_q}| = |\Hom(\HH_2(G, \mathcal{C}), \Q/\Z))^{\Frob_q}| = |\HH_2(G, \mathcal{C})(\bar{\F}_q)_{\Frob_q}|$. This has the same size as $|\HH_2(G, \mathcal{C})(\F_q)|$ by the exact sequence 
    \[
        0 \to \HH_2(G, \mathcal{C})(\F_q) \to \HH_2(G, \mathcal{C})(\bar{\F}_q) \xrightarrow{\Frob_q - 1} \HH_2(G, \mathcal{C})(\bar{\F}_q) \to \HH_2(G, \mathcal{C})(\bar{\F}_q)_{\Frob_q} \to 0
    \]
\end{proof}
\section{Counting}
\subsection{Point-counting on Hurwitz spaces}
Let $\F_q$ be a finite field and $G$ a finite tame \'etale group scheme over $\F_q$. Recall that for a subscheme $\mathcal{C} \subset \mathcal{C}_G^*$ and element $\underline{n} \in \Hom_{\F_q}(\mathcal{C}, \N)$ we  write $|\underline{n}| = \sum_{c \in \mathcal{C}_G^*(\bar{\F}_q)} n_c$.

We isolate what we need from Landesman--Levy \cite{landesman2025cohenlenstramomentsfunctionfields, Landesman2025Homological, landesman2025stablehomologyhurwitzmodules} in the following lemma.
\begin{lemma} \label{lem:application_landesman_levy}
    Assume that $q$ and $|G(\bar{\F}_q)|$ are coprime and that $q$ is sufficiently large in terms of $|G(\bar{\F}_q)|$.
    \begin{enumerate}\label{eq:bound_hurwitz_spaces_points}
        \item We have uniformly in $(\varphi, \gamma) \in BG(\F_q((t^{-1})))$ and $\underline{n} \in \Hom_{\F_q}(\mathcal{C}_G^*, \N)$ that 
    \begin{equation}
        \# \CHur^{G, (\varphi, \gamma)}_{\underline{n}}(\F_q) = O(q^{|\underline{n}|})
    \end{equation}
    \item Consider a chain of subsets $\mathcal{C}_{\mathrm{st}} \subseteq \mathcal{C} \subseteq \mathcal{C}_G^*$ such that the conjugacy classes in $\mathcal{C}_{\mathrm{st}}$, resp.~$\mathcal{C}$, generate $G/Z(G)$, resp.~$G$. Let $C \subset G(-1)$ be the subscheme which is the union of elements of $\mathcal{C}$.

    There exists $\delta > 0$ such that uniformly in $(\varphi, \gamma) \in BG(\F_q((t^{-1}))$ and $\underline{n} \in \Hom_{\F_q}(\mathcal{C}, \Z_{> 0})$ we have that 
    \begin{equation}\label{eq:asymptotic_hurwitz_spaces_points}
        \# \CHur^{G, (\varphi, C, \gamma)}_{\underline{n}}(\F_q) = \#U(G_{\varphi}, \mathcal{C})_{\underline{n}, \gamma}(\F_q) \# \Conf_{\mathcal{C}, \underline{n}}(\F_q) + O\left(q^{|\underline{n}| - \delta \min_{c \in \mathcal{C}_{\mathrm{st}}(\bar{\F}_q)}(n_c)}\right).
    \end{equation}
    \end{enumerate}
    Here $\Conf_{\mathcal{C}, \underline{n}}$ denotes the multicolored configuration space on the colors $\mathcal{C}$.
\end{lemma}
\begin{proof}
    If $G$ is constant and $\varphi = e_{BG}$ then this is a somewhat standard consequence of the homological stability  result \cite[Thm.~1.4.9]{landesman2025stablehomologyhurwitzmodules}; see ~\cite[Proof of Thm.~10.0.13]{landesman2025stablehomologyhurwitzmodules} for the case $G = S_n$. The proof generalises essentially without change. Let us provide more details.

    We claim that the number of geometric components of $\CHur^{G, (\varphi, \gamma)}_{\underline{n}} \subset \CHur^{G, \varphi}_{\underline{n}}$ is bounded uniformly in $\underline{n}$. Indeed, this is a purely geometric statement so it reduces to the field $\C$ where it follows from the statement about $\HH^0$ in \cite[Thm.~1.4.8]{landesman2025stablehomologyhurwitzmodules}.

    For (1) it thus suffices to give upper bounds on the number of points on each component of the Hurwitz scheme. The case that $G$ is constant and $\varphi = e_{BG}$ is \cite[Lem.~8.4.4]{Landesman2025Homological}. The proof generalizes without changes to the more general setting.

    Let us now consider (2). By part (1) and Lemma \ref{lem:summary_hurwitz_spaces} we may assume that the Frobenius fixed components of $\CHur^{G, C, (\varphi, \gamma)}_{\underline{n}}$ are in bijection with $U(G_{\varphi}, \mathcal{C})_{\underline{n}, \gamma}(\F_q)$. It thus suffices to show that if $Z \subset \CHur^{G, \varphi}_{\underline{n}}$ is a Frobenius fixed
    component of a Hurwitz scheme then
    
    \begin{equation}\label{eq:components_hurwitz_space_point_count}
        Z(\F_q) = \Conf_{\mathcal{C}, \underline{n}}(\F_q) + O\left(q^{|\underline{n}| - \delta \min_{c \in \mathcal{C}_{\mathrm{st}}(\bar{\F}_q)}(n_c)}\right)
    \end{equation}

    To prove this we use the comparison map $Z \to \Conf_{\mathcal{C}, \underline{n}}$ induced by \Cref{def:map_to_configuration_space}. We claim that this map induces an isomorphism on $\ell$-adic cohomology groups of degree $\gg \min_{c \in \mathcal{C}_{\mathrm{st}}(\bar{\F}_q)}(n_c)$. Indeed, this is a geometric statement so it reduces to a statement over $\C$ where it follows from repeated applications of \cite[Thm.~1.4.6]{landesman2025stablehomologyhurwitzmodules}. 
    
    The precise application of this theorem is as follows. We work over $\C$ to remove the anticyclotomic twists, this is possible as the cohomology groups over $\bar{\F}_q$ and $\C$ are isomorphic by \cite[Rem.~8.4.3]{Landesman2025Homological}. Write $\mathcal{C}_{\mathrm{st}}(\C) = \{c_1, \dots, c_k\}$ and $\mathcal{C}(\C) = \{c_1, \dots, c_m\}$. The union $C := \bigcup_{i = 1}^m c_i$ defines a rack as in \cite[Ex.~2.1.3]{Landesman2025Homological}. Iterated applications of \cite[Thm.~1.4.6]{landesman2025stablehomologyhurwitzmodules} show that the natural map $\CHur^{C}_{n_{c_1}, \dots, n_{c_m}} \to \CHur^{(C/c_1) \cdots /c_{k}}_{n_{c_1}, \dots, n_{c_m}}$ induces isomorphisms on the rational cohomology groups of degree $\gg \min_{c \in \mathcal{C}_{\mathrm{st}}}(n_c)$. However, the condition that $\mathcal{C}_{\mathrm{st}}$ generates the inner automorphism group $G/Z(G)$ of $G$ ensures that each component of the rack $(C/c_1) \cdots /c_{k}$ has only a single element, which means that the corresponding Hurwitz space is the corresponding multicoloured configuration space by definition (see \cite[Def.~2.2.2, Rem.~6.0.3]{Landesman2025Homological}).

    The statement about the isomorphism of cohomology groups then implies \eqref{eq:components_hurwitz_space_point_count} by \cite[Lem.~5.2.2]{landesman2025cohenlenstramomentsfunctionfields} and \cite[Lem.~8.4.2]{Landesman2025Homological}.
    \end{proof}

We now describe the number of points on Hurwitz spaces in terms of configuration spaces and the Brauer group.
\begin{lemma}\label{lem:Hurwitz_spaces}
    Let $\mathcal{C}_{\mathrm{st}} \subset \mathcal{C} \subset \mathcal{C}_G^*, \underline{n}$ be as in \Cref{lem:application_landesman_levy}(2). Then $\# \CHur_{\underline{n}}^{G, C, (\varphi, \gamma)}(\F_q)$ is equal to        
        \[
         \frac{\# \Conf_{\mathcal{C}, \underline{n}}(\F_q)}{|G^{\mathrm{ab}}(-1)(\F_q)|}\sum_{\beta \in \Br_{\bar{\mathcal{C}}} BG_{\F_q}} e(\partial_{\gamma}(\beta)(\varphi)) \prod_{c \in \pi_0(\mathcal{C})} e(n_c \mathrm{cor}_{\F_q(c)/\F_q}(\partial_c(\beta)) + O(q^{|\underline{n}| - \delta \min_{c \in \mathcal{C}_{\mathrm{st}}(\bar{\F}_q)}(n_c)}).
        \]
\end{lemma}
\begin{proof}
    By character orthogonality, \Cref{lem:application_landesman_levy} and \Cref{prop:Brauer_group_controls_components} we have that it is up to an error term $O(q^{|\underline{n}| - \delta \min_{c \in \mathcal{C}_{\mathrm{st}}(\bar{\F}_q)}(n_c)})$ equal to 
    \[\sum_{\beta \in \Br_{\bar{\mathcal{C}}} BG_{\F_q}} e\left(\partial_{\gamma}(\beta)(\varphi) + \sum_{c \in \pi_0(\mathcal{C})} n_c \mathrm{cor}_{\F_q(c)/\F_q}(\partial_c(\beta)))\right)\frac{\# \Conf_{\mathcal{C}, \underline{n}}(\F_q)|\HH_2(G, \mathcal{C})(\F_q)|}{|\Br_{\bar{\mathcal{C}}} BG_{\F_q}|}.\]

    The lemma then follows from \Cref{cor:size_Brauer_group}.
\end{proof}
\subsection{Tamagawa volumes}
\subsubsection{Tamagawa measures}\label{sec:Tamagawa_measures}
Consider the global field $K := \F_q(t)$. For each closed point $P \in \Proj^1_{\F_q}$ let $K_P$ be the corresponding local field. Let $\Adele_K := \prod'_{P \in \Proj^1_{\F_q}} K_P, \Adele_{K, \infty} := \prod'_{P \in \A^1_{\F_q}} K_P$ be the ring of adeles of $K$, resp.~the ring of adeles outside $\infty$. We used the notation $\prod'$ for the restricted product.

Let $\tau_{BG, \F_q} := \prod_{P \in \Proj^1_{\F_q}} \tau_P, \tau_{BG, \Adele_{K, \infty}} := \prod_{P \in \A^1_{\F_q}} \tau_P$ be the Tamagawa measure defined on $BG[\Adele_{K}]$, resp.~$BG[\Adele_{K, \infty}]$. These are defined analogously as in \cite[\S8.2, 8.4]{Loughran2025mallesconjecturebrauergroups} but with the constant heigth $1$.

To be precise, $\tau_P$ is the local Tamagawa measures defined on $BG[K_P]$ by the groupoid counting formula, i.e.~for $\Omega_P \subset BG[K_P]$ we have 
\[
\tau_P(\Omega_P) := \sum_{\varphi_P \in \Omega_P} \frac{1}{|\Aut(\varphi_P)(K_P)|}.
\]
The $\tau_{BG, \F_q}, \tau_{BG, \Adele_{K, \infty}}$ are then defined as the corresponding product measures, which converge absolutely by \cite[Lem.~8.9]{Loughran2025mallesconjecturebrauergroups}.
\subsubsection{Brauer-Manin pairing}\label{sec:Brauer_Manin_pairing}
We recall \cite[\S5.6.1]{Loughran2025mallesconjecturebrauergroups} that there is a global Brauer-Manin pairing 
\[
\langle \cdot , \cdot \rangle_{\mathrm{BM}} : BG[\Adele_{\F_q(t)}] \times \Br BG_{\F_q} \to \Q/\Z
\]
which is the sum $\langle \cdot , \cdot \rangle_{\mathrm{BM}} = \sum_{P \in \Proj^1_{\F_q}} \langle \cdot , \cdot \rangle_{\mathrm{BM}, P}$ of local Brauer-Manin pairings
\[
\langle \cdot , \cdot \rangle_{\mathrm{BM}, P} : BG[K_P] \times \Br BG_{\F_q} \to \Q/\Z: (\varphi_P, \beta) \to \text{inv}_P(\beta(\varphi_P)).
\]
Here $\text{inv}_P: \Br K_P \cong \Q/\Z$ is the local invariant map coming from local class field theory.

For any subset $\Omega \subset BG[\Adele_{\F_q(t)}]$ we define $\Omega^{\Br} \subset \Omega$ as the subset of elements which are orthogonal to $\Br BG_{\F_q}$ with respect to the Brauer-Manin pairing. More generally, for any subgroup $B \subset \Br BG_{\F_q}$ we define $\Omega^{B} \subset \Omega$ as the subset of elements orthogonal to $B$ under the Brauer-Manin pairing.

For $\beta \in \Br BG_{\F_q}$ and any Borel subset $\Omega \subset BG[\Adele_{\F_q(t)}]$ we consider the Brauer transform
\[
\hat{\tau}(\beta ; \Omega) := \int_{\Omega} e^{2 i \pi \langle \varphi, \beta \rangle_{\mathrm{BM}}} d\tau_{BG, \F_q}(\varphi).
\]
This is an important quantity because of the following lemma, c.f.~\cite[Lem.~8.21]{Loughran2025mallesconjecturebrauergroups}.
\begin{lemma}\label{lem:character_orthogonality}
    Let $\Omega \subset BG[\Adele_{\F_q(t)}]$ be such that $\tau_{BG, \F_q}(\Omega) < \infty$ and $B \subset  \Br BG_{\F_q}$ a subgroup. Then 
    \[
    |B| \tau_{BG, \F_q}(\Omega^B) = \sum_{\beta \in B} \hat{\tau}(\beta ; \Omega).
    \]
\end{lemma}
\begin{proof}
    This is an immediate consequence of character orthogonality, as for $\varphi \in \Omega$ we have
    \[
    \sum_{\beta \in B} e^{2 i \pi \langle \varphi, \beta \rangle_{\mathrm{BM}}} = \begin{cases}
        |B| \text{ if } \varphi \in \Omega^B \\
        0 \text{ otherwise}.
    \end{cases}
    \]
\end{proof}
For any point $P \in \Proj^{1}_{\F_q}$ and subset $\Omega_P \subset BG[K_P]$ we will also consider the local Brauer transform
\[
\hat{\tau}_P(\beta ; \Omega_P) := \int_{\Omega_P} e^{2 i \pi \langle \varphi_P, \beta \rangle_{\mathrm{BM}, P}} d\tau_{P}(\varphi_P).
\]

If $P = \infty$ then we have the following explicit formula, where we recall from the abbreviation $\partial_{\gamma}(\beta) := \partial_{c_{\gamma}}(\beta)((\varphi, \gamma))$ from just before \Cref{prop:Brauer_group_controls_components}.
\begin{equation}\label{eq:formula_local_Brauer_transform}
   \hat{\tau}_P(\beta ; \Omega_P) = \sum_{(\varphi, \gamma) \in BG[\F_q((t^{-1}))]} \frac{e^{2 i \pi \text{inv}_{\infty}(\beta(\varphi, \gamma))}}{|\mathrm{Aut}_{BG}(\varphi, \gamma)(\F_q)|} = \sum_{(\varphi, \gamma) \in BG[\F_q((t^{-1}))]} \frac{e(\partial_{\gamma}(\varphi)(\beta))}{|\mathrm{Aut}_{BG}(\varphi, \gamma)(\F_q)|}.
\end{equation}
The first equality is just the definition of the integral and the second equality is by \cite[Lem.~7.3]{Loughran2025mallesconjecturebrauergroups}.

\subsubsection{Volume computations}
For all $c \in \pi_0(\mathcal{C}^*_G)$ we introduce a formal variable $T_c$ and use the multi-index notation $\underline{T}^{\underline{n}} = \prod_{c} T_c^{n_c}$. For $\varphi_P \in BG[K_P]$, we write $T(\varphi_P) := T_c^{\deg P}$ such that $c$ is the ramification type of $\varphi_P$. For $(\varphi_P)_P \in BG[\Adele_{K, \infty}]$, we write that $T(\varphi) = \prod_P T(\varphi_P)$.

We require the following computation of a global Brauer transform. One should think of the left-hand side as a `Brauer-twisted adelic Igusa integral' in analogue with \cite[\S4.3]{Chambert-Loir2010Igusa}.
\begin{lemma} \label{lem:Brauer_twisted_Igusa_integral}
    Let $\beta \in \Br BG_{\F_q}$ and let $\mathcal{C}_{\beta} \subset \mathcal{C}_G^*$ be the set of conjugacy classes $c$ where the residue of $\beta$ at $c$ is algebraic. For $c \in \pi_0(\mathcal{C}_{\beta})$ with field of definition $\F_q(c)$ we thus have $\partial_c(\beta) \in \HH^1(\F_q(c), \Q/\Z)$. We then have the equality
    \begin{equation}\label{eq:formula_global_Brauer_transform}
       \int_{BG[\Adele_{K, \infty}]} e^{2 i \pi \langle \beta, \varphi \rangle_{\mathrm{BM}}} T(\varphi) d\tau_{BG, \Adele_{K, \infty}}(\varphi) = \prod_{P \in \A^1_{\F_q}} \left( 1 + \sum_{c \in \mathcal{C}_{\beta}(\F_q(P))} e(\mathrm{cor}_{\F_q(P)/\F_q}(\partial_c(\beta))T_c^{\deg P} \right). 
    \end{equation}
\end{lemma}
\begin{proof}
    As $\tau_{BG, \Adele_{K, \infty}} = \prod_{P \in \A^1_{\F_q}}\tau_P$ and $\langle\varphi, \beta \rangle_{\mathrm{BM}} = \sum_{P \in \A^1_{\F_q}} \langle\varphi_P, \beta \rangle_{\mathrm{BM}, P}$ we have
    \[
     \int_{BG[\Adele_{K, \infty}]} e^{2 i \pi \langle \varphi, \beta \rangle_{\mathrm{BM}}} d\tau_{BG, \Adele_{K, \infty}} = \prod_{P \in \A^1_{\F_q}} \int_{BG(K_P)}  e^{2 i \pi \langle \varphi_P, \beta \rangle_{\mathrm{BM}, P}} d\tau_P(\varphi_P).
    \]

    The local integrals for $P$ are by Denef's formula \cite[Thm.~8.23]{Loughran2025mallesconjecturebrauergroups} equal to 
    \[
    \left( 1 + \sum_{c \in \mathcal{C}_{\beta}(\F_q(P))} e(\mathrm{cor}_{\F_q(P)/\F_q}(\partial_c(\beta))T_c^{\deg P} \right).
    \]
    Here we used that $\mathrm{cor}_{\F_q(P)/\F_q}(\partial_c(\beta))(\Frob_q) = \partial_c(\beta)(\Frob_{\F_q(P)})$.
\end{proof}

\begin{definition}
    For $\underline{n} \in \Hom_{\F_q}(\mathcal{C}_G^{*}, \N)$ let $BG[\Adele_{K, \infty}]_{\underline{n}} \subset BG[\Adele_{K, \infty}]$ be the subset consisting of elements $(\varphi_P)_{P \in \A^1_{\F_q}}$ such that there are exactly $n_c$ points $P$, counted with multiplicity $[\F_q(P): \F_q(c)]$, such that $P$ has ramification type $c \in \mathcal{C}_G^*(\bar{\F}_q)$.

    For any subset $\Omega_{\infty} \subset BG[\F_q((t^{-1}))]$ we will write 
    \[BG[\Adele_{\F_{q}(t)}]_{\underline{n}, \Omega_{\infty}} := \Omega_{\infty} \times BG[\Adele_{K, \infty}]_{\underline{n}} \subset BG[\Adele_{\F_q(t)}].\]
\end{definition}

The Brauer transform of the subsets $BG[\Adele_{\F_{q}(t)}]_{\underline{n}, \Omega_{\infty}}$ turns out to be related to the number of points on configuration spaces. 
\begin{proposition}\label{prop:hurwitz_spaces_Tamagawa_measures} 
    Let $\beta \in \Br BG_{\F_q}$ and let $\mathcal{C}_{\beta} \subset \mathcal{C}_G^*$ be the set of conjugacy classes $c$ where the residue $\partial_c(\beta)$ is algebraic.

    For any subset $\Omega_{\infty} \subset BG[\F_q((t^{-1}))]$ and $\underline{n} \in \Hom(\mathcal{C}_G^*, \N)$ we have the following:
    \begin{enumerate}
        \item If $\underline{n}$ is not supported on $\mathcal{C}_{\beta}$ then $\hat{\tau}(\beta ; \Omega_\infty \times BG[\Adele_{K, \infty}]_{\underline{n}}) = 0$.
        \item If $\underline{n}$ is  supported on $\mathcal{C}_{\beta}$ then 
    \[
    \hat{\tau}(\beta ; BG[\Adele_{\F_{q}(t)}
    ]_{\underline{n}, \Omega_{\infty}}) = \hat{\tau}_\infty(\beta ; \Omega_\infty) \# \Conf_{\mathcal{C}_{\beta}, \underline{n}}(\F_q)\prod_{c \in \pi_0(\mathcal{C}_{\beta})}e(n_c \mathrm{cor}_{\F_q(c)/\F_q}(\partial_c(\beta))) .
    \]
    \end{enumerate}
\end{proposition}
\begin{proof}
    As $\tau_{BG, \F_q} = \tau_{\infty} \times \tau_{BG, \A^{1}_{\F_q}}$ and since the Brauer-Manin pairing is a sum of local pairings we have $\hat{\tau}(\beta ; BG[\Adele_{\F_{q}(t)}
    ]_{\underline{n}, \Omega_{\infty}}) = \hat{\tau}_\infty(\beta ; \Omega_\infty) \cdot \hat{\tau}(\beta ; \underline{n})$ with
    \[
        \hat{\tau}(\beta ; \underline{n}) := \int_{BG[\Adele_{K, \infty}]_{\underline{n}}} e^{2 i \pi \langle \beta , \varphi\rangle_{\mathrm{BM}}} d\tau_{\A^{1}_{\F_q}}(\varphi).
    \]

    For all $c \in \pi_0(\mathcal{C}_G)$ we introduce a formal variable $T_c$ and use the multi-index notation $\underline{T}^{\underline{n}} = \prod_{c} T_c^{n_c}$.
    
    Consider the generating series
    \[
    F(\underline{T}; \beta) := \sum_{\underline{n}} \hat{\tau}(\beta ; \underline{n}) \underline{T}^{\underline{n}}.
    \]
    Note that if $\varphi = (\varphi_P)_P \in BG[\Adele_{K, \infty}]_{\underline{n}}$ then by definition we have that $T(\varphi) := \prod_{P} T(\varphi_P) = \underline{T}^{\underline{n}}$. We can thus rewrite the generating series $F(\underline{T}; \beta)$ as the integral
    \[
    \sum_{\underline{n}} \int_{BG[\Adele_{K, \infty}]_{\underline{n}}} e^{2 i \pi \langle \beta , \varphi\rangle_{\mathrm{BM}}} T(\varphi) d\tau_{\A^{1}_{\F_q}}(\varphi) = \int_{BG[\Adele_{K, \infty}]} e^{2 i \pi \langle \beta , \varphi\rangle_{\mathrm{BM}}}T(\varphi)d\tau_{\A^{1}_{\F_q}}(\varphi).
    \]

    This integral is computed in \Cref{lem:Brauer_twisted_Igusa_integral}.
    In particular, we see that the coefficient of $\underline{T}^{\underline{n}}$ is zero unless $\underline{n}$ is supported on $\mathcal{C}_{\beta}$ as the only $T_c$ appearing in the Euler factors are those with $c \in \mathcal{C}_{\beta}$.
    
   After the substitution $\hat{T}_c := e(\partial_c(\beta))T_c$ the Euler product in \Cref{lem:Brauer_twisted_Igusa_integral} agrees with the Euler product in Lemma \ref{lemma:number_points_Configuration_space} for the generating series of configuration spaces on $\mathcal{C}_{\beta}$. It follows from the uniqueness of generating series that the $\underline{T}^{\underline{n}}$ coefficient of $F(\underline{T} ; \beta)$ is
   \[
   \hat{\tau}(\beta ; \underline{n}) \underline{T}^{\underline{n}} = \# \Conf_{\mathcal{C}_{\beta}, \underline{n}}(\F_q) \hat{\underline{T}}^{\underline{n}} = \# \Conf_{\mathcal{C}_{\beta}, \underline{n}}(\F_q)\prod_{c \in \pi_0(\mathcal{C}_{\beta})} e(\partial_c(\beta)) \underline{T}^{\underline{n}}
   \]
   which proves the proposition.
\end{proof}
\begin{corollary}\label{cor:brauer_manin_volume_same}
    For all $\underline{n} \in \Hom(\mathcal{C}_G^*, \N)$ with support $\mathcal{C}$ and all subsets $\Omega_{\infty} \subset BG[\F_q((t^{-1}))]$ we have
    \[
    |\Br BG_{\F_q}| \tau_{BG,  \F_q}( BG[\Adele_{\F_{q}(t)}
    ]_{\underline{n}, \Omega_{\infty}}^{\Br}) = |\Br_{\bar{\mathcal{C}}} BG_{\F_q}| \tau_{BG,  \F_q}( BG[\Adele_{\F_{q}(t)}
    ]_{\underline{n}, \Omega_{\infty}}^{\Br_{\bar{\mathcal{C}}} BG_{\F_q}})
    \]
\end{corollary}
\begin{proof}
    By applying \Cref{lem:character_orthogonality} to $B = \Br BG_{\F_q}$ and $B =  \Br_{\bar{\mathcal{C}}} BG_{\F_q}$ it suffices to show that for $\beta \not \in \Br_{\bar{\mathcal{C}}} BG_{\F_q}$ we have $\hat{\tau}(\beta ; BG[\Adele_{\F_{q}(t)}
    ]_{\underline{n}, \Omega_{\infty}} ) = 0$.

    This is an immediate corollary of \Cref{prop:hurwitz_spaces_Tamagawa_measures}(1) and the definition of $\Br_{\bar{\mathcal{C}}} BG_{\F_q}$.
\end{proof}
\subsection{Multi-height counting} 
We are now ready to prove one of the main theorems of the paper. Let us first define the counting function.
\begin{definition}
    Given $\Omega_{\infty} \subset BG[\F_q((t^{-1}))]$ and $\underline{n} \in \Hom_{\F_q}(\mathcal{C}_G^*, \N)$ we define $N(G,\underline{n}, \Omega_{\infty})$ as the number of isomorphism classes of geometrically connected $G$-covers $\varphi: X \to \Proj^1_{\F_q}$ such that the corresponding local extension at $\infty$ is contained in $\Omega_{\infty}$ and such that for all $c \in  \mathcal{C}_G^*(\bar{\F}_q)$ we have 
    \[
    \sum_{\substack{P \in \A^1_{\F_q} \\ \rho_{G, P}(\varphi) = c}} [\F_q(P): \F_q(c)] = n_c.
    \]
    Where $\rho_{G, P}$ denotes the ramification type as in \Cref{con:inertia_type}. 
    
    Informally, this is the number of points $P \in \A^1_{\F_q}$, counted with multiplicity, whose ramification type is $c$.
\end{definition}
\begin{lemma}\label{lem:decomposition_counting_function}
    We have
    \[
    N(G,\underline{n}, \Omega_{\infty}) = |Z(G(\F_q))| \sum_{(\varphi, \gamma) \in \Omega_{\infty}} \frac{1}{|\mathrm{Aut}_{BG}(\varphi, \gamma)(\F_q)|}
 \# \CHur_{\underline{n}}^{G, (\varphi, \gamma)}(\F_q).
    \]
\end{lemma}
\begin{proof}
The set counted by $N(G,\underline{n}, \Omega_{\infty})$ is by definition the set of isomorphism classes of the groupoid $\HurStack^G_{\underline{n}}(\F_q)$ such that the corresponding curve is geometrically connected and whose image in $[G(-1)/G](\F_q) \cong BG(\F_q((t^{-1})))$ is contained in $\Omega_{\infty}$.

The lemma then follows from the definition of $\CHur_{\underline{n}}^{G, (\varphi, \gamma)}(\F_q)$ and comparing the groupoid counts, where we use that any element $\HurStack^G_{\underline{n}}(\F_q) $ whose corresponding curve is geometrically connected has automorphism group scheme $Z(G)$ (which is just by the definition of a morphism of a $G$-cover, c.f.~\cite[Lem.~2.9]{Loughran2025mallesconjecturebrauergroups}).
\end{proof}

\begin{theorem}\label{thm:main_theorem_multiheight}
    Assume that $q$ is sufficiently large in terms of $|G(\bar{\F}_q)|$. Let $\mathcal{C}_{\mathrm{st}} \subset \mathcal{C}_G^*$ be a subset of conjugacy classes which generates $G/Z(G)$. Let $\Omega_{\infty} \subset BG[\F_q((t^{-1})]$ be a non-empty subset. There exists $\delta > 0$ such that we for all in $\underline{n} \in \Hom(\mathcal{C}_G^*, \N)$ whose support generates $G$ that
    \[
    N(G, \underline{n} ,\Omega_{\infty}) = \frac{|Z(G(\F_q))| \cdot |\Br BG_{\F_q}| }{|G^{\mathrm{ab}}(-1)(\F_q)|} \tau_{BG,  \F_q}( BG[\Adele_{\F_{q}(t)}
    ]_{\underline{n}, \Omega_{\infty}}^{\Br}) + O(q^{|\underline{n}|- \delta \min_{c \in \mathcal{C}_{\mathrm{st}}(\bar{\F}_q)}(n_c)}).
    \]
\end{theorem}
\begin{proof}
    We first do the case that the support of $\underline{n}$ does not contain $\mathcal{C}_{\mathrm{st}}$. It then suffices to show that the error term dominates the right-hand side by \Cref{lem:application_landesman_levy} and \Cref{lem:decomposition_counting_function}.
    
    The error term dominates as 
    \[
    \tau_{BG,  \F_q}( BG[\Adele_{\F_{q}(t)}
    ]_{\underline{n}, \Omega_{\infty}}^{\Br}) \leq \tau_{BG,  \F_q}( BG[\Adele_{\F_{q}(t)}
    ]_{\underline{n}, \Omega_{\infty}}) = \tau_{\infty}(\Omega_{\infty})\Conf_{\mathcal{C}_G^*, \underline{n}}(\F_q) \leq \tau_{\infty}(\Omega_{\infty})q^{|n|}
    \]from \Cref{prop:hurwitz_spaces_Tamagawa_measures} and \Cref{cor:points_configuration_space}.

    Assume now that the support $\mathcal{C}$ of $\underline{n}$ contains $\mathcal{C}_{\mathrm{st}}$. By combining \Cref{lem:Hurwitz_spaces}, \Cref{lem:decomposition_counting_function}, \eqref{eq:formula_local_Brauer_transform} and \Cref{prop:hurwitz_spaces_Tamagawa_measures} we find that
    \[
    N(G, \underline{n} ,\Omega_{\infty})  = \frac{|Z(G(\F_q))|}{|G^{\mathrm{ab}}(-1)(\F_q)|} \sum_{\beta \in \Br_{\bar{\mathcal{C}}} BG_{\bar{\F}_q}} \hat{\tau}(\beta ;BG[\Adele_{\F_{q}(t)}
    ]_{\underline{n}, \Omega_{\infty}}) + O(q^{|\underline{n}|- \delta \min_{c \in \mathcal{C}_{\mathrm{st}}(\bar{\F}_q)}(n_c)}).
    \]
    The sum is equal to $|\Br BG_{\F_q}|  \tau_{BG,  \F_q}( BG[\Adele_{\F_{q}(t)}
    ]_{\underline{n}, \Omega_{\infty}}^{\Br})$ by \Cref{lem:character_orthogonality} and \Cref{cor:brauer_manin_volume_same}.
\end{proof}
\begin{remark} \hfill
\begin{enumerate}
    \item This theorem can be understood as a multi-height variant of Malle's conjecture over function fields, as introduced for number fields in \cite{gundlach2022mallesconjecturemultipleinvariants}.

    \item We will apply \Cref{thm:main_theorem_multiheight} to count extensions via the usual notion of a height function. But it seems plausible that one can also apply it to count $G$-extensions of $\F_q(t)$ for more exotic orderings.

    \item One can consider the adelic set $ BG[\Adele_{\F_{q}(t)}]_{\underline{n}, \Omega_{\infty}}$ as an \emph{adelic height ball} for the stack $BG$ in the sense of \cite[\S1.3]{Chambert-Loir2010Igusa}. Such adelic height balls were introduced and studied to provide a heuristic justification for Manin's conjecture and Peyre's constant. We will study the asymptotic properties of such height balls with the same techniques as in that paper.
    \end{enumerate}
\end{remark}
\subsection{Analytic number theory}
The goal of this section is to deduce \Cref{thm:main_theorem} from \Cref{thm:main_theorem_multiheight}. Let $f \in \Hom_{\F_q}(\mathcal{C}_G^*, \Z_{> 0})$. This is a weight function in the sense of \cite[Def.~3.8]{Loughran2025mallesconjecturebrauergroups}. 
We define $a(f) = \min_{c \in \mathcal{C}_G^*} (f(c))^{-1}$, $b(\F_q, f) = |\pi_0(\{ c \in \mathcal{C}_G^* : f(c) = a(f)^{-1}\})|$ as in \cite[Lem.~3.26]{Loughran2025mallesconjecturebrauergroups}. For $\underline{n} \in \Hom(\mathcal{C}_G^*, \N)$ we write $f(\underline{n}) = \sum_{c \in \mathcal{C}_G^*(\bar{\F}_q)} f(n_c)$.

We will first deal with the error term.
\begin{lemma}\label{lem:error_term}
    Let $\delta > 0$. Then for all $d \in \N$ we have
    \[
    \sum_{\substack{\underline{n} \in \Hom_{\F_q}(\mathcal{C}_G^*, \N) \\ f(\underline{n}) = d}} q^{|\underline{n}|  - \min(\underline{n}_{\min}) \delta} = O_f\left(d^{b(\F_q, f) - 2}\frac{q^{a(f)d}}{1 - q^{-\delta}}\right).
    \]
\end{lemma}
\begin{proof}
    Let $\mathfrak{c} \in \pi_0(\mathcal{C}_f)$ and let $a(f)^{-1} < A \in \Z$ be such that $f(c) \geq A$ for all $c \in \mathcal{C}_G^* \setminus \mathcal{C}_f$. The contribution to this sum with $n_\mathfrak{c} = \min_{c \in \pi_0(\mathcal{C}_f)} n_c$ is then bounded by
    \[
    \sum_{\substack{\underline{m} \in \Hom_{\F_q}(\mathcal{C}_{G}^* \setminus \mathcal{C}_f, \N) \\ f(\underline{m}) \leq d }} q^{|\underline{m}|}
    \sum_{k = 0}^{\frac{d - f(\underline{m})}{[\F_q(\mathfrak{c}): \F_q]}} q^{[\F_q(\mathfrak{c}): \F_q]k - \delta k}
     \sum_{\substack{\underline{n} \in \Hom_{\F_q}(\mathcal{C}_G^* \setminus \mathfrak{c}, \N) \\\ f(\underline{n}) = d - f(\underline{m}) -  k[\F_q(\mathfrak{c}): \F_q]f(\mathfrak{c})}} q^{|\underline{n}|}.
    \]

    For the inner sum $f(\underline{n}) = a(f)^{-1} |\underline{n}|$. The inner sum is thus constant so standard estimates for lattice points in a simplex show that it is
        \[
    \ll  d^{b(\F_q, f) - 2}q^{a(f)(d - f(\underline{m}) -[\F_q(\mathfrak{c}): \F_q]k)} \ll d^{b(\F_q, f) - 2}q^{a(f)d} q^{ - a(f) A|\underline{m}|} q^{[\F_q(\mathfrak{c}): \F_q]k}.
    \]

    The outer sums are geometric series which converge when extended to $\infty$. The total contribution is thus $\ll d^{b(\F_q, f) - 2} q^{a(f) d}(1 - q^{- \delta})^{-1}$. The lemma follows by summing over $\mathfrak{c}$.
\end{proof}

We will now consider the main term. Let $\Omega_{\infty} \subset BG[\F_q((t^{-1}))]$. For each $d$ we will consider the following approximate main term
\[
N_{f, \text{main}}(G, q^d, \Omega_{\infty}) := \sum_{\substack{\underline{n} \in \Hom(\mathcal{C}_G^*, \N) \\ f(\underline{n}) = d}} |\Br BG_{\F_q}| \tau_{BG,  \F_q}( BG[\A_{\F_{q}(t)}]_{\underline{n}, \Omega_{\infty}}^{\Br}).
\]

Let $F(X) := \sum_{d = 0}^{\infty} N_{f, \text{main}}(G, q^d, \Omega_{\infty}) X^d$ be the generating series, this can be explicitly computed as a sum of Euler products.
\begin{lemma}\label{lem:decompose_main_term_Brauer}
    We have $F(X) = \sum_{\beta \in \Br BG_{\F_q}} \hat{\tau}_{\infty}(\beta ; \Omega_{\infty})  F_{\beta}(X)$ with
    \[
    F_{\beta}(X) := \prod_{P \in \A^1_{\F_q}}  \left( 1 + \sum_{c \in \mathcal{C}_{\beta}(\F_q(P))} e(\mathrm{cor}_{\F_q(P)/\F_q}(\partial_c(\beta)))X^{(\deg P) f(c)} \right).
    \]
\end{lemma}
\begin{proof}
    By \Cref{lem:character_orthogonality} we have 
    \[
     |\Br BG_{\F_q}|\tau_{BG,  \F_q}( BG[\Adele_{\F_{q}(t)}]_{\underline{n}, \Omega_{\infty}}^{\Br}] = \sum_{\beta \in \Br BG} \hat{\tau}(\beta ; BG[\Adele_{\F_{q}(t)}]_{\underline{n}, \Omega_{\infty}}).  \]

     We thus find that $F(X)$ is a sum over $\beta$ with terms
     \[
     \hat{\tau}_{\infty}(\beta ; \Omega_{\infty})\sum_{\underline{n} \in \Hom(\mathcal{C}_G^*, \N)}  X^{f(\underline{n})} \int_{BG[\Adele_{K, \infty}]_{\underline{n}}} e^{2 i \pi \langle \varphi, \beta \rangle_{\mathrm{BM}}} d\tau_{BG, \Adele_{K, \infty}} .
     \]

     As measures are countably additive this is equal to the left-hand side of \eqref{eq:formula_global_Brauer_transform} with $T_c = X^{f(c)}$. We are then done by Lemma \ref{lem:Brauer_twisted_Igusa_integral}.
\end{proof}
We recall that the zeta function $Z_{\A^1_{\F_q}}(X) := \prod_{P \in \A^1_{\F_q}}(1 - X^{- \deg P})^{-1}$ of $\A^1_{\F_q}$ converges absolutely for $|X| < q^{-1}$ and is equal to $(1 - q X)^{-1}$. We need the following analytic fact.
\begin{lemma}\label{lem:conditional_convergence}
    The Euler product $ \prod_{P \in \A^1_{\F_q}}(1 - X^{- \deg P})^{-1}$ converges conditionally to $Z_{\A^1_{\F_q}}(X)$ for all $|X| = q^{-1}$ with $X \neq q^{-1}$
\end{lemma}
\begin{proof}
    The analogous statement for the Riemann zeta function is well-known \cite[\S3.15]{Titchmarsh1986Zeta} and the same argument works (easier) over function fields. To spell it out, taking logarithms it suffices to show that the following sum converges conditionally
    \[
        \sum_{P \in \A^1_{\F_q}} \sum_{m = 1}^{\infty} \frac{X^{m\deg P}} {m} =  \sum_{n} \sum_{m = 1}^{\infty} \frac{X^{nm}}{m}\#\{P \in \A^1_{\F_q}: \deg P = n\} 
    \]
    This is an immediate consequence of the fact that $\#\{P \in    \A^1_{\F_q}: \deg P = n\} = \frac{q^{n}}{n} + O(q^{\frac{n}{2}})$ (the prime number theorem over function fields), partial summation, and the formula for a geometric series.
\end{proof}
Recall that the map $e: \HH^1(\F_q, \Q/\Z) \to \C^{\times}$ is injective.
\begin{lemma}\label{lem:analytic_properties_F_beta}
Let $\beta \in \Br BG_{\F_q}$. The power series $F_{\beta}(X)$ converges absolutely for $|X| < q^{-a(f)}$. Moreover, there exists $\delta > 0$ such that it has a meromorphic continuation to the disc $|X| < q^{- a(f) + \delta}$. All poles of $F_{\beta}(X)$ on the circle $|X| = q^{-a(f)}$ have order at most $b(f, \F_q)$ and the only poles on this circle with order equal to $b(f, \F_q)$ are of the form $e(-\alpha)q^{-a(f)}$ with $\alpha \in \HH^1(\F_q, \Q/\Z)$ such that $\beta \in \Br_{\mathcal{C}, a(f)^{-1} \cdot \alpha} BG_{\F_q}$. 

If $\beta \in \Br_{\mathcal{C}, \alpha} BG_{\F_q}$ then $\lim_{X \to e(-\alpha) q^{-a(f)}} (1 - e(\alpha)q^{a(f)}X)^{b(f, \F_q)} F_{\beta}(X)$ is equal to the following conditionally convergent Euler product
\[
a(f)^{b(f, \F_q)}\prod_{P \in \mathbf{A}^1_{\F_q}}(1 - q^{-\deg P})^{b(\F_q, f)}\left( 1 + \sum_{c \in \mathcal{C}_{\beta}(\F_q(P))} e(\mathrm{cor}_{\F_q(P)/\F_q}(\partial_c(\beta))) (e(\alpha)q^{a(f)})^{-(\deg P) f(c)} \right).
\]
\end{lemma}
\begin{proof}
    Absolute convergence follows by comparing the Euler factors with $|Z_{\A^{1}_{\F_q}}(X)|^{|\mathcal{C}_{\beta}(\bar{\F}_q)|}$.

    Let us now analyse the Euler factors. Choose $\frac{a(f)}{4} > \delta > 0$ such that the terms in the Euler factor coming from $c \not \in \mathcal{C}_f$ contribute at most $O(q^{- 1 - \delta})$ for $|X| < q^{-a(f) + \delta}$. They are then irrelevant for absolute convergence by comparison with the Euler factors of $Z_{\A^{1}_{\F_q}}(X)$. 
    
    For $C \in \pi_0(\mathcal{C}_f)$ let $\chi_C \in \HH^1(\F_q, [\F_q(c): \F_q]^{-1} \Z/\Z) \cong \Z/[\F_q(C): \F_q] \Z$ be a generator. We can write the number of points on $C(\F_q(P))$ using character orthogonality as $\sum_{k = 1}^{[\F_q(C): \F_q]} e((\deg P) \cdot k \cdot \chi_C)$. The restriction map $\Q/\Z \cong \HH^1(\F_q, \Q/\Z) \to \HH^1(\F_q(C), \Q/\Z) \cong \Q/\Z$ is surjective by a direct computation. It followss that there exists a (non)-unique element $d_C \in \HH^1(\F_q, \Q/\Z)$ whose restriction to $\F_q(C)$ is $\partial_C(\beta)$
    
    The remaining terms in the Euler factor thus contribute
    \begin{align*}
        &1 + \sum_{c \in \mathcal{C}_f(\F_q(P))}e(\mathrm{cor}_{\F_q(P)/\F_q} (\partial_c(\beta))) X^{-(\deg P) a(f)^{-1}} = \\
        &1 + \sum_{C \in \pi_0(\mathcal{C}_f)}\sum_{k = 1}^{[\F_q(C): \F_q]} e((\deg P) \cdot k \cdot \chi_C) e( (\deg P) d_C) X^{-(\deg P) a(f)^{-1}} = \\
        &  \prod_{C \in \pi_0(\mathcal{C}_f)} \prod_{k = 1}^{[\F_q(c): \F_q]}(1 - Y_{C, k}^{-\deg P})^{-1} + O(X^{2 (\deg P) a(f)^{-1}}).
    \end{align*}
    Where $Y_{C, k} := e(- k \cdot \chi_C) e(- d_C) X^{ a(f)^{-1}}$ and we used that if $C(\F_q(P)) \neq \emptyset$ then
    \[
    \mathrm{cor}_{\F_q(P)/\F_q} (\partial_C(\beta)) =  \mathrm{cor}_{\F_q(P)/\F_q}(d_C|_{\F_q(P)})= (\deg P)d_C. 
    \]

    In the region $|X| < q^{-a(f) + \delta}$ the termS $O(X^{2 (\deg P) a(f)^{-1}})$ is irrelevant for absolute convergence by comparison with $Z_{\A^1_{\F_q}}(X)$. We have thus shown that
    \begin{equation}\label{eq:decomposition_F_beta}
        F_{\beta}(X) = G(X) \prod_{C \in \pi_0(\mathcal{C}_f)} \prod_{k = 1}^{[\F_q(c): \F_q]}  Z_{\mathbf{A}^1_{\F_{q}}}(Y_{C, k})
    \end{equation}
    where $G(X)$ is an Euler product that is absolutely convergent in the disc $|X| < q^{-a(f) + \delta}$.

    To determine the poles we thus need to determine when $Y_{C, k} = q^{-1}$, which by definition and the fact that $a(f)^{-1} \in \N$ happens when $X = e(-\alpha) q^{-a(f)}$ with
    \begin{equation}\label{eq:existence_of_k}
        a(f)^{-1} \cdot \alpha = k \chi_C + d_C.
    \end{equation}

    The group $\HH^1(\F_q, [\F_q(c): \F_q]^{-1} \Z/\Z)$ is exactly the kernel of the restriction $\HH^1(\F_q, \Q/\Z) \to \HH^1(\F_q(c), \Q/\Z)$. It follows that $k$ exists as in \eqref{eq:existence_of_k} if and only if $(a(f)^{-1} \cdot \alpha)|_{\F_q(c)} = \partial_C(\beta)$ and if it exists, it is unique.

    It follows that $F(X)$ has a pole at $e(-\alpha) q^{-a(f)}$ of order at most the number of $C \in \pi_0(\mathcal{C}_f)$ for which $\alpha|_{\F_q(c)} = \partial_C(\beta)$. This is at most $b(f, \F_q)$ with equality if and only if $\beta \in \Br_{\mathcal{C}_f, a(f)^{-1} \alpha} BG_{\F_q}$ by definition, see \Cref{def:Br_C_ell}.

    It remains to compute the limit in the case that $\beta \in \Br_{\mathcal{C}_f, a(f)^{-1} \alpha} BG_{\F_q}$. For each $C \in \pi_0(\mathcal{C}_f)$ let $k_C$ be the $k$ satisfying \eqref{eq:existence_of_k}. The formula in the limit is then by \eqref{eq:decomposition_F_beta} equal to
    \begin{equation}\label{eq:residues_F_beta}
       G(X) \prod_{C \in \pi_0(\mathcal{C}_f(\F_q(P))} a(f) \prod_{\substack{k = 1 \\ k \neq k_C}}^{[\F_q(c): \F_q]}  Z_{\mathbf{A}^1_{\F_{q}}}(Y_{C, k}). 
    \end{equation}
    The Euler product of $G(X)$ is absolutely convergent and the Euler products of the zeta factors are conditionally convergent and equal to the analytic continuation for $X = e(-\alpha) q^{-a(f)}$ by Lemma \ref{lem:conditional_convergence}. The Euler factors of \eqref{eq:residues_F_beta} are exactly those in the statement of the lemma.
\end{proof}
\subsubsection{Heights}\label{sec:heights}
To state our main theorem we recall the notion of a height from \cite[\S8.1]{Loughran2025mallesconjecturebrauergroups}.

For $P \in \A^1_{\F_q}$ the local height coming from $f$ is the function $H_{P, f}: BG[K_P] \to q^{\N}$ defined by the formula $H_{P, f}(\varphi_P) := q^{[\F_q(P): \F_q] f(\rho_{G, P}(\varphi))}$. For $P = \infty$ we let $h_{\infty}: BG[\F_q((t^{-1}))] \to \N$ be any function, which one may think of as a logarithmic local height, the corresponding exponential local height is defined as $H_{\infty} := q^{h_{\infty}}$.

We then define the global height $H: BG[\F_q(t)] \to q^{\N}$ as the product of the local heights $H := H_{\infty}\prod_{P \in \A^1_{\F_q}} H_{P, f}$.

\begin{definition}
    Given $\Omega_{\infty} \subset BG[\F_q((t^{-1}))]$ we define $N_H(G, q^{d}, \Omega_{\infty})$ as the number of isomorphism classes of geometrically connected $G$-covers of $\varphi: X \to \Proj^{1}_{\F_q}$ such that $H(\varphi_{\F_q(t)}) = q^{d}$ and such that the corresponding local extension at $\infty$ is contained in $\Omega_{\infty}$.
\end{definition}
\subsubsection{The main theorem}
\begin{theorem}\label{thm:main_theorem_height_etale}
    Let $\F_q$ be a finite field, $G$ a tame finite \'etale group scheme over $\F_q$, $f \in \Hom(\mathcal{C}_G^*, \Z_{> 0})$ a function such that $\mathcal{C}_f$ generates $G$ and $\Omega_{\infty} \subset BG[\F_q((t^{-1}))]$.
    
    If $q$ is sufficiently large in terms of $|G(\bar{\F}_q)|$ then we have
    \begin{equation*}
      N_H(G, q^{d}, \Omega_{\infty}) = c_H( G, q^{d}, \Omega_{\infty}) d^{b(f, \F_q) - 1}q^{a(f) d} + O(d^{b(f, \F_q) - 2}q^{a(f) d}).  
    \end{equation*}
    With leading constant $c_H( G, q^{d}, \Omega_{\infty})$ given by
    \begin{equation*}
    \frac{a(f)^{b(f, \F_q)} |Z(G)(\F_q)|}{|G^{\mathrm{ab}}(-1)(\F_q)| (b(f, \F_q) - 1)!}\sum_{\alpha \in \HH^1(\F_q, \Q/\Z)} e(d \cdot \alpha) \sum_{\beta \in \Br_{\mathcal{C}_f, a(f)^{-1} \cdot \alpha} BG_{\F_q}} \hat{\tau}_{H, \alpha}(\beta)
    \end{equation*}
    Here $\hat{\tau}_{H, \alpha}(\beta)$ is the conditionally convergent Euler product
    \begin{equation}\label{eq:Euler_products_main_theorem}
        \hat{\tau}_{h_{\infty}, \alpha; \Omega_{\infty}}(\beta) \prod_{P \in \A^1_{\F_q}}(1 - q^{-\deg P})^{b(f, \F_q)}\hat{\tau}_{f, \alpha, P}(\beta).
    \end{equation}
    Where $\hat{\tau}_{f, \alpha, P}(\beta)$ is equal to
    \[ 
     1 + \sum_{\substack{c \in \mathcal{C}_\beta(\F_q(P)))}} e( \mathrm{cor}_{\F_q(P)/\F_q}(\partial_c(\beta)))e(-\deg P \cdot f(c) \cdot \alpha)q^{-(\deg P) a(f) f(c)}.
    \]
    And the local factor at $\infty$ is given by the integral
    \[
    \hat{\tau}_{h_{\infty}, \alpha, \infty}(\beta ; \Omega_{\infty}) = \int_{\Omega_{\infty}} e^{2 i \pi\langle \beta, (\varphi_{\infty}, \gamma) \rangle_{\mathrm{BM}}} e(-h_{\infty}(\varphi_{\infty}) \cdot \alpha)q^{- a(f) h_{\infty}(\varphi_{\infty})} d\tau_{\infty}(\varphi_{\infty}).
    \]

    Moreover, $\Br_{\mathcal{C}_{f}, a(f)^{-1} \cdot \alpha} BG_{\F_q} = 0 $ unless $\alpha \in \HH^1(\F_q, a(f)|G|^{-2}\Z/\Z)$ so $c_H( G, q^{d}, \Omega_{\infty})$ is $ a(f)^{-1}|G|^2$-periodic in $d$.
\end{theorem}
\begin{proof}
    We may assume that $\Omega_{\infty}$ consists of a single element $\varphi_{\infty}$. If $h_{\infty}(\varphi_{\infty}) = k$ then let $H'$ be the height function with $h_{\infty} = 1$. By definition, we have $N(G, H, q^{d}, \Omega_{\infty}) = N(G, H', q^{d - k}, \Omega_{\infty})$. One checks that the theorem is compatible with changing $H$ by $H'$ so we may assume that $h_{\infty}$ is trivial. In this case $\hat{\tau}_{h_{\infty}, \alpha, \infty}(\beta ; \Omega_{\infty}) = \hat{\tau}_{\infty}(\beta ;\Omega_{\infty})$.

    By the definition of the counting functions and \Cref{thm:main_theorem_multiheight} we have
    \begin{align*}
    &N_H(G, q^{d}, \Omega_{\infty}) = \sum_{\substack{\underline{n} \in \Hom_{\F_q}(\mathcal{C}_G^*, \N) \\ f(\underline{n}) = d}} N(G, \underline{n}, \Omega_{\infty}) = \\ &\frac{|Z(G)(\F_q)| \cdot |\Br BG_{\F_q}|}{|G^{\mathrm{ab}}(-1)(\F_q)| }\sum_{\substack{\underline{n} \in \Hom_{\F_q}(\mathcal{C}_G^*, \N) \\ f(\underline{n}) = d}} \tau_{BG,  \F_q}( BG[\A_{\F_{q}(t)}]_{\underline{n}, \Omega_{\infty}}^{\Br}) + O(q^{|\underline{n}|- \delta \min_{c \in \mathcal{C}_f}(n_c)}).
    \end{align*}
    The sum of the error terms has been dealt with in \Cref{lem:error_term}. It thus suffices to estimate $N_{\mathrm{main}}(G, q^d, \Omega_{\infty})$. 
    
    The periodicity is because \Cref{lem:exact_sequence_brauer_group} implies that $\partial_c(\beta)$ is $|G|^2$-torsion for all $\beta \in \Br_{\bar{\mathcal{C}_f}} BG_{\F_q}$ and $c \in \mathcal{C}_f$.
    
    Using \Cref{lem:decompose_main_term_Brauer} and \Cref{lem:analytic_properties_F_beta} the theorem follows from the following Tauberian theorem.
\end{proof}
\begin{lemma}\label{lem:Tauberian_theorem}
    Let $F(X) = \sum_{n}a_n X^n$ be a power series with positive radius of convergence. Assume that there exist $r, \delta > 0$ such that $F(X)$ has a meromorphic continuation to the disc $|X| \leq q^{r + \delta}$ and such that all of its poles lie on the circle with radius $q^{r}$. Let $b$ be the maximal order of any pole and let $z_1 q^{r}, \cdots, z_k q^{r}$ be the poles of maximal order, with $|z_i| = 1$. Put $c_i = \lim_{X \to z_iq^{r}}(1 - z_i^{-1} q^{-r} X)^b F(X)$. Then 
    \[
    a_n = \frac{n^{b - 1} q^{- r n}}{(b - 1)!}\sum_{i = 1}^k c_i(z_i)^{-n} + O(n^{b - 2} q^{- r n}).
    \]
\end{lemma}
\begin{proof}
    This is just a convenient rephrasing of (the main term of) \cite[Thm.~5.2.1]{Wilf2006generating}.
\end{proof}
\begin{remark}\hfill
\begin{enumerate}
    \item The Euler product in \eqref{eq:Euler_products_main_theorem} can also be written as an absolutely convergent Euler argument via a similar argument as in \cite[Lem.~8.19]{Loughran2025mallesconjecturebrauergroups}.
    \item If $b(G, \F_q) = 1$ then by being slightly more careful one can get a power-saving error term.
\end{enumerate}
\end{remark}
\begin{remark}\label{rem:different_powers_a(f)}
    A careful comparison using \cite[Lem.~8.21]{Loughran2025mallesconjecturebrauergroups} shows that the contribution from $\alpha = 0$ in \Cref{thm:main_theorem_height_etale} is analogous to the leading constant in \cite[Conj.~9.1]{Loughran2025mallesconjecturebrauergroups}, except that we have a factor $a(f)^{b(f, \F_q)}$ instead of $a(f)^{b(f, \F_q(t)) - 1}$.
    
    Our interpretation of this is that the power of $a(f)$ is there to make sure that the leading constant is compatible when changing the height $H$ by $H^k$ for some $k \in \N$. In the number field case the required power is $a(f)^{b(f, \F_q(t)) - 1}$, c.f.~\cite{SawinLeading}, but in the function field case the correct power is $a(f)^{b(f, \F_q(t))}$ because when replacing $H$ by $H^k$ one also needs to increase the period of $c_H(G, q^d, \Omega_{\infty})$.
\end{remark}
\begin{proof}[Proof of \Cref{thm:main_theorem} when $\mathcal{C}_f$ generates $G$]
    This is just \Cref{thm:main_theorem_height_etale}, with $G$ a constant group scheme, $\Omega_{\infty} = BG[\F_q((t^{-1}))]$ and $h_{\infty} = f$. Note that the $(1 - q^{-1})^{-b(f, \F_q)}$factor cancels the $(1 - q^{-1})^{b(f, \F_q)}$factor at $\infty$ in the Euler product.
\end{proof}
\begin{proof}[Proof of \Cref{thm:main_theorem_radical_disc}]
    Apply \Cref{thm:main_theorem} with the constant function $f = 1$. The leading constant is $|G|^2$-periodic as $a(f) = 1$. So by averaging we remove all Euler products except those coming from $\alpha = 0$ and thus, $\beta \in \Br_{\mathcal{C}_f, 0} BG_{\F_q} = \Br_{\mathrm{unr}} BG_{\F_q}$, where the equality is by definition.

    We are done, as for $\beta \in \Br_{\mathrm{unr}} BG_{\F_q}$ we have $\partial_c(\beta) = 0$ for all $c \in \mathcal{C}_G^*$ by definition.
\end{proof}
\subsubsection{Unbalanced heights}\label{sec:non_balanced height}
Consider now the case that $\mathcal{C}_f$ generates a normal subgroup $M \subsetneq G$ such that $M$ generates $G/Z(G)$. The main extra difficulty is that $\CHur^G_{\underline{n}}$ is non-empty if and only if the support of $\underline{n}$ generates $G$. As in \cite[\S2]{Wright1989Abelian} we deal with this via M\"obius inversion.

Let $\mu$ be the M\"obius function of the lattice of abelian groups. Its defining property is that for all abelian groups $A$ we have
\begin{equation}\label{eq:Mobius_inversion}
    \sum_{B \subseteq A} \mu(B) = \begin{cases}
        1 \text{ if } A = 0 \\
        0 \text{ otherwise}.    \end{cases}
\end{equation}
Its precise values are computed in \cite{Delsarte1948Fonctions}, but will not be relevant to us.

We will need the following group-theoretic properties.
\begin{lemma} The quotient $G/M$ is abelian. Any subgroup $M \subseteq L \subseteq G$ is thus normal. Moreover, a conjugacy class of $L$ remains a conjugacy class in $G$.
\end{lemma}
\begin{proof}
    The quotient $G/M$ is abelian as $Z(G)$ surjects onto it. The claim about conjugacy classes is because we an isomorphism of inner automorphism groups $L/Z(L) = L/(L \cap Z(G)) \cong G/Z(G)$.
\end{proof}
The following lemma thus makes sense.
\begin{lemma}\label{lem:mobius_inversion_application}
    For all $d$ we have 
    \begin{align*}
    &\sum_{\substack{\underline{n} \in \Hom_{\F_q}(\mathcal{C}_G^*, \N) \\ f(\underline{n}) = d}} N(G, \underline{n}, \Omega_{\infty}) = \frac{|Z(G)(\F_q)| \cdot |\Br BG_{\F_q}|}{|G^{\mathrm{ab}}(-1)(\F_q)| } \cdot\\ &\sum_{M \subseteq L \subseteq G } \mu(G/L)\sum_{\substack{\underline{n} \in \Hom_{\F_q}(\mathcal{C}_L^*, \N) \\ f(\underline{n}) = d}}  \tau_{BG,  \F_q}( BG[\A_{\F_{q}(t)}]_{\underline{n}, \Omega_{\infty}}^{\Br}) + O(d^{b(f, \F_q) - 2}q^{a(f)d}) 
    \end{align*}
\end{lemma}
\begin{proof}
    Note that if the support of $\underline{n}$ does not generate $G$ then $N(G, \underline{n}, \Omega_{\infty}) = 0$ because it only counts geometrically connected curves. For any $\underline{n}$ whose support contains $\mathcal{C}_f$ we thus have by \Cref{thm:main_theorem_multiheight} and \eqref{eq:Mobius_inversion} that $N(G, \underline{n}, \Omega_{\infty})$ is equal to
    \begin{align*}
         \frac{|Z(G)(\F_q)| \cdot |\Br BG_{\F_q}|}{|G^{\mathrm{ab}}(-1)(\F_q)| } \sum_{\substack{M \subseteq L \subseteq G  \\ \underline{n} \in \Hom_{\F_q}(\mathcal{C}^*_L, \N)}} \mu(G/L) \tau_{BG,  \F_q}( BG[\A_{\F_{q}(t)}]_{\underline{n}, \Omega_{\infty}}^{\Br}) + O(q^{|\underline{n}|- \delta \min_{c \in \mathcal{C}_f}(n_c)}).
    \end{align*}
    Note that this also holds if the support of $\underline{n}$ does not contain $\mathcal{C}_f$ as in this case the error term dominates the main term by the same argument as in the proof of \Cref{thm:main_theorem_multiheight}.

    The lemma then follows by summing over $\underline{n}$ and switching the order of summation. The sum over error terms is dealt with in \Cref{lem:error_term}. 
\end{proof}
It thus remains to estimate for all $M \subseteq L \subseteq G$ the sum
\[
N_{f, L, \text{main}}(G, q^d, \Omega_{\infty}) := \sum_{\substack{\underline{n} \in \Hom_{\F_q}(\mathcal{C}_L^*, \N) \\ f(\underline{n}) = d}}  |\Br BG_{\F_q}| \tau_{BG,  \F_q}( BG[\A_{\F_{q}(t)}]_{\underline{n}, \Omega_{\infty}}^{\Br}).
\]
We will estimate this sum in the exact same way as the case that $\mathcal{C}_f$ generates $G$. We will give the relevant statements, but omit the proofs as they are identical to the case that $\mathcal{C}_f$ generates $G$.

Let $F_L(X) := \sum_{d = 0}^{\infty} N_{f, L, \text{main}}(G, q^d, \Omega_{\infty}) X^d$ be the generating series. The analogue of \Cref{lem:decompose_main_term_Brauer} is
\begin{lemma}\label{lem:decompose_main_term_Brauer_L}
    We have $F_L(X) = \sum_{\beta \in \Br BG_{\F_q}} \hat{\tau}_{\infty}(\beta ; \Omega_{\infty})  F_{L, \beta}(X)$ with
    \[
    F_{L, \beta}(X) := \prod_{P \in \A^1_{\F_q}}  \left( 1 + \sum_{c \in \mathcal{C}_{L, \beta}(\F_q(P))} e(\mathrm{cor}_{\F_q(P)/\F_q}(\partial_c(\beta)))X^{(\deg P) f(c)} \right).
    \]

    Where $\mathcal{C}_{L, \beta} := \mathcal{C}^*_L \cap \mathcal{C}_{\beta}$.
\end{lemma}
The analogue of \Cref{lem:analytic_properties_F_beta} is as follows.
\begin{lemma}\label{lem:analytic_properties_F_L_beta}
Let $\beta \in \Br BG_{\F_q}$. The power series $F_{L, \beta}(X)$ converges absolutely for $|X| < q^{-a(f)}$. Moreover, there exists $\delta > 0$ such that it has a meromorphic continuation to the disc $|X| < q^{- a(f) + \delta}$. All poles of $F_{L, \beta}(X)$ on the circle $|X| = q^{-a(f)}$ have order at most $b(f, \F_q)$ and the only poles on this circle with order equal to $b(f, \F_q)$ are of the form $e(-\alpha)q^{-a(f)}$ with $\alpha \in \HH^1(\F_q, \Q/\Z)$ such that $\beta \in \Br_{\mathcal{C}, a(f)^{-1} \cdot \alpha} BG_{\F_q}$. 

If $\beta \in \Br_{\mathcal{C}, \alpha} BG_{\F_q}$ then $\lim_{X \to e(-\alpha) q^{-a(f)}} (1 - e(\alpha)q^{a(f)}X)^{b(f, \F_q)} F_{L, \beta}(X)$ is equal to the following conditionally convergent Euler product
\[
a(f)^{b(f, \F_q)}\prod_{P \in \mathbf{A}^1_{\F_q}}(1 - q^{-\deg P})^{b(\F_q, f)}\left( 1 + \sum_{c \in \mathcal{C}_{L, \beta}(\F_q(P))} e(\mathrm{cor}_{\F_q(P)/\F_q}(\partial_c(\beta))) (e(\alpha)q^{a(f)})^{-(\deg P) f(c)} \right).
\]
\end{lemma}

Recall the notion of a height $H$ and the definition of $N_H(G, q^{d}, \Omega_{\infty})$ from \S\ref{sec:heights}.
\begin{theorem}\label{thm:main_theorem_height_etale_unbalanced}
    Let $\F_q$ be a finite field, $G$ a tame finite \'etale group scheme over $\F_q$, $f \in \Hom(\mathcal{C}_G^*, \Z_{> 0})$ a function such that $\mathcal{C}_f$ generates $G$ and $\Omega_{\infty} \subset BG[\F_q((t^{-1}))]$.
    
    If $q$ is sufficiently large in terms of $|G(\bar{\F}_q)|$ then we have
    \begin{equation*}
      N_H(G, q^{d}, \Omega_{\infty}) = c_H( G, q^{d}, \Omega_{\infty}) d^{b(f, \F_q) - 1}q^{a(f) d} + O(d^{b(f, \F_q) - 2}q^{a(f) d}).  
    \end{equation*}
    With leading constant $c_H( G, q^{d}, \Omega_{\infty})$ given by
    \begin{equation*}
    \frac{a(f)^{b(f, \F_q)} |Z(G)(\F_q)|}{|G^{\mathrm{ab}}(-1)(\F_q)| (b(f, \F_q) - 1)!}\sum_{M \subseteq L \subseteq G}\mu(G/L)\sum_{\alpha \in \HH^1(\F_q, \Q/\Z)} e(d \cdot \alpha) \sum_{\beta \in \Br_{\mathcal{C}_f, a(f)^{-1} \cdot \alpha} BG_{\F_q}} \hat{\tau}_{H, L,\alpha}(\beta)
    \end{equation*}
    Here $\hat{\tau}_{H, L, \alpha}(\beta)$ is the conditionally convergent Euler product
    \begin{equation*}\label{eq:Euler_products_main_theorem_unbalanced}
        \hat{\tau}_{h_{\infty}, \alpha; \Omega_{\infty}}(\beta) \prod_{P \in \A^1_{\F_q}}(1 - q^{-\deg P})^{b(f, \F_q)}\hat{\tau}_{f, L,  \alpha, P}(\beta).
    \end{equation*}
    Where $\hat{\tau}_{f, L,\alpha, P}(\beta)$ is equal to
    \[ 
     1 + \sum_{\substack{c \in \mathcal{C}_{L, \beta}(\F_q(P)))}} e( \mathrm{cor}_{\F_q(P)/\F_q}(\partial_c(\beta)))e(-\deg P \cdot f(c) \cdot \alpha)q^{-(\deg P) a(f) f(c)}.
    \]
    And the local factor at $\infty$ is given by the integral
    \[
    \hat{\tau}_{h_{\infty}, \alpha, \infty}(\beta ; \Omega_{\infty}) = \int_{\Omega_{\infty}} e^{2 i \pi\langle \beta, (\varphi_{\infty}, \gamma) \rangle_{\mathrm{BM}}} e(-h_{\infty}(\varphi_{\infty}) \cdot \alpha)q^{- a(f) h_{\infty}(\varphi_{\infty})} d\tau_{\infty}(\varphi_{\infty}).
    \]

    Moreover, $\Br_{\mathcal{C}_{f}, a(f)^{-1} \cdot \alpha} BG_{\F_q} = 0 $ unless $\alpha \in \HH^1(\F_q, a(f)|G|^{-2}\Z/\Z)$ so $c_H( G, q^{d}, \Omega_{\infty})$ is $ a(f)^{-1}|G|^2$-periodic in $d$.
\end{theorem}
\begin{proof}
    We reduce to the case that $\Omega_{\infty}$ consists of a single element and that $h_{\infty} = 0$ in the same way as in the proof of \Cref{thm:main_theorem_height_etale}. We then have that 
    \begin{align*}
        N_H(G, q^d, \Omega_{\infty}) = &\sum_{\substack{\underline{n} \in \Hom_{\F_q}(\mathcal{C}_G^*, \N) \\ f(\underline{n}) = d}} N(G, \underline{n}, \Omega_{\infty}) \\ =  &\frac{|Z(G)(\F_q)| }{|G^{\mathrm{ab}}(-1)(\F_q)| }  \sum_{M \subseteq L \subseteq G } \mu(G/L) N_{f, L, \text{main}}(G, q^d, \Omega_{\infty}) + O(d^{b(f, \F_q) - 2}q^{a(f)d}).
    \end{align*}
    Where the second equality is by \Cref{lem:mobius_inversion_application} and the definition of $N_{f, L, \text{main}}(G, q^d, \Omega_{\infty})$.

    The theorem then follows from \Cref{lem:decompose_main_term_Brauer_L}, \Cref{lem:analytic_properties_F_L_beta} and a Tauberian theorem such as \Cref{lem:Tauberian_theorem}.
    \end{proof}
    \begin{proof}[Proof of \Cref{thm:main_theorem} when $\mathcal{C}_f$ does not generate $G$]
    This is just \Cref{thm:main_theorem_height_etale_unbalanced}, with $G$ a constant group scheme, $\Omega_{\infty} = BG[\F_q((t^{-1}))]$ and $h_{\infty} = f$. 
    \end{proof}
    \begin{remark}
        The leading constant in \Cref{thm:main_theorem_height_etale_unbalanced} has some undesired properties. For example, it is not immediately clear how this constant agrees with the leading constant of \cite{Wright1989Abelian} for $G$ abelian and $H$ the discriminant. This is because there are elements $\beta \in \Br_{\mathcal{C}_f, a(f)^{-1} \cdot \alpha}$ such that $\sum_{M \subseteq L \subseteq G} \mu(G/L) \hat{\tau}_{H, L,\alpha}(\beta) = 0$, in this case all transcendental Brauer elements.

        Such elements do not contribute to the leading constant and so we would prefer to remove them, but it is not clear how to interpret all elements $\beta$ such that $\sum_{M \subseteq L \subseteq G} \mu(G/L) \hat{\tau}_{H, L,\alpha}(\beta) = 0$ for general groups $G$.

        Another issue is that it is not clear how to identify this leading constant with the conjectural constant of \cite[Conj.~9.6]{Loughran2025mallesconjecturebrauergroups}. We expect that the the equality should follow from a Poisson summation argument to rewrite the infinite sum over $BG/M[\F_q(t)]$ in \cite[Conj.~9.6]{Loughran2025mallesconjecturebrauergroups} as a finite sum of Euler products. But this is a rather technical computation and lies outside the scope of this paper.
    \end{remark}
    \begin{remark}\label{rem:why_unbalanced_works}
        In this remark we will give an explanation in terms of the framework of \cite[\S9.2]{Loughran2025mallesconjecturebrauergroups} why the case that $\mathcal{C}_f$ generates $G/Z(G)$ is not significantly more difficult than the balanced case.

        The idea of loc.~cit.~is to decompose elements of $BG[\F_q(t)]$ in terms of their image in $BG/M[\F_q(t)]$. The groupoid fiber $X_{\psi}$ of the map $BG \to BG/M$ at $\psi \in BG/M[\F_q(t)]$ is a gerbe over $\F_q(t)$. The map $\coprod_{\psi}X_{\psi}[\F_q(t)] \to BG[\F_q(t)]$ is not a bijection but the size of the fibers can be controlled with \cite[Lem.~2.14]{Loughran2025mallesconjecturebrauergroups}.

        If $\varphi$ is a lift of $\psi$ then $X_{\psi} \cong BM_{\varphi}$ by \cite[Lem.~2.13]{Loughran2025mallesconjecturebrauergroups}. The assumption that $M$ generates $G/Z(G)$ means that $Z(G) \to G/M$ is surjective. This implies that $\psi_{\bar{\F}_q(t)}$ can be lifted to an element $\bar{\varphi} \in (BZ(G))(\bar{\F}_q(t))$. So $(X_{\psi})_{\bar{\F}_q(t)} \cong BM_{\bar{\varphi}} = BM$ as $Z(G)$ acts trivially by conjugation. This does not imply that $X_{\psi}$ is the base change of a gerbe from $\F_q$, for example if $G = \Z/4\Z$ and $M = 2\Z/4\Z$ then $X_{\psi}(\F_q(t))$ can be empty, but every gerbe over $\F_q$ has a rational point by \cite[Lem~8.8]{Loughran2025mallesconjecturebrauergroups}. One can show that this is the only obstruction.
        
        If $X_{\psi}$ does come from a gerbe over $\F_q$ then the restriction of $H$ to $X_{\psi}$ is balanced. Unfortunately, we cannot apply \Cref{thm:main_theorem_height_etale} since the local heights of this restricted height at a point $P \in \A^1_{\F_q}$ will generally be different from $H_f$. But it shows that this case is not too far from the balanced case.
    \end{remark}
\bibliographystyle{alpha} 
\bibliography{references}
\end{document}